\documentclass[leq]{amsart}

\usepackage[centertags]{amsmath}
\usepackage{amsfonts}
\usepackage{amssymb}
\usepackage{amsthm}
\usepackage{newlfont}
\usepackage{amscd}
\usepackage{amsmath,amscd}
\usepackage{MnSymbol}

\usepackage[curve,arrow,cmactex,matrix]{xy}
\usepackage{verbatim}

\usepackage{color}
\usepackage{eucal}

\usepackage{abstract}

\definecolor{darkblue}{rgb}{0,0,0.6}
\usepackage[ocgcolorlinks,colorlinks=true, citecolor=darkblue, filecolor=darkblue, linkcolor=darkblue, urlcolor=darkblue]{hyperref}

\usepackage{chngcntr}
\usepackage{enumitem}

\newtheorem{thm}{Theorem}[section]
\newtheorem{cor}[thm]{Corollary}

\newtheorem{lem}[thm]{Lemma}
\newtheorem{prop}[thm]{Proposition}

\newtheorem{assume}[thm]{Hypothesis}
\newtheorem{observe}[thm]{Observation}

\theoremstyle{definition}
\newtheorem{define}[thm]{Definition}
\newtheorem{warning}[thm]{Warning}
\newtheorem{const}[thm]{Construction}
\newtheorem{notate}[thm]{Notation}

\theoremstyle{remark}
\newtheorem{rem}[thm]{Remark}
\newtheorem{example}[thm]{Example}
\newtheorem{examples}[thm]{Examples}

\counterwithin*{thm}{section}

\newcommand{\QQ}{\mathbb{Q}}
\newcommand{\ZZ}{\mathbb{Z}}

\renewcommand{\SS}{\mathbb{S}}
\newcommand{\EE}{\mathbb{E}}

\newcommand{\A}{\mathcal{A}}

\newcommand{\C}{\mathcal{C}}
\newcommand{\D}{\mathcal{D}}
\newcommand{\F}{\mathcal{F}}
\newcommand{\G}{\mathcal{G}}
\newcommand{\U}{\mathcal{U}}
\newcommand{\E}{\mathcal{E}}

\newcommand{\M}{\mathcal{M}}

\renewcommand{\L}{\mathcal{L}}
\renewcommand{\P}{\mathcal{P}}
\newcommand{\R}{\mathcal{R}}

\newcommand{\I}{\mathcal{I}}
\newcommand{\J}{\mathcal{J}}
\newcommand{\cS}{\mathcal{S}}
\newcommand{\T}{\mathcal{T}}
\newcommand{\K}{\mathcal{K}}

\newcommand{\rN}{\mathrm{N}}
\newcommand{\rL}{\mathrm{L}}
\newcommand{\rR}{\mathrm{R}}

\DeclareMathOperator{\Hom}{Hom}

\DeclareMathOperator{\Map}{Map}
\DeclareMathOperator{\Fun}{Fun}

\DeclareMathOperator{\Id}{Id}

\DeclareMathOperator{\Ho}{Ho}

\DeclareMathOperator{\op}{op}

\DeclareMathOperator{\precolim}{colim}

\DeclareMathOperator{\Ne}{N}

\DeclareMathOperator{\ad}{ad}
\DeclareMathOperator{\ev}{ev}
\DeclareMathOperator{\coev}{coev}

\DeclareMathOperator{\Set}{Set}
\DeclareMathOperator{\Cat}{Cat}

\DeclareMathOperator{\Mod}{Mod}

\DeclareMathOperator{\tr}{tr}

\DeclareMathOperator{\Add}{Add}

\DeclareMathOperator{\Fin}{\F in}
\DeclareMathOperator{\CAlg}{CAlg}
\DeclareMathOperator{\CMon}{CMon}
\DeclareMathOperator{\Mon}{Mon}
\DeclareMathOperator{\TrFm}{TrFm}

\DeclareMathOperator{\Cone}{Cone}
\DeclareMathOperator{\Span}{Span}
\DeclareMathOperator{\Sp}{Sp}
\DeclareMathOperator{\fin}{fin}
\DeclareMathOperator{\St}{St}

\DeclareMathOperator{\car}{car}

\DeclareMathOperator{\dl}{dl}
\DeclareMathOperator{\Bord}{Bord}
\DeclareMathOperator{\Tw}{Tw}
\DeclareMathOperator{\Kan}{Kan}

\def\alp{{\alpha}}

\def\del{{\delta}}
\def\eps{{\varepsilon}}

\def\sig{{\sigma}}
\def\vphi{{\varphi}}

\def\Gam{{\Gamma}}
\def\Del{{\Delta}}
\def\Sig{{\Sigma}}

\def\Lam{{\Lambda}}

\def\vphi{{\varphi}}

\def\lrar{\longrightarrow}
\def\llar{\longleftarrow}
\def\hrar{\hookrightarrow}

\def\x{\stackrel}

\def\ovl{\overline}

\def\bksl{\;\backslash\;}

\def\colim{\mathop{\precolim}}

\newcommand{\twocell}[5]{\ar@<#4ex>@{}[#1] \ar@<#4ex>@{=>}?(#3)+/dl #5cm/;?(#3)+/ur #5cm/^{#2}}

\makeatletter
\renewcommand{\tocsection}[3]{%
  \indentlabel{\@ifnotempty{#2}{\bfseries\ignorespaces#1 #2\quad}}\bfseries#3} 

\renewcommand{\tocsubsection}[3]{%
  \indentlabel{\@ifnotempty{#2}{\hspace{1.6em}\ignorespaces#1 #2\quad}}#3}
\makeatother

\setlist[itemize]{leftmargin=*}
\setlist[enumerate]{leftmargin=*}

\title{Ambidexterity and the universality of finite spans}

\author{Yonatan Harpaz}
\date{}

\begin{document}
\maketitle

\begin{abstract}
Pursuing the notions of ambidexterity and higher semiadditivity as developed by Hopkins and Lurie, we prove that the span $\infty$-category of $m$-finite spaces is the free $m$-semiadditive $\infty$-category generated by a single object. Passing to presentable $\infty$-categories we obtain a description of the free presentable $m$-semiadditive $\infty$-category in terms of a new notion of $m$-commutative monoids, which can be described as spaces in which families of points parameterized by $m$-finite spaces can be coherently summed. Such an abstract summation procedure can be used to give a formal $\infty$-categorical definition of the finite path integral described by Freed, Hopkins, Lurie and Teleman in the context of $1$-dimensional topological field theories. 
\end{abstract}

\tableofcontents

\section{Introduction}

The notion of \textbf{ambidexterity}, as developed by Lurie and Hopkins in~\cite{ambi} in the $\infty$-categorical setting, is a duality phenomenon concerning diagrams $\vphi: K \lrar \C$ whose limit and colimit \textbf{coincide}. The simplest case where this can happen is when $K$ is empty. In this case a colimit of $K$ is simply an initial object of $\C$, and a limit of $K$ is a final object of $\C$. If $\C$ has both an initial object $\emptyset \in \C$ and a final object $\ast \in \C$ then there is an essentially unique map $\emptyset \lrar \ast$. Given that both $\emptyset$ and $\ast$ exist there is hence a canonical way to require that they coincide, namely, asserting that the unique map $\emptyset \lrar \ast$ is an equivalence. In this case we say that $\C$ is \textbf{pointed}. An object $0 \in \C$ which is both initial and final is called a \textbf{zero object}. 

Generalizing this property to cases where $K$ is non-empty involves an immediate difficulty. In general, even if $\vphi: K \lrar \C$ admits both a limit and a colimit, there is apriori no natural choice of a map relating the two. Informally speaking, choosing a map $\colim_{x \in K}\vphi(x) \lrar \lim_{x \in K}\vphi(x)$ is the same as choosing, compatibly for every two objects $x,y \in K$, a map $\vphi(x) \lrar \vphi(y)$ in $\C$. The diagram $\vphi$, on its part, provides such maps $\vphi(e): \vphi(x) \lrar \vphi(y)$ for every $e \in \Map_K(x,y)$. We thus have a whole space of maps $\vphi(e):\vphi(x) \lrar \vphi(y)$ at our disposal, but no a-priori way to choose a specific one naturally in both $x$ and $y$.

To see how this problem might be resolved assume for a moment that $\C$ is pointed, i.e., admits a zero object $0 \in \C$, and that $\Map_K(x,y)$ is either empty or contractible for every $x,y \in K$ (i.e.,\ $K$ is equivalent to a partially ordered set, or a poset). 
Then for every $X,Y \in \C$ there is a distinguished point in $\Map_\C(X,Y)$, namely the essentially unique map which factors as $f:X \lrar 0 \lrar Y$, where $0 \in \C$ is a zero object. We may call this map $X \lrar Y$ the \textbf{zero map}. We then obtain a choice of a map $\rN_{x,y}:\vphi(x) \lrar \vphi(y)$ which is natural in both $x$ and $y$: if $\Map_K(x,y)$ is contractible then we take $\rN_{x,y}$ to be $\vphi(e)$ for the essentially unique map $e: x \lrar y$, and if $\Map_K(x,y)$ is empty then we just take the zero map. It is then meaningful to ask whether the limits and colimit of a diagram $\vphi: K \lrar \C$ coincide: assuming both of them exist, we may ask whether the map $\rN_{\vphi}:\colim\vphi \lrar \lim\vphi$ we have just constructed is an equivalence.

For general posets and general pointed $\infty$-categories $\C$ the map $\rN_{\vphi}$ 
is rarely an equivalence. For example, if $K = [1]$ then $\rN_{\vphi}$ is simply the $0$-map. 
However, there is a class of posets for which this property turns out to yield something interesting: the class of \textbf{finite sets}, i.e., finite posets for which the order relation is the equality. In this case we may identify $\colim\vphi \simeq \coprod_{x \in K}\vphi(x)$ and $\lim\vphi \simeq \prod_{x \in K}\vphi(x)$. The map 
$$ \rN_\vphi:\coprod_{x \in K}\vphi(x) \lrar \prod_{x \in K}\vphi(x) $$
we constructed above is then given by the ``matrix'' of maps $[\rN_{x,y}]_{x,y \in K}$, where $\rN_{x,y}:\vphi(x) \lrar \vphi(y)$ is the identity if $x = y$ and the zero map if $x \neq y$. When a pointed $\infty$-category satisfies the property that $\rN_\vphi$ is an equivalence for every finite set $K$ and every diagram $\vphi:K \lrar \C$ we say that $\C$ is \textbf{semiadditive}. Examples of semiadditive $\infty$-categories include all abelian (discrete) categories and all stable $\infty$-categories. For more general examples, if $\C$ is any $\infty$-category with finite products then the $\infty$-category $\Mon_{\EE_\infty}(\C)$ of $\EE_\infty$-monoids in $\C$ is semiadditive.

In their paper~\cite{ambi}, Hopkins and Lurie observed that the passage from pointed $\infty$-categories to semiadditive ones is just a first step in a more general process. Suppose, for example, that $\C$ is a semiadditive $\infty$-category. Then for every $X,Y \in \C$, the mapping space $\Map_{\C}(X,Y)$ carries a natural structure of an $\EE_\infty$-monoid, where the sum of two maps $f,g: X \lrar Y$ is given by the composition
$$ \xymatrix{
X \ar[r]^-{\del} & X \times X \ar[rr]^-{f \times g} && Y \times Y \simeq Y \coprod Y \ar[r]^-{\eps} & Y \\
} .$$
Now suppose that $K$ is an $\infty$-category whose mapping spaces are equivalent to finite sets and that $\vphi: K \lrar \C$ is a diagram which admits both a limit and colimit. Then we may construct a natural map $\colim\vphi \lrar \lim\vphi$ by choosing, for every $x,y \in K$, the map 
\begin{equation}\label{e:map-1}
\rN_{x,y} = \sum_{e \in \Map_K(x,y)} \vphi(e):\vphi(x) \lrar \vphi(y) ,
\end{equation} 
where the sum is taken with respect to the natural $\EE_\infty$-monoid structure on $\Map_\C(X,Y)$. We may now ask if the induced map
\begin{equation}\label{e:map-2}
\rN_\vphi: \colim_{x \in K}\vphi(x) \lrar \lim_{x \in K}\vphi(x),
\end{equation} 
which is often called the \textbf{norm map}, is an equivalence. When~\eqref{e:map-2} is an equivalence for every finite \textbf{groupoid} $K$ we say that $\C$ is \textbf{$1$-semiadditive}. We note that when $\C$ is stable the cofiber of~\eqref{e:map-2} is also known as the associated \textbf{Tate object}, and hence $1$-semiadditivity in the stable context was often considered as a phenomenon of Tate vanishing.
This vanishing happens in many examples of interest, e.g., when $\C$ is a $\QQ$-linear $\infty$-category,
or when $\C$ is the $\infty$-category of $K(n)$-local spectra, where $K(n)$ is the $n$'th Morava $K$-theory at a prime $p$ (Hovey--Sadofsky--Greenlees \cite{HS96,GS96}). 
Hopkins and Lurie constructed an inductive approach for continuing this process, where at the $m$'th stage one considers $\infty$-groupoids whose homotopy groups are all finite and vanish above dimension $m$. In this paper we will refer to such $\infty$-groupoids as \textbf{$m$-finite spaces}. This yields the notion of \textbf{$m$-semiadditive $\infty$-category} for every $m \geq -1$. 
The main result of Hopkins--Lurie \cite{ambi} is that the $\infty$-category of $K(n)$-local spectra is not just $1$-semiadditive, but in fact $m$-semiadditive for every $m$. A similar result was recently established for the $\infty$-category of $T(n)$-local spectra by Carmeli--Schlank--Yanovski \cite{Tn}, generalizing a theorem of Khun \cite{Kuh04} on the $1$-semiadditivity of $T(n)$-local spectra, and using, among other things, results from the present paper. 
 
Our goal in this work is to form a link between the theory of higher semiadditivity as developed in~\cite{ambi} and the $\infty$-category of \textbf{spans} of $m$-finite spaces. To understand the role of this $\infty$-category, let us consider for a moment the central role played by the $\infty$-category $\Sp_{\fin}$ of \textbf{finite spectra} in the theory of \textbf{stable $\infty$-categories}. To begin, $\Sp_{\fin}$ can be described as the free stable $\infty$-category generated by a single object $\SS \in \Sp_{\fin}$, the sphere spectrum. Furthermore, one can use $\Sp_{\fin}$ in order to characterize stable $\infty$-categories inside the $\infty$-category $\Cat_{\fin}$ of all small $\infty$-categories with finite colimits (and right exact functors between them). Indeed, $\Cat_{\fin}$ carries a natural symmetric monoidal structure (see~\cite[\S 4.8.1]{higher-algebra}) whose unit is the the smallest full subcategory of spaces $\cS_{\fin} \subseteq \cS$ closed under finite colimits. One can then show that $\Sp_{\fin}$ is an \textbf{idempotent} object in $\Cat_{\fin}$ in the following sense: the suspension spectrum functor $\Sig^{\infty}_+:\cS_{\fin} \lrar \Sp_{\fin}$ induces an equivalence $\Sp_{\fin} \simeq \Sp_{\fin} \otimes \cS_{\fin} \x{\simeq}{\lrar} \Sp_{\fin} \otimes \Sp_{\fin}$. The fact that $\Sp_{\fin}$ is idempotent has a remarkable consequence: it endowed $\Sp_{\fin}$ with a canonical commutative algebra structure in $\Cat_{\fin}$ such that the forgetful functor $\Mod_{\Sp_{\fin}}(\Cat_{\fin}) \lrar \Cat_{\fin}$ is fully-faithful. From a conceptual point of view, this fact can be described as follows: given an $\infty$-category with finite colimits $\C$, the structure of being an $\Sp_{\fin}$-module is essentially unique once it exists, and can hence be considered as a \textbf{property}. One can then show that this property coincides with being stable. In other words, stable $\infty$-categories are exactly those $\C \in \Cat_{\fin}$ which admit an action of $\Sp_{\fin}$, in which case the action is essentially unique. 

This double aspect of stability, as either a property or a structure, is very useful. On one hand, in a higher categorical setting structures are often difficult to construct explicitly, while properties are typically easier to define and to check. On the other hand, having a higher categorical structure available is often a very powerful tool. An equivalence between a given property and the existence of a given structure allows one to enjoy both advantages simultaneously. Indeed, while the property of being stable is easy to define and often to establish, once we know that a given $\infty$-category is stable we can use the canonically defined $\Sp_{\fin}$-module structure at our disposal. For example, it implies that any stable $\infty$-category is canonically enriched in spectra, and in particular its  mapping spaces carry a canonical $\EE_\infty$-group structure.

In this paper we describe a completely analogous picture for the property of $m$-semiadditivity. Let $\K_m$ be the set of equivalence classes of $m$-finite Kan complexes. Let $\Cat_{\K_m}$ be the $\infty$-category of small $\infty$-categories which admit $\K_m$-indexed colimits and functors which preserve $\K_m$-indexed colimits between them. Then $\Cat_{\K_m}$ carries a natural symmetric monoidal structure whose unit is the $\infty$-category $\cS_m$ of $m$-finite spaces. Another object contained in $\Cat_{\K_m}$ is the $\infty$-category $\Span(\cS_m)$ whose objects are $m$-finite spaces and whose morphisms are given by spans (see~\ref{s:prelim} for a formal definition). Our main result can then be phrased as follows:
\begin{thm}
$\Span(\cS_m)$ is the free $m$-semiadditive $\infty$-category generated by a single object. More precisely, if $\D$ is an $m$-semiadditive $\infty$-category then evaluation at $\ast \in \Span(\cS_m)$ induces an equivalence
$$ \Fun_{\K_m}(\Span(\cS_m),\D) \x{\simeq}{\lrar} \D $$
where the left hand side denotes the $\infty$-category of functors which preserve $\K_m$-indexed colimits.
\end{thm}

Furthermore, we will show in \S\ref{s:modules} that $\Span(\cS_m)$ is in fact an \textbf{idempotent} object of $\Cat_{\K_m}$. Consequently, $\Span(\cS_m)$ carries a canonical commutative algebra structure in $\Cat_{\K_m}$, and the forgetful functor $\Mod_{\Span(\cS_m)}(\Cat_{\K_m}) \lrar \Cat_{\K_m}$ is fully-faithful. The structure of a $\Span(\cS_m)$-module on a given $\infty$-category $\C \in \Cat_{\K_m}$ is hence essentially a property. This property is exactly the property of being $m$-semiadditive. 

The flexibility of switching the point of view between a property and a structure seems to be especially useful in the setting of $m$-semiadditivity. Indeed, while $m$-semiadditivity is a property (involving the coincidence of limits and colimits indexed by $m$-finite spaces) it is quite hard to define directly. The reason, as described above, is that in order to define the various norm maps which are required to induce the desired equivalences, one needs to use the fact that the $\infty$-category in question is already known to be $(m-1)$-semiadditive. Even then, describing these maps requires an elaborate inductive process (see~\cite[\S 4]{ambi}). On the other hand, having a canonical $\Span(\cS_{m-1})$-module structure on an $(m-1)$-semiadditive $\infty$-category leads to a direct and short definition of when an $(m-1)$-semiadditive $\infty$-category is $m$-semiadditive (see, e.g., Corollary~\ref{c:limit}).

The picture becomes even more transparent when one passes to the world of \textbf{presentable $\infty$-categories}. Let $\Pr^L$ denote the $\infty$-category of presentable $\infty$-categories and left functors between them. Then one has a natural symmetric monoidal functor $\P_{\K_m}:\Cat_{\K_m} \lrar \Pr^L$ which sends $\C \in \Cat_{\K_m}$ to the $\infty$-category $\P_{\K_m}(\C)$ of presheaves of spaces on $\C$ that take $\K_m$-indexed colimits in $\C$ to limits of spaces. Applying this functor to $\Span(\cS_m)$ one obtains a presentable $\infty$-category which is equivalent to a certain $\infty$-category of \textbf{higher commutative monoids}, and which we will investigate in~\S\ref{s:monoids}. Informally speaking, an $m$-commutative monoid can be described as a space $X$ endowed with the following type of structure: for any map $f: K \lrar X$ from an $m$-finite space $K$ to $X$, we have an associated point $\int_{K} f \in X$, which we can think of as the ``continuous sum'' of the family of points $\{f(x)\}_{x \in K}$. This association is of course required to satisfy various compatibility conditions. For $m=0$ we have that $f$ is indexed by a finite set and we obtain the structure of an $\EE_{\infty}$-monoid. When $m=-1$ this is just the structure of a pointed space. Now since the functor $\P_{\K_m}$ is monoidal the $\infty$-category $\Mon_m$ of $m$-commutative monoids is idempotent as a presentable $\infty$-category, and the property characterizing $\Mon_m$-modules in $\Pr^L$ is again $m$-semiadditivity. There is, however, an advantage for considering $\Mon_m$ in addition to $\Span(\cS_m)$. Note that given an $m$-semiadditive $\infty$-category $\C$, the canonical action of $\Span(\cS_m)$ described above is not \textbf{closed} in general, i.e., it doesn't endow $\C$ with an enrichment in $\Span(\cS_m)$. It does, however, endow $\C$ with an enrichment in $\Mon_m$. In particular, mapping spaces in $\C$ are $m$-commutative monoids, and so we have a canonically defined summation over families of maps indexed by $m$-finite spaces. This structure can be used in order to redefine the norm maps of~\cite{ambi} and hence to define when an $(m-1)$-semiadditive $\infty$-category is $m$-semiadditive (in a manner analogous to the cases of $m=-1,0,1$ described above). Indeed, for every $m$-finite space $K$, any diagram $\vphi: K \lrar \C$ and any $x,y \in K$ we obtain a natural map
\begin{equation}\label{e:map-m}
\rN_{x,y}:\int_{e \in \Map_K(x,y)} \vphi(e):\vphi(x) \lrar \vphi(y).
\end{equation} 
by using the $(m-1)$-commutative monoid structure of $\Map_{\C}(x,y)$ and the fact that the mapping spaces in $K$ are $(m-1)$-finite. The compatible collection of maps $\rN_{x,y}$ then induces a map 
\begin{equation}\label{e:map-m-2}
\rN_\vphi: \colim_{x \in K}\vphi(x) \lrar \lim_{x \in K}\vphi(x)
\end{equation} 
which coincide with the norm maps constructed in~\cite{ambi}. In particular, an $(m-1)$-semiadditive $\infty$-category $\C$ is $m$-semiadditive if and only if the maps $\rN_\vphi$ are equivalences for every $K \in \K_m$ and every $\vphi: K \lrar \C$.

In the final part of the paper we will explain a relation between the above results and $1$-dimensional topological field theories, specifically with respect to the \textbf{finite path integral} described in~\cite[\S 3]{tft}. In particular, our approach allows one to formally define this finite path integral whenever the target $\infty$-category is $m$-semiadditive. This requires a description of the free $m$-semiadditive $\infty$-category generated by an arbitrary $\infty$-category $\D$, which we establish in~\ref{s:decorated} using the formalism of \textbf{decorated spans}. The link with finite path integrals is then described in~\ref{s:tft}.

\section{Preliminaries}\label{s:prelim}

In this paper we work in the higher categorical setting of $\infty$-categories as set up in~\cite{higher-topos}. In particular, by an $\infty$-category we will always mean a simplicial set $\C$ which has the right lifting property with respect to inner horns. We will often refer to the vertices of $\C$ as objects and to edges in $\C$ as morphisms. In the same spirit, if $\I$ is an ordinary category then we will often depict maps $\rN(\I) \lrar \C$ to an $\infty$-category $\C$ in diagrammatic form, as would be the case if $\C$ was an ordinary category. By a \textbf{space} we will always mean a Kan simplicial set, which we will generally regard as an $\infty$-groupoid, i.e., a $\infty$-category in which every morphism is invertible. Given an $\infty$-category $\C$, we will denote by $\C^{\simeq}$, the maximal subgroupoid (i.e, maximal sub Kan complex) of $\C$.

\subsection{$\infty$-Categories of spans}\label{s:spans}
In this section we will recall the definition of the $\infty$-category of spans in a given $\infty$-category $\C$ with admits pullbacks. To obtain more flexibility it will be useful to consider a slightly more general case, following the approach of Barwick \cite{barwick}. Recall that a functor $\F: \C \lrar \D$ between $\infty$-categories is called
\textbf{faithful} if for every $X,Y \in \C$ the induced map $\F_{X,Y}:\Map_\C(X,Y) \lrar \Map_\D(\F(X),\F(Y))$ is $(-1)$-truncated (i.e., each homotopy fiber of $\F_{X,Y}$ is either empty or contractible). Equivalently, $\F$ is faithful if the induced map $\Ho(\C) \lrar \Ho(\D)$ on homotopy categories is faithful and the square
$$ \xymatrix{
\C \ar^{\F}[r]\ar[d] & \D \ar[d] \\
\Ho(\C) \ar^{\Ho(\F)}[r] & \Ho(\D) \\
}$$
is homotopy Cartesian. In this case we will also say that $\C$ is a \textbf{subcategory of $\D$} (and will often omit the explicit reference to $\F$ and use the abusive notation $\C \subseteq \D$). A subcategory $\C \subseteq \D$ is called \textbf{wide} if the induced map $\C^{\simeq} \lrar \D^{\simeq}$ is an equivalence of spaces. Given a wide subcategory $\C \subseteq \D$ and a morphism $f$ in $\D$ we will say that $f$ \textbf{belongs} to $\C$ if it is equivalent in the arrow category of $\D$ to a morphism in the image of $\C$.
\begin{define}\label{d:wald}
Let $\C$ be an $\infty$-category. A \textbf{weak coWaldhausen} structure on $\C$ is a wide subcategory $\C^{\dagger} \subseteq \C$ such that any diagram
$$
\xymatrix{
& X \ar^{g}[d] \\
Y \ar^{f}[r] & Z \\
}
$$
in which $g$ belongs to $\C^{\dagger}$ extends to a pullback square
$$
\xymatrix{
P \ar[r]\ar_{g'}[d] & X \ar^{g}[d] \\
Y \ar^{f}[r] & Z \\
}
$$
in which $g'$ belongs to $\C^{\dagger}$. In this case we will refer to the pair $(\C,\C^{\dagger})$ as a weak coWaldhausen $\infty$-category.
\end{define}

\begin{example}
For any $\infty$-category $\C$ the maximal subgroupoid $\C^{\simeq} \subseteq \C$ is a weak coWaldhausen structure on $\C$. If $\C$ admits pullbacks then $\C$ itself is a weak coWaldhausen structure as well. We may consider these examples as the minimal and maximal coWaldhausen structures respectively.
\end{example}

Given a weak coWaldhausen $\infty$-category $(\C,\C^{\dagger})$ we would like to define an associated $\infty$-category $\Span(\C,\C^{\dagger})$. Informally speaking, $\Span(\C,\C^{\dagger})$ is the $\infty$-category whose objects are the objects of $\C$ and whose morphisms are given by diagrams of the form
\begin{equation}\label{e:span-C}
\xymatrix{
& Z\ar_{p}[dl]\ar^{q}[dr] & \\
X && Y \\
}
\end{equation}
such that $p$ belongs to $\C^{\dagger}$. We will refer to such diagrams as \textbf{spans} in $(\C,\C^{\dagger})$. A composition of two spans can be described by forming the diagram
$$ 
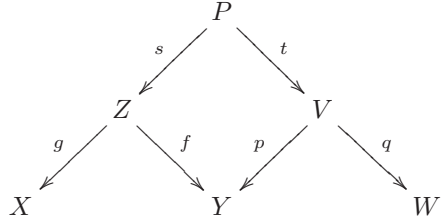
\xymatrix{
&& P \ar^{t}[dr]\ar_{s}[dl] && \\
& Z \ar^{f}[dr]\ar_{g}[dl] && V \ar^{q}[dr]\ar_{p}[dl] &\\
X && Y && W\\
}$$
in which the central square is a pullback square, and the external span is the composition of the two bottom spans. Note that since $p$ and $g$ belong to $\C^{\dagger}$ Definition~\ref{d:wald} insures that this pullback exists and $g \circ s$ belongs to $\C^{\dagger}$. To define $\Span(\C,\C^{\dagger})$ formally, it is convenient to use the \textbf{twisted arrow category} $\Tw(\Del^n)$ of the $n$-simplex $\Del^n$. This $\infty$-category can be described explicitly as the nerve is the category whose objects are pairs $(i,j) \in [n] \times [n]$ with $i \leq j$ and such that $\Hom((i,j),(i',j'))$ is a singleton if $i' \leq i \leq j \leq j'$ and empty otherwise. Given a weak coWaldhausen $\infty$-category $(\C,\C^{\dagger})$ we will say that a map $f: \Tw(\Del^n)^{\op} \lrar \C$ is \textbf{Cartesian} if for every $i' \leq i \leq j \leq j' \in [n]$ the square
$$ \xymatrix{
f(i',j') \ar[r]\ar[d] & f(i',j) \ar[d] \\
f(i,j') \ar[r] & f(i,j) \\
}$$
is Cartesian and its vertical maps belong to $\C^{\dagger}$. 

\begin{define}[cf.\cite{barwick}]
Let $(\C,\C^{\dagger})$ be a weak coWaldhausen $\infty$-category. The \textbf{span $\infty$-category} $\Span(\C,\C^{\dagger})$ is the simplicial set whose set of $n$-simplices is the set of Cartesian maps $f: \Tw(\Del^n)^{\op} \lrar \C$. 
\end{define}

By~\cite[\S 3.4-3.8]{barwick} the simplicial set $\Span(\C,\C^{\dagger})$ is always an $\infty$-category. We refer the reader to loc.\ cit.\ for a more detailed discussion of this construction and its properties.

\begin{rem}\label{r:map-span}
Let $(\C,\C^{\dagger})$ be a weak coWaldhausen $\infty$-category. Unwinding the definitions we see that the objects of $\Span(\C,\C^{\dagger})$ are the objects of $\C$ and the morphisms are given by spans of the form~\eqref{e:span-C} such that $p$ belongs to $\C^{\dagger}$. Furthermore, a homotopy from the span $X \llar Z \lrar Y$ to the span $X \llar Z' \lrar Y$ is given by an equivalence $\eta: Z \lrar Z'$ over $X \times Y$. Elaborating on this argument one can identify the mapping space from $X$ to $Y$ in $\C$ with the full subgroupoid of $(\C_{/X \times Y})^{\simeq}$ spanned by objects of the form~\eqref{e:span-C} such that $p$ belongs to $\C^{\dagger}$.
\end{rem}

\begin{rem}\label{r:subcategory}
It follows from Remark~\ref{r:les} that if $(\C^{\dagger})' \subseteq \C^{\dagger} \subseteq \C$ are two weak coWaldhausen structures on $\C$ then the associated functor $\Span(\C,(\C^{\dagger})') \lrar \Span(\C,\C^{\dagger})$ is a subcategory inclusion. In particular, $\C \simeq \Span(\C,\C^{\simeq})$ can be considered as a subcategory of $\Span(\C,\C^{\dagger})$ for any coWaldausen structure $\C^{\dagger}$. In the cases considered in this paper this subcategory will often be wide, see, e.g., Corollary~\ref{c:wide}.
\end{rem}

\subsection{Spans of finite spaces}


\begin{define}
Let $X$ be a space. For $n \geq 0$ we say that $X$ is \textbf{$n$-truncated} if $\pi_i(X,x) = 0$ for every $i > n$ and every $x \in X$. We will say that $X$ is $(-1)$-truncated if it is either empty of contractible and that $X$ is $(-2)$-truncated if it is contractible. We will say that a map $f: X \lrar Y$ is \textbf{$n$-truncated} if the homotopy fiber of $f$ over every point of $Y$ is $n$-truncated.
\end{define}

\begin{define}\label{d:finite}
Let $X$ be a space. For $n \geq -2$ we will say that $X$ is \textbf{$n$-finite} if it is $n$-truncated and all its homotopy groups/sets are finite. We will say that $X$ is \textbf{$\pi$-finite} if it is $n$-finite for some $n$.
\end{define}

The collection of weak equivalence types of $n$-finite Kan complexes is a set. We will denote by $\K_n$ be a complete set of representatives of equivalence types of $n$-finite Kan complexes. 

\begin{warning}
The notion of a finite space should not be confused with the notion of a space equivalent to a simplicial set with finitely many non-degenerate simplices. 
\end{warning}

Let $\Set_{\Del}$ denote the category of simplicial sets and let $\Kan \subseteq \Set_{\Del}$ denote the full subcategory spanned by Kan simplicial sets. Then $\Kan$ is a fibrant simplicial category and its coherent nerve $\cS := \rN(\Kan)$ is a model for the $\infty$-category of spaces. We let $\cS_n \subseteq \cS$ denote the full subcategory spanned by $n$-finite spaces. The condition of being $n$-finite is closed under pullbacks, which means that the $\infty$-category $\cS_n$ has pullbacks and those are computed in $\cS$. 
For $-2 \leq m \leq n$ we may consider the wide subcategory $\cS_{n,m} \subseteq \cS_n$ containing all objects and whose mapping spaces are spanned by the $m$-truncated maps. 
Then $(\cS_n,\cS_{m,n})$ is a weak coWaldhausen $\infty$-category, and we will denote by 
$$ \cS^m_n := \Span(\cS_n,\cS_{n,m}) $$ 
the associated span $\infty$-category (see~\S\ref{s:spans}). 
We note that $\cS_{n,-2}$ is just the maximal subgroupoid of $\cS_n$ and hence $\cS^{-2}_n \simeq \cS_n$. By~\cite[Theorem 1.3(iv)]{span}) the Cartesian monoidal product on $\cS_n$ induces a symmetric monoidal structure on $\cS^m_n$, which is given on the level of objects by $(X,Y) \mapsto X \times Y$, and on the level of morphisms by taking levelwise Cartesian products of spans. We remark that this monoidal structure on $\cS^m_n$ is \textbf{not} the Cartesian one.

\subsection{Colimits in $\infty$-categories of spans}
In this section we will prove some basic results concerning $\K_n$-indexed colimits in $\cS^m_n$. We begin with the following basic observation:
\begin{lem}\label{l:les}
The $\infty$-category $\cS_n$ admits $\K_n$-indexed colimits which are preserved and detected by the inclusion $\cS_n \subseteq \cS$.
\end{lem}
\begin{proof}
Since the inclusion $\cS_n \subseteq \cS$ is fully-faithful it detects colimits, i.e., every cone diagram in $\cS_n$ which is a colimit diagram in $\cS$ is already a colimit diagram in $\cS_n$. Since $\cS$ admits all small colimits it will suffice to show that $\cS_n \subseteq \cS$ is closed under $\K_n$-indexed colimits. More explicitly, we need to show that if $X$ is an $n$-finite Kan complex and $\vphi: X \lrar \cS$ is an $X$-indexed diagram of spaces such that $\vphi(x)$ is $n$-finite for every $x \in X$ then the colimit of $\vphi$ is also $n$-finite. For this it is convenient to use the fact that colimits in spaces can be modeled by the total space of the left fibration $p:E_\vphi \lrar X$ classified by $\vphi$ (see~\cite[Corollary 3.3.4.6]{higher-topos}). Since $X$ is a Kan complex $E_\vphi$ is also a Kan complex and $p$ is a Kan fibration. We thus need to show that the total space of a Kan fibration with an $n$-finite base and $n$-finite fibers is also $n$-finite. But this is now a direct consequence of the long exact sequence of homotopy groups associated to a Kan fibration.
\end{proof}

\begin{rem}\label{r:les}
Lemma~\ref{l:les} implies that we can model the colimit of a diagram $\vphi: X \lrar \cS_n$ by the total space of the Kan fibration $E_\vphi \lrar X$ classified by $\vphi$. More precisely, for such a $\vphi$ the total space $E_\vphi$ is $n$-finite 
and the collection of fiber inclusions $\{\vphi(x)\simeq(E_\vphi)_x \lrar E_\vphi\}_{x \in X}$ exhibits $E_\vphi$ as the colimit of $\vphi$ in $\cS_n$.
\end{rem}

Given a space $X$ and a point $x \in X$ we will denote by $i_x: \ast \lrar X$ the map which sends the point to $x$.

\begin{lem}\label{l:colimits}
Let $\D$ be an $\infty$-category which admits $\K_n$-indexed colimits and let $\F: \cS_{n} \lrar \D$ be a functor. Then $\F$ preserves $\K_n$-indexed colimits if and only if for every $X \in \cS_n$ the collection $\{\F(i_x):\F(\ast) \lrar \F(X)\}_{x \in X}$ exhibits $\F(X)$ as the colimit of the constant $X$-indexed diagram with value $\F(\ast)$.
\end{lem}
\begin{proof}
The ``only if'' direction is due to the fact that the collection of maps $i_x:\ast \lrar X$ exhibits $X$ as the colimit in $\cS_n$ of the constant $X$-indexed diagram with value $\ast$ (see Remark~\ref{r:les}). Now suppose that for every $X \in \cS_n$ the collection $\{\F(i_x)\}_{x \in X}$ exhibits $\F(X)$ as the colimit of the constant $X$-indexed diagram with value $\F(\ast)$. Let $Y \in \K_n$ be an $n$-finite space and let $\vphi: Y \lrar \cS_n$ be a $Y$-indexed diagram in $\cS_n$. Let $p:E_{\vphi} \lrar Y$ be the Kan fibration classified by $\vphi$, so that, in light of Remark~\ref{r:les}, we have that $E_\vphi$ is $n$-finite and the collection of fiber inclusions $\{(E_\vphi)_y \simeq \vphi(y) \lrar E_\vphi\}_{y \in Y}$ exhibits $E_\vphi$ as the colimit of $\vphi$ in $\cS_n$. To finish the proof we need to show that the collection of maps $\{\F(\vphi(y)) \lrar \F(E_\vphi)\}_{y \in Y}$ exhibits $\F(E_\vphi)$ as the colimit of the diagram $\psi := \F \circ \vphi: Y \lrar \D$.
Consider the \textbf{lax} commutative diagram
\begin{equation}\label{e:lax} 
\xymatrix{
E_\vphi \ar^{p}[r]\ar_{\F(\ast)}[dr] & Y \ar[r]\ar_-{\psi}[d] & \ast \ar^{\F(E_\vphi)}[dl] \\
& \D \twocell{ur}{}{0.2}{3}{0.15} \twocell{ur}{}{0.5}{1}{0.15} & \\
}
\end{equation}
where the left diagonal functor is the constant functor with value $\F(\ast)$, and the right diagonal functor sends the point to $\F(E_\vphi)$. By our assumption for every $y \in Y$ the collection of maps $\{\F(i_{z}):\F(\ast) \lrar \F(\vphi(y))\}_{z \in \vphi(y)}$ exhibits $\psi(y) = \F(\vphi(y))$ as the colimit in $\D$ of the constant $\vphi(y)$-indexed diagram with value $\F(\ast)$. Identifying $\vphi(y)$ with the homotopy fiber of $p: E_\vphi \lrar Y$ over $y$ we may conclude that left lax triangle in~\eqref{e:lax} exhibits $\psi:Y \lrar \D$ as a left Kan extension along $p: E_\vphi \lrar Y$ of the constant diagram $E_\vphi \lrar \D$ with value $\F(\ast)$. Similarly, our assumption implies that the collection of maps $\{\F(i_z):\F(\ast) \lrar \F(E_{\vphi})\}_{z \in E_\vphi}$ exhibits $\F(E_\vphi)$ as the colimit in $\D$ of the constant $E_\vphi$-indexed diagram with value $\F(\ast)$, and so the external lax triangle is a left Kan extension diagram as well. It then follows the right lax triangle is a left Kan extension triangle by the pasting lemma for left Kan extensions. But this exactly means that the collection of maps $\{\psi(y) \lrar \F(E_\vphi)\}_{y \in Y}$ exhibits $\F(E_\vphi)$ as the colimit in $\D$ of the diagram $\psi: Y \lrar \D$, as desired.
\end{proof}

\begin{prop}\label{p:adjoint}
For every $-2 \leq m \leq n$ the subcategory inclusion $\cS_n \hrar \cS^m_n$ preserves $\K_n$-indexed colimits.
\end{prop}
\begin{proof}
By Lemma~\ref{l:colimits} it will suffice to show that for every $X \in \cS^m_n$, the collection of morphisms $\{i_x:\ast \lrar X\}_{x \in X}$ exhibit $X$ as the colimit of the constant diagram $\{\ast\}_{x \in X}$ in $\cS^m_n$. Equivalently, we need to show that given any test object $Y \in \cS^m_n$, the map 
$$ \Map_{\cS^m_n}(X,Y) \lrar \Map_{\cS}(X,\Map_{\cS^m_n}(\ast,Y)) $$ 
determined by the collection of restriction maps $i_x \circ (-): \Map_{\cS^m_n}(X,Y) \lrar \Map_{\cS^m_n}(\ast,Y)$ is an equivalence of spaces. Let $p_X: X \times Y \lrar X$ denote the projection on the first coordinate. By Remark~\ref{r:map-span} we may identify $\Map_{\cS^m_n}(X,Y)$ with the full subgroupoid of $((\cS_n)_{/X \times Y})^{\simeq}$ spanned by those objects $Z \lrar X \times Y$ such that the composite map $Z \lrar X \times Y \x{p_X}{\lrar} X$ is $m$-truncated. Under this equivalence, the restriction map $i_x \circ (-)$ is induced by the pullback functor $i_{\{x\} \times Y}^*: (\cS_n)_{/X \times Y} \lrar (\cS_n)_{/\{x\} \times Y}$.
Now by the straightening-unstraightening equivalence the collection of pullback functors $i_{\{x\} \times \{y\}}^*: \cS_{/X \times Y} \lrar \cS$ induces an equivalence of $\infty$-categories
$$ \St:\cS_{/X \times Y} \x{\simeq}{\lrar} \Fun(X \times Y,\cS) .$$
Using straightening-unstraightening again and the equivalence $\Fun(X \times Y,\cS) \simeq \Fun(X,\Fun(Y,\cS))$ we may conclude that the collection of pullback functors 
$i_{\{x\} \times Y}^*: \cS_{/X \times Y} \lrar \cS_{/\{x\} \times Y}$ induces an equivalence of $\infty$-categories
$$ \St_X:\cS_{/X \times Y} \x{\simeq}{\lrar} \Fun(X, \cS_{/Y}) .$$
and hence an equivalence on the corresponding maximal subgroupoids
$$ \St^{\simeq}_X: (\cS_{/X \times Y})^{\simeq} \x{\simeq}{\lrar} \Fun(X, \cS_{/Y})^{\simeq} \simeq \Map(X, (\cS_{/Y})^{\simeq}).$$
Given an object $Z \lrar X \times Y$ in $(\cS_{/X \times Y})^{\simeq}$, the condition that the composite map $Z \lrar X \times Y \x{p_X}{\lrar} X$ is an $m$-truncated map is equivalent to the condition that the essential image of $\St^{\simeq}_X(Z): X \lrar (\cS_{/Y})^{\simeq}$ is contained in $((\cS_m)_{/Y})^{\simeq}$. Furthermore, since $X$ is $n$-truncated this condition automatically implies that $Z$ is $n$-truncated. Identifying $((\cS_m)_{/Y})^{\simeq}$ with $\Map_{\cS^m_n}(\ast,Y)$ we may then conclude that the collection of pullback functors $i_{\{x\} \times Y}^*: (\cS_n)_{/X \times Y} \lrar (\cS_n)_{/\{x\} \times Y}$ induces an equivalence of $\infty$-groupoids
$$ \Map_{\cS^m_n}(X,Y) \x{\simeq}{\lrar} \Map(X, \Map_{\cS^m_n}(\ast,Y)), $$
as desired. 
\end{proof}

\begin{lem}\label{l:yoneda}
Let $-2 \leq m \leq n$ be integers. Then any equivalence in $\cS^m_n$ is homotopic to the image of an equivalence in $\cS_n \subseteq \cS^m_n$.
\end{lem}
\begin{proof}
Given a span
\begin{equation}\label{e:corr}
\xymatrix{
& Z\ar_{p}[dl]\ar^{q}[dr] & \\
X && Y \\
}
\end{equation}
we may associate to it the functor $p_!q^*: \Fun(Y,\cS) \lrar \Fun(X,\cS)$, where $q^*: \Fun(Y,\cS) \lrar \Fun(Z,\cS)$ is the restriction functor and $p_!:\Fun(Z,\cS) \lrar \Fun(X,\cS)$ is given by left Kan extension. The Beck-Chevalley condition (see~\cite[Proposition 4.3.3]{ambi}) implies that this association respects composition of spans up to homotopy. It follows that if~\eqref{e:corr} is an equivalence in $\cS^m_n$ then the induced functor $p_!q^*: \Fun(Y,\cS) \lrar \Fun(X,\cS)$ is an equivalence of $\infty$-categories. In particular, $p_!q^*$ preserves terminal objects. Since the restriction functor $q^*$ also preserves terminal object it now follows that $p_!(\ast):X \lrar \cS$ is the terminal functor. On the other hand, by the pointwise construction of left Kan extensions $p_!(\ast)$ is also the functor which associates to $x \in X$ the homotopy fiber $Z_x$ of $p$ over $x$. It then follows that the homotopy fibers of $p$ are all contractible, and so $p$ is an equivalence. Let $r: X \lrar Z$ be a homotopy inverse of $p$. Then~\eqref{e:corr} is homotopic to the image in $\Map_{\cS^m_n}(X,Y)$ of the arrow $q\circ r\in \Map_{\cS_n}(X,Y)$, as desired.

\end{proof}

\begin{cor}\label{c:wide}
The subcategory inclusion $\cS_n \subseteq \cS_n^m$ is wide.
\end{cor}

\begin{cor}\label{c:cool}
Let $X$ be a space. Then any $X$-indexed diagram in $\cS^m_n$ comes from an $X$-indexed diagram in $\cS_n$.
\end{cor}

\begin{cor}\label{c:n-colimits}
For every $-2 \leq m \leq n$ the $\infty$-category $\cS^m_n$ admits $\K_n$-indexed colimits. Furthermore, if $\F: \cS^m_n \lrar \D$ is any functor then $\F$ preserves $\K_n$-indexed colimits if and only if the composed functor $\cS_n \hrar \cS^m_n \lrar \D$ preserves $\K_n$-indexed colimits.
\end{cor}
\begin{proof}
Combine Corollary~\ref{c:cool} and Proposition~\ref{p:adjoint}. 
\end{proof}

\begin{cor}\label{c:n-tensor}
The symmetric monoidal product $\cS^{m}_{n} \times \cS^m_n \lrar \cS^m_n$ preserves $\K_n$-indexed colimits in each variable separately.
\end{cor}
\begin{proof}
We have a commutative diagram
$$ \xymatrix{
\cS \times \cS \ar[d] & \cS_n \times \cS_n \ar[l]\ar[r]\ar[d] & \cS^m_n \times \cS^m_n \ar[d] \\
\cS & \cS_n \ar[r]\ar[l] & \cS^m_n \\
}$$
Where the left and middle vertical maps are the respective Cartesian product and right vertical map is the one induced by the Cartesian product on span $\infty$-categories. Since Cartesian products in $\cS$ preserve colimits in each variable separately and the inclusion $\cS_n \hrar \cS$ preserves products and $\K_m$-indexed colimits (Lemma~\ref{l:les}) we get that Cartesian products in $\cS_n$ preserve $\K_n$-indexed colimits in each variable separately. Since the functor $\cS_n \lrar \cS^m_n$ is essentially surjective the desired result now follows from Corollary~\ref{c:cool} and Proposition~\ref{p:adjoint}. 
\end{proof}

We will denote by $\Cat_{\K_n}$ the $\infty$-category of small $\infty$-categories which admit $\K_n$-indexed colimits and functors which preserve $\K_n$-indexed colimits between them. If $\C,\D$ are $\infty$-categories which admit $\K_n$-indexed colimits then we denote by $\Fun_{\K_n}(\C,\D) \subseteq \Fun(\C,\D)$ the full subcategory spanned by those functors which preserve $\K_n$-indexed colimits. Recall that by~\cite[corollary 4.8.4.1]{higher-algebra} we may endow this $\infty$-category with a symmetric monoidal structure $\Cat^{\otimes}_{\K_n} \lrar \Ne(\Fin_\ast)$ such that for $\C,\D \in \Cat_{\K_n}$ their tensor product $\C \otimes_{\K_n} \D$ admits a map $\C \times \D \lrar \C \otimes_{\K_n} \D$ from the Cartesian product and is characterized by the following universal property: for every $\infty$-category $\E \in \Cat_{\K_n}$ the restriction
$$ \Fun_{\K_n}(\C \otimes_{\K_n} \D,\E) \lrar \Fun(\C \times \D,\E) $$
is fully-faithful and its essential image is spanned by those functors $\C \times \D \lrar \E$ which preserve $\K_n$-indexed colimits in each variable separately. In particular, we may identify \textbf{commutative algebra objects} in $\Cat_{\K_m}$ with symmetric monoidal $\infty$-categories which admit $\K_m$-indexed colimits and such that the monoidal product preserves $\K_m$-indexed colimits in each variable separately.

Corollary~\ref{c:n-colimits} and Corollary~\ref{c:n-tensor} now imply the following:
\begin{cor}\label{c:algebra}
The $\infty$-category $\cS^m_n$ together with its symmetric monoidal structure determines a commutative algebra object in $\Cat^{\otimes}_{\K_m}$
\end{cor}

\section{Ambidexterity and duality}\label{s:ambi}

\begin{define}[see{~\cite[Definition 4.4.2]{ambi}}]
Let $\D$ be an $\infty$-category and $-2 \leq m$ an integer. Following~\cite{ambi}, we shall say that $\D$ is \textbf{$m$-semiadditive} if $\D$ admits $\K_m$-indexed colimits and every $m$-finite space is $\D$-ambidextrous in the sense of~\cite[Definition 4.3.4]{ambi}.
\end{define}

Informally speaking, $m$-semiadditive $\infty$-categories are $\infty$-categories in which $\K_m$-indexed colimits and limits coincide. The reason we do not recall~\cite[Definition 4.3.4]{ambi} in full is that it requires a somewhat elaborate inductive process in order to define the maps which induce the desired equivalence. That said, if $\D$ is an $\infty$-category with $\K_m$-indexed colimits which admits a compatible action of $\cS^{m-1}_m$, then we will see below that the condition that $\D$ is $m$-semiadditive can be expressed rather succinctly (see Proposition~\ref{p:non-deg} and Corollary~\ref{c:limit}). On the other hand, the main result of this paper (Theorem~\ref{t:main} below) implies in particular that any $(m-1)$-semiadditive $\infty$-category which admits $\K_m$-indexed colimits acquires a canonical action of $\cS^{m-1}_m$, and so this approach can be considered as an alternative way to define higher semiadditivity.

\begin{examples}\
\begin{enumerate}
\item
An $\infty$-category $\D$ is $(-1)$-semiadditive if and only if it is \textbf{pointed}, i.e., contains an object which is both initial and final.
\item
Every stable $\infty$-category is $0$-semiadditive.
\item
Let $\D$ be an $\infty$-category which admits finite products. Then the $\infty$-category of $\EE_\infty$-monoid objects in $\D$ is $0$-semiadditive (see~\S\ref{s:monoids}).
\item
For any prime $p$ and integer $n \geq 0$, the $\infty$-category of $K(n)$-local spectra is $m$-semiadditive for any $m$ (where $K(n)$ denotes the Morava $K$-theory spectrum at height $n$). This is the main result of~\cite{ambi}.
\item
For any prime $p$ and integer $n \geq 0$, the $\infty$-category of $T(n)$-local spectra is $m$-semiadditive for any $m$ (where $T(n)$ denotes the telescope of a finite $p$-local type $n$ spectrum). This is the main result of~\cite{Tn}.
\item
For every $-2 \leq m \leq n$ the $\infty$-category $\cS^m_n$ is $m$-semiadditive (see Corollary~\ref{c:origin} below). 
\item
The $\infty$-category $\Cat_{\K_m}$ of small $\infty$-categories which admit $\K_m$-indexed colimits is $m$-semiadditive (see Proposition~\ref{p:cat_Km}
 below).
\item
If $\D$ is $m$-semiadditive then $\D^{\op}$ is $m$-semiadditive.
\end{enumerate}
\end{examples}

\begin{notate}
For a space $X$ and an $\infty$-category $\D$, we will typically refer to functors $X \lrar \D$ as $\D$-valued \textbf{local systems} on $X$. We will denote by $\D^X := \Fun(X,\D)$ the $\infty$-category of $\D$-valued local system on $X$. Given a map $f: X \lrar Y$ of spaces we have a restriction functor $f^*: \D^Y \lrar \D^X$. If $\D$ admits $\K_m$-indexed colimits and the homotopy fibers of $f$ are $m$-finite 
then $f^*$ admits a left adjoint $f_!: \D^X \lrar \D^Y$ given by left Kan extension. If in addition $\D$ is $m$-semiadditive then $f_!$ is also \textbf{right adjoint} to $f^*$. We will say that a natural transformation $u: \Id \Rightarrow f_!f^*$ \textbf{exhibits $f$ as $\D$-ambidextrous} if it is a unit of an adjunction $f^* \dashv f_!$.
\end{notate}

In this section we fix an integer $m \geq -1$ and consider the situation where $\D$ is an $\infty$-category satisfying the following properties:
\begin{assume}\label{a:fix}\
\begin{enumerate}
\item
$\D$ admits $\K_m$-indexed colimits.
\item
$\D$ is $(m-1)$-semiadditive.
\item
$\D$ admits a structure of a $\cS^{m-1}_{m}$-module in $\Cat_{\K_m}$. In other words, there is an action of the monoidal $\infty$-category $\cS^{m-1}_m$ on $\D$ such that the action map $\cS^{m-1}_m \times \D \lrar \D$ preserves $\K_m$-indexed colimits in each variable separately. Following~\cite{ambi}, we will denote the functor $X \otimes (-)$ also by $[X]: \D \lrar \D$.
\end{enumerate}
\end{assume}

Our first goal in this section is to show that if $\D$ satisfies Hypothesis~\ref{a:fix}, and if $f: X \lrar Y$ is an $(m-1)$-truncated map of $m$-finite spaces, then a unit transformations $u: \Id \Rightarrow f_!f^*$ exhibiting $f$ as $\D$-ambidextrous can be written in terms of the $\cS^{m-1}_m$ action on $\D$ (see Lemma~\ref{l:lem-2}). We will use this description in order to give an explicit criteria characterization those $\D$ satisfying~\ref{a:fix} which are also $m$-semiadditive (see Proposition~\ref{p:non-deg} and Corollary~\ref{c:limit}).

Let $X \in \cS^{m-1}_{m}$ be an object. Recall that for a point $x \in X$ we denote by $i_x: \ast \lrar X$ the map in $\cS_m \subseteq \cS^{m-1}_{m}$ which sends $\ast$ to the point $x$. By Proposition~\ref{p:adjoint} and Lemma~\ref{l:colimits}, for any object $D \in \D$ the collection of induced maps
$$ (i_x)_*: [\ast](D) \lrar [X](D) $$
exhibits $[X](D)$ as the colimit in $\D$ of the constant $X$-indexed diagram with value $[\ast](D) = D$. Letting $D$ vary we obtain that the maps $[i_x]:[\ast] \Rightarrow [X]$ exhibit $[X]$ as the colimit in $\Fun(\D,\D)$ of the constant diagram $\{\Id\}_{x \in X}$. In particular, we may identify the functor $[X]$ with the composed functor $\D \x{p^*}{\lrar} \D^X \x{p_!}{\lrar} \D$, where $p: X \lrar \ast$ is the terminal map.

\begin{lem}\label{l:basic}
Let $\D$ be as in Hypothesis~\ref{a:fix}, let $X$ be an $m$-finite space and let $p: X \lrar \ast$ be the terminal map in $\cS_m$ (which we naturally consider as a map in $\cS^{m-1}_m$). Then the natural transformation 
$$ p_!p^* \simeq [X] \x{[p]}{\Rightarrow}  [\ast] \simeq \Id $$
is a counit exhibiting $p_!: \D^X \lrar \D$ as left adjoint to $p^*:\D \lrar \D^X$. 
\end{lem}
\begin{proof}
If $X$ is empty then $[X]$ is initial in $\Fun(\D,\D)$, and since such a counit exists it must be homotopic to $[p]$. We may hence suppose that $X$ is not empty. Since $p_!$ is left adjoint to $p^*$ the desired claim is equivalent to the natural transformation $[p]^{\ad}:p^* \Rightarrow p^*$ adjoint to $[p]:p_!p^* \Rightarrow \Id$ being an equivalence. 
We note that to specifying a natural transformation $T:p^* \Rightarrow p^*$ is the same as giving an $X$-indexed family of natural transformations $\{T_x: \Id \Rightarrow \Id\}_{x \in X}$ from the identify functor $\Id \in \Fun(\D,\D)$ to itself. Furthermore, since the maps $[i_x]:[\ast] \Rightarrow [X]$ exhibit $[X]$ as the colimit in $\Fun(\D,\D)$ of the constant diagram $\{\Id\}_{x \in X}$ it follows that if $T: [X] \simeq p_!p^* \Rightarrow \Id$ is a natural transformation then the 
adjoint natural transformation $T^{\ad}: p^* \Rightarrow p^*$ is given by the family $\{T \circ [i_x]:\Id \Rightarrow \Id\}_{x \in X}$.
Since $p \circ i_x: \ast \lrar \ast$ is an equivalence in $\cS^{m-1}_m$ it follows that $[p] \circ [i_x]: \Id \rightarrow \Id$ is a natural equivalence for every $x \in X$ and so the natural transformation $[p]^{\ad}: p^* \Rightarrow p^*$ is an equivalence, as desired.
\end{proof}

\begin{define}\label{d:dual}
Given a map $f: X \lrar Y$ in $\cS_{m-1}$ let us denote by $\hat{f}: Y \lrar X$ the morphism in $\cS^{m-1}_m$ determined by the span
$$ \xymatrix{
& X \ar[dr]\ar_{f}[dl] & \\
Y && X \\
}.$$
We will refer to $\hat{f}$ as the \textbf{dual span} of $f$.
\end{define}


\begin{lem}\label{l:lem-1}
Let $\D$ be as in Hypothesis~\ref{a:fix}, let $X$ be an $(m-1)$-finite space and let $p: X \lrar \ast$ be the terminal map in $\cS_{m-1}$. Then the natural transformation
$$ \Id  \simeq [\ast] \x{[\hat{p}]}{\Rightarrow} [X] \simeq p_!p^* $$ 
is a unit exhibiting $p_!$ as \textbf{right adjoint} to $p^*$, where $\hat{p}: \ast \lrar X$ is the span dual to $p$ (see Definition~\ref{d:dual}). In other words, it exhibits $p:X \lrar \ast$ as $\D$-ambidextrous.
\end{lem}
\begin{proof}
Since $X$ is $\D$-ambidextrous there exists a counit $v_X: p^*p_! \Rightarrow \Id$ exhibiting $p_!$ as right adjoint to $p^*$. As in Hopkins--Lurie \cite[Notation 5.1.7]{ambi} let us define the \textbf{trace form} $\TrFm_X:[X]\circ [X] \Rightarrow \Id$ by the composition
$$ \xymatrix{ 
(p_!p^*)(p_!p^*) \simeq p_!(p^*p_!)p^* \ar@{=>}[r]^-{p_!v_Xp^*} & p_!p^* \ar@{=>}[r]^-{\phi_X} & \Id \\
} $$
where $\phi_X$ is a counit exhibiting $p_!$ as left adjoint to $p^*$. Since $\D$ is assumed to be $(m-1)$-semiadditive,~\cite[Proposition 5.1.8]{ambi} implies that the trace form exhibits $[X]$ as self dual in $\Fun(\D,\D)$. Let $u_X: \Id \Rightarrow p_!p^*$ be a unit which is compatible with $v_X$. It will then be enough to show that $[\hat{p}]$ is equivalent to $u_X$ in the arrow category of $\Fun(\D,\D)$. Since $[X]$ is self dual it will suffice to compare the natural transformations $[X] \Rightarrow \Id$ which are dual to $u_X$ and $[\hat{p}]$ respectively. In the case of $u_X$ we observe that
$$ \xymatrix{ 
p_!p^* \ar@{=>}[r]^-{(p_!p^*)u_X} & (p_!p^*)(p_!p^*) \simeq p_!(p^*p_!)p^* \ar@{=>}[r]^-{p_!v_Xp^*} & p_!p^* \\
} $$
is homotopic to the identity in light of the compatibility of $u_X$ and $v_X$. It hence follows that the map $[X] \Rightarrow \Id$ dual to $u_X$ is the counit
$$ \phi_X: p_!p^* \Rightarrow \Id .$$
On the other hand, the action functor $\cS^{m-1}_{m} \lrar \Fun(\D,\D)$ is monoidal and sends $X$ to $[X]$. Since $X$ is $(m-1)$-finite it is self-dual in $\cS^{m-1}_m$ (cf.\ Remark~\ref{r:self-dual} below). It follows that the dual of $[\hat{p}]$ is the image of the dual of $\hat{p}$ in $\cS^{m-1}_{m}$, which is given 
by the image in $\cS^{m-1}_m$ of the terminal map $p: X \lrar \ast$ of $\cS_m$. It will hence suffice to show that $[p]$ is equivalent to $\phi_X$ in the arrow category of $\Fun(\D,\D)$. But this now follows from Lemma~\ref{l:basic}.
\end{proof}


\begin{define}\label{d:st}
Given an $(m-1)$-truncated map $f:X \lrar Y$ of $m$-finite spaces let us denote by $\St_f: Y \lrar \cS$ the diagram obtained by applying to $f$ the straightening construction (see~\cite[\S 2.1]{higher-topos}). Informally, $\St_f: Y \lrar \cS$ sends $y \in Y$ to the homotopy fiber $X_y$ of $f$ over $y$. Since $f$ is $(m-1)$-truncated every homotopy fiber $X_y$ is $(m-1)$-finite and we may consequently consider $\St_f$ as a functor $Y \lrar \cS_{m-1}$. Using the inclusions $\cS_{m-1} \subseteq \cS_m \subseteq \cS^{m-1}_m$ (see Corollary~\ref{c:wide}) we may further consider $\St_f$ as a functor $Y \lrar \cS^{m-1}_m$, i.e., as a $\cS^{m-1}_m$-valued local system. To avoid confusion we will use the notation $\ovl{\St}_f: Y \lrar \cS^{m-1}_m$ to denote the straightening of $f$ when considered as taking values in $\cS^{m-1}_m$.
\end{define}

\begin{const}\label{c:st}
Let $f: X \lrar Y$ be an $(m-1)$-truncated map of $m$-finite spaces and let $\ovl{\St}_f: Y \lrar \cS^{m-1}_m$ be its straightening as in Definition~\ref{d:st}. The action of $\cS^{m-1}_m$ on $\D$ induces a pointwise action of $(\cS^{m-1}_{m})^Y$ on $\D^Y$. In particular, the action of $\ovl{\St}_f \in (\cS^{m-1}_{m})^Y$ determines a functor
$$ [\ovl{\St}_f]: \D^Y \lrar \D^Y, $$
given informally on a local system $\L \in \D^Y$ by the formula 
$$ [\ovl{\St}_f](\L)(y) = \ovl{\St}_f(y)(\L(y)) = [X_y](\L(y)) .$$
\end{const}


\begin{lem}\label{l:left-kan}
Let $\D$ be as in Hypothesis~\ref{a:fix} and let $f: X \lrar Y$ be an $(m-1)$-truncated map of $m$-finite spaces. Then there is a natural equivalence
$$ [\ovl{\St}_f] \simeq f_!f^* $$
of functors $\D^Y \lrar \D^Y$
\end{lem}
\begin{proof}
Consider the base change $g: X \times_Y X \lrar X$ of $f$ along itself. Then the diagonal map $\del:X \lrar X \times_Y X$ determines a section of $g$, which we can consider as a map of spaces over $X$. Applying the straightening construction as in Definition~\ref{d:st} over the base $X$ we obtain a natural transformation
$$ \ovl{\St}_{\del}: \ovl{\St}_{\Id} \simeq \ast \Rightarrow \ovl{\St}_g .$$
from the straightening of the identity $\Id:X \lrar X$ (which is the constant diagram with value $\ast$), to the straightening of $g: X \times_Y X \lrar X$. The latter, in turn, is naturally equivalent to the restriction along $f: X \lrar Y$ of $\ovl{\St}_f: Y \lrar \cS^{n-1}_n$, by the compatibility of unstraightening with base change (see~\cite{higher-topos}). Applying Construction~\ref{c:st} we obtain a natural transformation
$$ [\ovl{\St}_{\del}]:\Id \Rightarrow [\ovl{\St}_g] \simeq [f^*\ovl{\St}_f] $$
of functors from $\D^X$ to $\D^X$. Pre-composing with the restriction functor $f^*: \D^Y \lrar \D^X$ we now obtain a natural transformation
$$ \del_\ast: f^* \Rightarrow [f^*\ovl{\St}_f] \circ f^* \simeq f^* \circ [\ovl{\St}_f] $$
of functors $\D^Y \lrar \D^X$. 
For a local system $\L \in \D^Y$, the component at $\L$ of $\del_*$ is a map
$$ \del_{\L}:f^*(\L) \Rightarrow f^*([\ovl{\St}_f]\L) .$$
of $\D$-valued local systems on $X$. To finish the proof it will now suffice to show that $\del_{\L}$ exhibits $[\ovl{\St}_f]\L$ as the left Kan extension of $f^*(\L): X \lrar \D$ along $f$. Indeed, by the pointwise formula for the left Kan extension we need to check that for every $y  \in Y$ the diagram 
$$ X_y^{\triangleright} \lrar \D $$
which sends the cone point to $[X_y]\L(y)$ and sends the point $x \in X_y \subseteq X_y^{\triangleright}$ to $\L(y)$ (equipped with the map to $[X_y]\L(y)$ determined by $x \in X_y$) is a colimit diagram in $\D$. But this follows directly from our assumption that the action of $\cS^{m-1}_m$ on $\D$ is compatible with $\K_{m-1}$-indexed colimits, since the collection of maps $i_x: \ast \lrar X_y$ for $x \in X_y$ exhibit $X_y$ as the colimit in $\cS^{m-1}_m$ of the constant $X_y$-indexed diagram with value $\ast$ (by Proposition~\ref{p:adjoint}). 
\end{proof}

Let us now fix an $(m-1)$-truncated map $f: X \lrar Y$ between $m$-finite spaces. Consider the commutative diagram of spaces
$$ \xymatrix{
& X \ar[dr]\ar_{f}[dl] & \\
Y \ar[dr] && X \ar^{f}[dl] \\
& Y & \\
}$$
as a span in $(\cS_{m})_{/Y}$. Applying the straightening construction 
over $Y$ and using the assumption that $f$ is $(m-1)$-truncated we obtain a span of the form
$$ \xymatrix{
& \St_f \ar[dr]\ar[dl] & \\
\St_{\Id_Y} && \St_f \\
}$$
in the $\infty$-category $(\cS_{m-1})^Y$, which we can consider as a map
$$ \ovl{\St}_{\Id_Y} \lrar \ovl{\St}_{f} $$
of $\cS^{m-1}_m$-valued local systems on $Y$.
%
Applying Construction~\ref{c:st} and using Lemma~\ref{l:left-kan} we then obtain a natural transformation
\begin{equation}\label{e:f_hat}
[\hat{f}]_Y: \Id \simeq [\ovl{\St}_{\Id_Y}]  \Rightarrow [\ovl{\St}_f] \simeq f_!f^* 
\end{equation}
of functors $\D^Y \lrar \D^Y$. 

\begin{lem}\label{l:effect}
Let $\D$ be as in Hypothesis~\ref{a:fix}. Then for every $(m-1)$-truncated map $f: X \lrar Y$ of $m$-finite spaces the natural transformation $[\hat{f}]:[Y] \Rightarrow [X]$ associated to the span $\hat{f}: Y \x{f}{\llar} X \x{\Id}{\lrar} X$ of Definition~\ref{d:dual}
is homotopic to
$$ \xymatrix{
[Y] \simeq q_!q^* \ar@{=>}[rr]^-{q_![\hat{f}]_Yq^*} && q_!f_!f^*q^* \simeq [X] \\
}$$
where $q: Y \lrar \ast$ is the terminal map and $[\hat{f}]_Y$ is the natural transformation~\eqref{e:f_hat}. 
\end{lem}
\begin{proof}
Since $q_!$ is given by taking the colimit along $Y$ and the action map $\cS^{m-1}_{m} \lrar \Fun(\D,\D)$ preserves $\K_m$-indexed colimits it will suffice to show that the span $\hat{f} = [Y \x{f}{\llar} X \x{\Id}{\lrar} X]$ is the colimit in $\cS^{m-1}_m$ of $\St_{\Id} \llar \St_f \lrar \St_f$, considered as a $Y$-indexed family of morphisms in $\cS^{m-1}_m$. Now since $Y$ is $m$-finite the $\infty$-category $(\cS_m)^Y$ of $\cS_m$-valued local systems on $Y$ is equivalent by the straightnening-unstraightening construction to the slice $\infty$-category $(\cS_m)_{/Y}$, while the $\infty$-category $(\cS^{m-1}_m)^Y$ of $\cS^{m-1}_m$-valued local systems on $Y$ is equivalent in the same manner to the generalized span $\infty$-category of $(\cS_m)_{/Y}$ with respect to the weak coWaldhausen structure $(\cS_m)^{\dagger}_{/Y} \subseteq (\cS_m)_{/Y}$ consisting of the $(m-1)$-truncated maps in $(\cS_m)_{/Y}$. In addition, since $\K_m$-indexed colimits in $\cS^{m-1}_m$ are the same as the corresponding colimits in $\cS_m$ (Proposition~\ref{p:adjoint}), the colimit functor $(\cS^{m-1}_m)^Y \lrar \cS^{m-1}_m$ corresponds to the functor 
$$ \Span((\cS_m)_{/Y}, (\cS_m)^{\dagger}_{/Y}) \lrar \Span(\cS_m,\cS_{m,m-1}) $$
induced by the forgetful functor $(\cS_m)_{/Y} \lrar \cS_m$ (which respects the weak coWaldhausen structures on both sides). The desired claim is then simply a consequence of the fact that, since unstraightening is inverse to straightening, the span $\St_{\Id} \llar \St_f \lrar \St_f$ unstraightens to $Y \x{f}{\llar} X \x{\Id}{\lrar} X$. 
\end{proof}

\begin{lem}\label{l:lem-2}
Let $\D$ be as in Hypothesis~\ref{a:fix}. Then for every $(m-1)$-truncated map $f: X \lrar Y$ between $m$-finite spaces the natural transformation $[\hat{f}]_{Y}: \Id \Rightarrow f_!f^*$ of~\eqref{e:f_hat} 
exhibits $f_!$ as right adjoint to $f^*$. In other words, it exhibits $f$ as $\D$-ambidextrous.
\end{lem}
\begin{proof}
Let $\L_X: X \lrar \D$ and $\L_Y: Y \lrar \D$ be two local systems. We need to show that the composite map
$$ 
\xymatrix{
\alp: \Map_X(f^*\L_Y,\L_X) \ar[r] & \Map_Y(f_!f^*\L_Y,f_!\L_X) \ar[r]^-{(-) \circ [\hat{f}]_Y} & \Map_Y(\L_Y,f_!\L_X) 
}
$$
is an equivalence. By~\cite[Lemma 4.3.8]{ambi} it is enough prove this for objects of the form $\L_Y =(i_y)_!D$ where $i_y: \{y\} \lrar Y$ is the inclusion of some point $y \in Y$ and $D$ is an object of $\D$. Since $(i_y)_!$ is left adjoint to the restriction functor $i_y^*:\D^Y \lrar \D$ this is the same as showing that for every $y \in Y$ the composed natural transformation
\begin{equation}\label{e:natural}
\xymatrix{
\Id \ar@{=>}^-{u_y}[r] & i_y^*(i_y)_! \ar@{=>}^-{i_y^*[\hat{f}]_Y(i_y)_!}[rr] && i_y^*f_!f^*(i_y)_! \\
}
\end{equation}
%
exhibits $i_y^*f_!$ as right adjoint to $f^*(i_y)_!$, where $u_y$ is the unit of the adjunction $(i_y)_! \dashv i_y^*$. Now by the definition of the functor $[\ovl{\St}_f]$ (see Construction~\ref{c:st}) we have a natural equivalence 
\begin{equation}\label{e:Xy}
i_y^*[\ovl{\St}_f] \simeq [X_y]i_y^* 
\end{equation}
of functors $\D^Y \lrar \D$. Identifying $[\ovl{\St}_f]$ with $f_!f^*$ via Lemma~\ref{l:left-kan} we see that this is just an incarnation of the fact that left Kan extensions are determined pointwise. The latter fact is best phrased via the Beck-Chevalley transformation $\tau_y: (f_y)_!i_{X_y}^* \Longrightarrow i_y^*f_!$ associated to the Cartesian square
\begin{equation}\label{e:BC}
\xymatrix{
X_y \ar^{i_{X_y}}[r]\ar_{f_y}[d] & X \ar^{f}[d] \\
\{y\} \ar_{i_y}[r] & Y \\
}
\end{equation}
Indeed, by~\cite[Proposition 4.3.3]{ambi} the transformation $\tau_y$ is an equivalence, asserting, in effect, that the value of the left Kan extension $f_!\L$ at a given point $y$ is the colimit of $\L$ restricted to the homotopy fiber $X_y$. In this formalism the equivalence~\eqref{e:Xy} is obtained by composing the equivalences
$$ \xymatrix{
i_y^*f_!f^* \ar@{=>}[r]^-{\tau_y^{-1}f^*}_-{\simeq} & (f_y)_!i_{X_y}^*f^* \ar@{=>}[r]^-{\simeq} & (f_y)_!f_y^*i_y^* .\\
}$$
Now by the compatibility of the straightening-unstraightening equivalence with base change we see that under the equivalence~\eqref{e:Xy} the natural transformation $i_y^*[\hat{f}]_Y: i_y^* \Rightarrow i_y^*[\St_f]$ identifies with the natural transformation $[\hat{f}_y]i_y^*: i_y^* \Rightarrow [X_y]i_y^*$. In other words, we have a commuting square of the form
$$ \xymatrix{
i_y^*f_!f^*  & (f_y)_!i_{X_y}^*f^* \ar@{=>}[l]_-{\tau_yf^*} \\
i_y^* \ar@{=>}[u]^-{i_y^*[\hat{f}]_Y}\ar@{=>}[r]_-{[\hat{f}_y]i_y^*} & (f_y)_!f_y^*i_y^* \ar@{=>}[u]_-{\simeq} \\
}$$
in the $\infty$-category $\Fun(\D^Y,\D)$. Precomposing with the functor $(i_y)_!$ we obtain the top square in the diagram
$$ \xymatrix{
i_y^*f_!f^*(i_y)_!  && (f_y)_!i_{X_y}^*f^*(i_y)_! \ar@{=>}[ll]_-{\tau_yf^*(i_y)_!}^-{\simeq} \\
i_y^*(i_y)_! \ar@{=>}[u]^-{i_y^*[\hat{f}]_Y(i_y)_!}\ar@{=>}[rr]^-{[\hat{f}_y]i_y^*(i_y)_!} && (f_y)_!f_y^*i_y^*(i_y)_! \ar@{=>}[u]_-{\simeq} \\
\Id \ar@{=>}[u]^{u_y}\ar@{=>}[rr]^-{[\hat{f}_y]} && (f_y)_!(f_y)^*\ar@{=>}[u]_{(f_y)_!f_y^*u_y} \\
}
$$
in the $\infty$-category $\Fun(\D,\D)$, where $u_y: \Id \Rightarrow i_y^*(i_y)_!$ is the unit of the adjunction $(i_y)_! \dashv i_y^*$.
In particular, we can identify the composed transformation
\begin{equation}\label{e:natural-2}
\xymatrix{
\Id \ar@{=>}^-{u_y}[r] & i_y^*(i_y)_! \ar@{=>}^-{i_y^*[\hat{f}]_Y(i_y)_!}[rr] && i_y^*f_!f^*(i_y)_! \\
}
\end{equation}
with the composed transformation
\begin{equation}\label{e:natural-3}
\xymatrix{
\Id \ar@{=>}[r]^-{[\hat{f_y}]} & (f_y)_!f_y^* \ar@{=>}^-{(f_y)_!f_y^*u_y}[rr] && (f_y)_!f_y^*i_y^*(i_y)_! \ar@{=>}[r]^-{\simeq} & (f_y)_!i_{X_y}^*f^*(i_y)_! \ar@{=>}[rr]^-{\tau_yf^* (i_y)_!}_-{\simeq} && i_y^*f_!f^*(i_y)_! \\
}
\end{equation}
Now the transpose of the square~\eqref{e:BC} also has a Beck-Chevalley transformation $\sig_y: (i_{X_y})_!f_y^*  \Longrightarrow f^*(i_y)_!$, which (see~\cite[Remark 4.1.2]{ambi}) is given by the composition of transformations
$$
\xymatrix{
(i_{X_y})_!f_y^* \ar@{=>}[rr]^-{(i_{X_y})_!f_y^*u_y} && (i_{X_y})_!f_y^*i_y^*(i_y)_! \ar@{=>}[r]^-{\simeq} &  (i_{X_y})_!i_{X_y}^*f^*(i_y)_! \ar@{=>}[rr]^-{v_{X_y}f^*(i_y)_!} && f^*(i_y)_! \\
} 
$$
Applying~\cite[Proposition 4.3.3]{ambi} again we get that $\sig_y$ is an equivalence. Let $u_{X_y}: \Id \Longrightarrow i_{X_y}^*(i_{X_y})_!$ be a unit transformation compatible with the counit $v_{X_y}$ above. Then the compatibility of $u_{X_y}$ and $v_{X_y}$ implies that~\eqref{e:natural-3} (and hence~\eqref{e:natural-2}) is homotopic to the composed transformation
\begin{equation}\label{e:natural-5}
\xymatrix{
\Id \ar@{=>}[r]^-{[\hat{f_y}]} & (f_y)_!f_y^* \ar@{=>}[rr]^-{(f_y)_!u_{X_y}f_y^*} && (f_y)_!i_{X_y}^*(i_{X_y})_!f_y^*  \ar@{=>}[rr]^-{\tau_y\sig_y}_-{\simeq} &&  i_y^*f_!f^*(i_y)_! \\
}
\end{equation}
Comparing now~\eqref{e:natural-2} and~\eqref{e:natural-5} we have reduced to showing that
$$
\xymatrix{
\Id \ar@{=>}[r]^-{[\hat{f_y}]} & (f_y)_!f_y^* \ar@{=>}[rr]^-{(f_y)_!u_{X_y}f_y^*} && (f_y)_!i_{X_y}^*(i_{X_y})_!f_y^*  \\
}
$$
exhibits $(f_y)_!i_{X_y}^*$ as right adjoint to $(i_{X_y})_!f_y^*$. But this just follows from the fact that $u_{X_y}$ is the unit of $(i_{X_y})_! \dashv i_{X_y}^*$ by construction and $[\hat{f_y}]$ exhibits $(f_y)_!$ as right adjoint to $(f_y)^*$ by Lemma~\ref{l:lem-1}.
\end{proof}

\begin{const}\label{c:counit}
Let $X \in \cS_m$ be an $m$-finite space. We will denote by $\del: X \lrar X \times X$ the associated \textbf{diagonal map}. Let $\pi_1,\pi_2: X \times X \lrar X$ be the two projections and let $\sig_X: (\pi_1)_!\pi_2^* \Longrightarrow q^*q_!$ be the Beck-Chevalley transformation associated to the Cartesian square
\begin{equation}\label{e:BC-2}
\xymatrix{
X \times X \ar^{\pi_1}[r]\ar_{\pi_2}[d] & X \ar^{q}[d] \\
X \ar_{q}[r] & \ast \\
}
\end{equation}
By~\cite[Proposition 4.3.3]{ambi} the transformation $\sig_X$ is an equivalence. Let $[\hat{\del}]_{X \times X}: \Id \lrar (\del)_!\del^*$ be the natural transformation~\eqref{e:f_hat} for the map $\del$. We will then denote by 
$$ \xymatrix{
v_q: q^*q_! \ar@{=>}[r]^{\sig_X^{-1}}_{\simeq} & (\pi_1)_!\pi_2^* \ar@{=>}[rr]^-{[\hat{\del}]_{X \times X}} && (\pi_2)_!(\del)_!\del^*(\pi_1)^* \simeq \Id \circ \Id \simeq \Id \\
} $$
the composed natural transformation.
\end{const}

\begin{observe}\label{o:ambi}
Let $\D$ be as in Hypothesis~\ref{a:fix}. Let $X \in \cS_m$ be an $m$-finite spaces and $q: X \lrar \ast$ the terminal map. Then $X$ is $\D$-ambidextrous in the sense of~\cite[Definition 4.3.4]{ambi} if and only if the natural transformation $v_{q}: q^*q_! \Rightarrow \Id$ of Construction~\ref{c:counit} is a counit exhibiting $q^*$ as left adjoint to $q_!$.
\end{observe}
\begin{proof}
Since $\D$ is assumed to be $(m-1)$-semiadditive we have that $X$ is automatically weakly ambidextrous in the sense of~\cite{ambi}. We hence just need to identify the natural transformation $v_q$ with the natural transformation $v^{(m)}_q$ appearing in~\cite[Construction 4.1.8]{ambi}. Comparing the respective definitions we see that it will suffice to show that the natural transformation $[\hat{\del}]_{X \times X}: \Id \lrar (\del)_!\del^*$ is homotopic to the natural transformation $\mu^{(m-1)}_{\delta}:\Id \lrar (\del)_!\del^*$ appearing in~\cite[Construction 4.1.8]{ambi}. But this now follows from the uniqueness of units since both natural transformations exhibit $\del_!$ as right adjoint to $\del^*$ (Lemma~\ref{l:lem-2}).
\end{proof}


\begin{define}\label{d:diag}
Let $X$ be an $m$-finite space. We will denote by $\tr_X: X \times X \lrar \ast$ the morphism in $\cS^{m-1}_m$ given by the span
$$ \xymatrix{
& X\ar_{\del}[dl]\ar[dr] & \\
X \times X && \ast \\
}$$
where $\del: X \lrar X \times X$, as above, is the diagonal map.
\end{define}

\begin{rem}\label{r:self-dual}
Let $X \in \cS^{m-1}_m$ and $\tr_X: X \times X \lrar \ast$ be as in Definition~\ref{d:diag}. If we consider $X$ as an object of the larger (symmetric monoidal) $\infty$-category $\cS^m_m$, then $\tr_X$ exhibits $X$ as self-dual. To see this, observe that in $\cS^m_m$ the dual span $\hat{\tr}_X: \ast \llar X \x{\del}{\lrar} X \times X$ exists as well, and the pair $\tr_X$ and $\hat{\tr}_X$ form a compatible pair of evaluation and coevaluation maps exhibiting $X$ as self-dual. More precisely, the compositions
$$ \xymatrix{
X \ar[rr]^-{X \times \hat{\tr}_X} && X \times X \times X \ar[rr]^-{\tr_X \times X} && X \\
}
$$
and
$$ \xymatrix{
X \ar[rr]^-{\hat{\tr}_X \times X} && X \times X \times X \ar[rr]^-{X \times \tr_X} && X \\
}
$$
are both homotopic to the identity.
\end{rem}

Having an action of $\cS^{m-1}_m$ on $\D$ means in particular that for every $X \in \cS_m$ we are equipped with an equivalence $m_X:[X \times Y] \simeq [X] \circ [Y]$. Since the action of $\cS^{m-1}_m$ is compatible with $\K_m$-indexed colimits we have a canonical equivalence $[X] \simeq \colim_{x \in X} \Id$ and the map $m_X$ is completely determined by $m_\ast: \Id \simeq \Id \circ \Id$ via the canonical ``Fubini map'' 
$$ \colim_{(x,y) \in X \times Y} \Id \x{\simeq}{\Rightarrow} \colim_{x \in X}\colim_{y \in Y} \Id .$$ 
The latter can be described in terms of the Beck-Chevalley transformation $\sig_{X,Y}: (\pi_X)_!\pi_Y^* \Longrightarrow q_X^*(q_Y)_!$ associated to the Cartesian square
\begin{equation}\label{e:BC-3}
\xymatrix{
X \times Y \ar^{\pi_Y}[r]\ar_{\pi_X}[d] & Y \ar^{q_Y}[d] \\
X \ar_{q_X}[r] & \ast \\
}
\end{equation}
The Fubini map, and hence $m_{X,Y}$, can then be the identified with the map 
\begin{equation}\label{e:mx} 
\xymatrix{
m_{X,Y}:[X \times Y] = Q_!Q^* = (q_X)_!(\pi_X)_!\pi_Y^*q_Y^* \ar@{=>}[rr]^-{(q_X)_!\sig_{X,Y}q_Y^*} && (q_X)_!q_X^*(q_Y)_!q_Y^* = [X] \circ [Y] \\
},
\end{equation}
where $Q: X \times Y \lrar \ast$ is the terminal map.

\begin{prop}\label{p:non-deg}
Let $\D$ be as in Hypothesis~\ref{a:fix}. Then $\D$ is $m$-semiadditive if and only if the natural transformation
\begin{equation}\label{e:tr} 
\xymatrix{
[X] \circ [X] \ar@{=>}[r]^-{m_{X,X}^{-1}}_{\simeq} & [X \times X] \ar@{=>}[r]^-{[\tr_X]} & \Id \\
}
\end{equation}
exhibits the functor $[X]: \D \lrar \D$ as self-adjoint.
\end{prop}
\begin{proof}
By Lemma~\ref{l:effect} and~\eqref{e:mx} we may identify the composed natural transformation~\eqref{e:tr} with the composed natural transformation
$$ \xymatrix{
[X] \circ [X] \simeq q_!q^*q_!q^* \ar@{=>}[rr]^-{q_!v_qq^*} && q_!q^* \simeq [X] \ar@{=>}[r]^-{[q]} & [\ast] = \Id \\
} $$
where $v_q$ is defined as in Construction~\ref{c:counit}. Identifying $v_q$ with the natural transformation $v^{(m)}_q$ of~\cite[Construction 4.1.8]{ambi} as in the proof of Observation~\ref{o:ambi} and using 
Lemma~\ref{l:basic} to identify $[q]: q_!q^* \Rightarrow \Id$ as a counit exhibiting $q_!$ as left adjoint ot $q^*$ 
we may conclude that the natural transformation~\eqref{e:tr} is homotopic to the trace form associated to the map $q:X \lrar \ast$ by Hopkins--Lurie \cite[Notation 5.1.7]{ambi}. By~\cite[Proposition 5.1.8]{ambi} we the get that~\eqref{e:tr} exhibits $[X]$ as self-adjoint if and only if the natural transformation $v_q: q^*q_! \lrar \Id$ is a counit of an adjunction, and so if and only if $X$ is $\D$-ambidextrous (Observation~\ref{o:ambi}).
\end{proof}


\begin{cor}\label{c:limit}
Let $\D$ be as in Hypothesis~\ref{a:fix}. Then $\D$ is $m$-semiadditive if and only if the collection of natural transformations $[\hat{i}_x]: [X] \Rightarrow [\ast]$ exhibits $[X]$ as the limit, in $\Fun(\D,\D)$, of the constant diagram $X$-indexed diagram with value $[\ast]$ (here the morphism $\hat{i}_x: X \lrar \ast$ in $\cS^{m-1}_m$ is as in Definition~\ref{d:dual}).
\end{cor}
\begin{proof}
By Proposition~\ref{p:non-deg} it will suffice to show that the collection of natural transformations $[\hat{i}_x]: [X] \Rightarrow [\ast]$ exhibits $[X]$ as the limit in $\Fun(\D,\D)$ of the constant $X$-indexed diagram with value $[\ast]$ if and only if 
$[\tr_X] \circ m_{X,X}^{-1}:[X] \circ [X] \Rightarrow \Id$ 
exhibits $[X]$ as self-adjoint. Let $\G: \D \lrar \D$ be any other functor and let $\alp_\G$ be the composed map
$$ \xymatrix{
\alp_\G:\Map(\G,[X]) \ar[r] & \Map([X] \circ \G,[X] \circ [X]) \ar[rrr]^-{[\tr_X] \circ m_{X,X}^{-1} \circ (-)} &&& \Map([X] \circ \G,\Id)\\
} .$$
Recall that the collection of natural transformations $[i_x]: [\ast] \Rightarrow [X]$ exhibits $[X]$ as the colimit in $\Fun(\D,\D)$ of the constant $X$-indexed diagram with value $[\ast]$. 
Since colimits in functor categories are computed objectwise it follows that the natural transformations $[i_x] \circ \G: [\ast] \circ \G \Rightarrow [X] \circ \G$ exhibit $[X] \circ \G$ as the colimit in $\Fun(\D,\D)$ of the constant $X$-indexed diagram with value $[\ast] \circ \G \simeq \G$. We may hence identify a map $[X] \circ \G \Rightarrow \Id$ with a collection of natural transformations $\T_x:\G \Rightarrow \Id$ indexed by $x \in X$. Since the map $\hat{i}_x: X \lrar \ast$ is equivalent to the composition $X = X \times \ast \x{\Id \times i_x}{\lrar} X \times X \x{\tr_X}{\lrar} \ast$ we see that the map $\alp_\G$ associates to a natural transformation $\T:\G \lrar [X]$ the collection of natural transformations $[\hat{i}_x] \circ \T:\G \lrar [\ast] \simeq \Id$. It hence follows that the collection of natural transformations $[\hat{i}_x]: [X] \Rightarrow [\ast]$ exhibits $[X]$ as the limit in $\Fun(\D,\D)$ of the constant diagram $X$-indexed diagram with value $[\ast]$ if and only if $\alp_\G$ is an equivalence for every $\G$, i.e., if and only if $[\tr_X] \circ m_{X,X}^{-1}:[X] \circ [X] \Rightarrow \Id$ exhibits $[X]$ as self-adjoint.

\end{proof}

\begin{cor}\label{c:amusing}
Let $\D$ be an $\infty$-category which is tensored over $\cS^m_m$, such that the action functor $\cS^m_m \times \D \lrar \D$ preserves $\K_m$-indexed colimits separately in each variable. Then $\D$ is $m$-semiadditive.
\end{cor}
\begin{proof}
Let us prove that $\D$ is $m'$-semiadditive for every $-2 \leq m' \leq m$ by induction on $m'$. Since every $\infty$-category is $(-2)$-semiadditive we may start our induction at $m'=-2$. Now suppose that $\D$ is $m'$-semiadditive for some $-2 \leq m' \leq m$. As above let us denote by $[X]: \D \lrar \D$ the action of $X \in \cS^{m'-1}_{m'}$. By Proposition~\ref{p:non-deg} it will suffice to show that the morphism~\eqref{e:tr} 
exhibits the functor $[X]$ as self-adjoint. But this follows from the fact that the action of $\cS^{m'-1}_{m'}$ extends to an action of $\cS^{m'}_{m'}$, and the morphism $\tr_X: X \times X \lrar \ast$ exhibits $X$ as self-dual in the monoidal $\infty$-category $\cS^{m'}_{m'}$ (see Remark~\ref{r:self-dual}). 
\end{proof}

\begin{cor}\label{c:origin}
For every $-2 \leq m \leq n$ the $\infty$-category $\cS^m_n$ is $m$-semiadditive.
\end{cor}
\begin{proof}
Combine Corollary~\ref{c:algebra} and Corollary~\ref{c:amusing}.
\end{proof}

\section{The universal property of finite spans}\label{s:proof}

In this section we will prove our main result, establishing a universal property for the $\infty$-categories $\cS^m_n$ in terms of $m$-semiadditivity.
\begin{thm}\label{t:main}
Let $-2 \leq m \leq n$ be integers and let $\D$ be an $m$-semiadditive $\infty$-category which admits $\K_n$-indexed colimits. Then evaluation at $\ast \in \cS^m_n$ induces an equivalence of $\infty$-categories
$$ \Fun_{\K_n}(\cS^m_n,\D) \x{\simeq}{\lrar} \D .$$
In other words, the $\infty$-category $\cS^m_n$ is the free $m$-semiadditive $\infty$-category which admits $\K_n$-indexed colimits, generated by $\ast \in \cS^m_n$.
\end{thm}

Our strategy is essentially a double induction on $n$ and $m$. For this it will be useful to employ the following terminology: 
\begin{define}\label{d:good}
Let $\D$ be an $\infty$-category and let $-2 \leq m \leq n$ be integers. We will say that $\D$ is \textbf{$(n,m)$-good} if the following conditions are satisfied:
\begin{enumerate}
\item
$\D$ is $m$-semiadditive and admits $\K_n$-indexed colimits.
\item
Evaluation at $\ast$ induces an equivalence of $\infty$-categories
$$ \Fun_{\K_n}(\cS^{m}_{n},\D) \x{\simeq}{\lrar} \D .$$
\end{enumerate}
\end{define}
In other words, $\D$ is $(n,m)$-good if Theorem~\ref{t:main} holds for $m,n$ and $\D$. We may hence phrase the induction step on $n$ as follows: given an $(n-1,m)$-good $\infty$-category $\D$ which admits $\K_{n}$-indexed colimits, show that $\D$ is $(n,m)$-good. To establish this claim we will need to understand how to extend functors from $\cS^m_{n-1}$ to $\cS^m_n$ when $m \leq n-1$. Note that if $f: Z \lrar X$ is an $m$-truncated map and $X$ is $(n-1)$-truncated then $Z$ is $(n-1)$-truncated as well, and so the inclusion $\cS^m_{n-1} \hrar \cS^m_n$ is fully-faithful. The core argument for the induction step on $n$ is the following:
\begin{prop}\label{p:step-1}
Let $-2 \leq m < n$ be integers. Let $\D$ be an $m$-semiadditive $\infty$-category which admits $\K_n$-indexed colimits and let $\F: \cS^{m}_{n-1} \lrar \D$ be a functor which preserves $\K_{n-1}$-indexed colimits. Let $\iota: \cS^{m}_{n-1} \hrar \cS^{m}_{n}$ be the fully-faithful inclusion. Then the following assertion hold:
\begin{enumerate}
\item
$\F$ admits a left Kan extension 
$$ \xymatrix{
\cS^{m}_{n-1}  \ar^-{\F}[r]\ar_{\iota}[d] & \D \\
\cS^{m}_{n} \ar@{-->}_{\ovl{\F}}[ur] & \\
}$$
\item
An arbitrary extension $\ovl{\F}: \cS^{m}_{n} \lrar \D$ of $\F$ is a left Kan extension if and only if $\ovl{\F}$ preserves $\K_n$-indexed colimits.
\end{enumerate}
\end{prop}
\begin{proof}
For $Y \in \cS^{m}_{n}$ let us denote by
$$ \I_Y := \cS^{m}_{n-1} \times_{\cS^{m}_{n}} (\cS^{m}_{n})_{/Y} $$
the associated comma $\infty$-category. To prove (1), it will suffice by~\cite[Lemma 4.3.2.13]{higher-topos} to show that the composed map 
$$ \F_{Y}:\I_{Y} \lrar \cS^{m}_{n-1} \lrar \D $$ 
can be extended to a colimit diagram in $\D$ for every $Y\in \cS^{m}_{n}$. Now an object of $\I_{Y}$ corresponds to an object $X \in \cS^{m}_{n-1}$ together with a morphism $X \lrar Y$ in $\cS^{m}_{n}$, i.e., a span of the form
\begin{equation}\label{e:object}
\xymatrix{
& Z \ar^{f}[dr]\ar_{g}[dl] & \\
X && Y \\
}
\end{equation}
where $g$ is $m$-truncated (and hence $Z$ is $(n-1)$-finite). Since $\cS^m_{n-1} \hrar \cS^m_n$ is fully-faithful the mapping space from $(X,Z,f,g)$ to $(X',Z',f',g')$ in $\I_Y$ can be identified with the homotopy fiber of the map $\Map_{\cS^m_{n}}(X,X') \lrar \Map_{\cS^m_{n}}(X,Y)$ over $(Z,f,g) \in \Map_{\cS^m_n}(X,Y)$. Now let 
$$ \J_Y := \cS_{n-1} \times_{\cS_{n}} (\cS_{n})_{/Y} $$
be the analogue comma $\infty$-category for the inclusion $\cS_{n-1} \hrar \cS_n$. Then the inclusions $\cS_{n-1} \hrar \cS^m_{n-1}$ and $\cS_n \hrar \cS^m_n$ induce a functor $\rho:\J_Y \lrar \I_Y$, and it is not hard to check that $\rho$ is in fact fully-faithful, and its essential image consists of those objects as in~\eqref{e:object} for which $g$ is an equivalence. We now claim that $\rho$ is also \textbf{cofinal}. To prove this, we need to show that for every object $(X,Z,f,g) \in \I_Y$ as in~\eqref{e:object}, the comma $\infty$-category $(\J_Y)_{(X,Z,f,g)/} := \J_Y \times_{\I_Y} (\I_Y)_{(X,Z,f,g)/}$ is weakly contractible. Given an object $h: X' \lrar Y$ in $\J_Y$, we may identify the mapping space from $(X,Z,f,g)$ to $\rho(X',h)$ in $\I_Y$ with the homotopy fiber of the map 
\begin{equation}\label{e:fiber}
h_*:\Map_{\cS^m_{n}}(X,X') \lrar \Map_{\cS^m_{n}}(X,Y)
\end{equation} 
over the span $(Z,f,g) \in \Map_{\cS^m_{n}}(X,Y)$. Now clearly any span of $n$-finite spaces from $X$ to $X'$ whose composition with $h: X' \lrar Y$ belongs to $\cS^m_n$ already itself belongs to $\cS^m_n$. We may hence identify the homotopy fiber of~\eqref{e:fiber} with the homotopy fiber of the map
\begin{equation}\label{e:fiber-2}
h_*:((\cS_n)_{/X \times X'})^{\simeq} \lrar ((\cS_n)_{/X \times Y})^{\simeq}
\end{equation} 
over the object $(f,g):Z \lrar X \times Y$. Finally, using the general equivalence $\C_{/A \times B} \simeq \C_{/A} \times_{\C} \C_{/B}$ we may identify the homotopy fiber of~\eqref{e:fiber-2} with the homotopy fiber of the map
\begin{equation}\label{e:fiber-3}
((\cS_n)_{/X'})^{\simeq} \lrar ((\cS_n)_{/Y})^{\simeq}
\end{equation} 
over $f: Z \lrar Y$. We may then conclude that the functor from $\J_Y$ to spaces given by $(X',h) \mapsto \Map_{\I_Y}((X,Z,f,g),\rho(X',h))$ is corepresented by $f: Z \lrar Y$ (considered as an object of $\J_Y$). This implies that the comma $\infty$-category $(\J_Y)_{(X,Z,f,g)/}$ 
has an initial object and is hence weakly contractible. Since this is true for any $(X,Z,f,g) \in \I_Y$ it follows that $\rho$ is cofinal, as desired.

It will now suffice to show that each of the diagrams 
$$ \F_{Y}|_{\J_{Y}}:\J_{Y} \lrar \D $$ 
can be extended to a colimit diagram. 
Let $\J'_{Y} = \J_Y \times_{\cS_{n-1}} \{\ast\} \subseteq \J_{Y}$ be the full subcategory spanned by objects of the form $\ast \x{h}{\lrar} Y$. Then $\J'_{Y}$ is an $\infty$-groupoid which is equivalent to the underlying space of $Y$, and the composed functor $\J'_{Y} \lrar \J_{Y} \lrar \D$ is constant with value $\F(\ast) \in \D$. Since we assumed that $\F: \cS^m_{n-1} \lrar \D$ preserves $\K_m$-indexed colimits it follows from Proposition~\ref{p:adjoint} that the restriction $\F|_{\cS_{n-1}}: \cS_{n-1} \lrar \D$ preserves $\K_{n-1}$-indexed colimits and hence by Lemma~\ref{l:colimits} the functor $\F|_{\cS_{n-1}}$ is a left Kan extension of its restriction to the object $\ast \in \cS_{n-1}$. Now since the projection $\J_Y \lrar \cS_{n-1}$ is a right fibration (classified by the functor $X \mapsto \Map_{\cS_n}(X,Y)$) it induces an equivalence $(\J_Y)_{/(X,h)} \lrar (\cS_{n-1})_{/X}$ for every $(X,h) \in \J_Y$. We may then conclude that $\F|_{\J_Y}$ is a left Kan extension of $\F|_{\J_Y'}$. Since $\D$ admits $\K_n$-indexed colimits (and $\J'_Y$ is an $n$-finite Kan complex) the diagram $\F_{Y}|_{\J'_{Y}}$ admits a colimit. It then follows that the diagram $\F_{Y}|_{J_Y}: \J_{Y} \lrar \D$ admits a colimit, as desired. 
 
To prove (2), note that by the above considerations an arbitrary functor $\ovl{\F}: \cS^m_n \lrar \D$ is a left Kan extension of $\F$ if and only if the extension $(\J'_{Y})^{\triangleright} \lrar \D$ determined by $\ovl{\F}$ is a colimit diagram. By construction, this means that $\ovl{\F}$ is a left Kan extension if and only if for every $Y \in \cS^m_n$ the collection of maps $\{\ovl{\F}(i_y):\ovl{\F}(\ast) \lrar \ovl{\F}(Y)\}_{y \in Y}$ exhibits $\ovl{\F}(Y)$ as the colimit of the constant $Y$-indexed diagram with value $\ovl{\F}(\ast)$. It then follows from Lemma~\ref{l:colimits} that $\ovl{\F}$ is a left Kan extension of $\F$ if and only if $\ovl{\F}$ preserves $\K_n$-indexed colimits.

\end{proof}

\begin{cor}\label{c:step-1}
Let $-2 \leq m < n$ be integers and let $\D$ be an $\infty$-category which admits $\K_n$-indexed colimits. Then the restriction map
$$ \Fun_{\K_n}(\cS^{m}_{n},\D) \lrar \Fun_{\K_{n-1}}(\cS^{m}_{n-1},\D) $$
is an equivalence of $\infty$-categories.
\end{cor}
\begin{proof}
This a direct consequence of Proposition~\ref{p:step-1} in light of \cite[Proposition 4.3.2.15]{higher-topos}.
\end{proof}

\begin{cor}\label{c:step-1a}
Let $-2 \leq m \leq n \leq n'$ be integers and let $\D$ be an $m$-semiadditive $\infty$-category which admits $\K_{n'}$-indexed colimits. Then $\D$ is $(n',m)$-good if and only if $\D$ is $(n,m)$-good.
\end{cor}

We shall now proceed to perform the induction step on $m$. We begin with the following lemma.

\begin{lem}\label{l:export}
Let $\D$ be an $(m,m-1)$-good $\infty$-category and let $\F: \cS^{m-1}_m \lrar \D$ be a functor which preserves $\K_m$-indexed colimits. If $\D$ is $m$-semiadditive then the collection of maps $\F(\hat{i}_x): \F(X) \lrar \F(\ast)$ for $x \in X$ exhibits $\F(X)$ as the limit in $\D$ of the constant $X$-indexed diagram with value $\F(\ast)$.
\end{lem}
\begin{proof}
Using the symmetric monoidal structure of $\cS^{m-1}_{m}$ we may consider $\cS^{m-1}_{m}$ as tensored over itself. Since the monoidal structure preserves $\K_m$-indexed colimits separately in each variable (see Corollary~\ref{c:n-tensor}), and since $\D$ is $(m,m-1)$-good, we may endow $\Fun_{\K_m}(\cS^{m-1}_{m},\D) \simeq \D$ with an action of $\cS^{m-1}_m$ via precomposition. Then $\D$ is tensored over $\cS^{m-1}_{m}$ and the action map $\cS^{m-1}_{m} \times \D \lrar \D$ preserves $\K_m$-indexed colimits separately in each variable. As above let us denote by $[X]: \D \lrar \D$ the action of $X \in \cS^{m-1}_m$. 

By Corollary~\ref{c:limit} the collection of natural transformations $[\hat{i}_x]: [X] \Rightarrow [\ast]$ exhibits $[X]$ as the limit, in $\Fun(\D,\D)$, of the constant $X$-indexed diagram with value $[\ast]$. 
Evaluating at $\F(\ast)$ we may conclude that the collection of maps $[\hat{i}_x](\F(\ast)):  [X](\F(\ast)) \lrar \F(\ast)$ exhibits $[X](\F(\ast))$ as the limit in $\D$ of the constant $X$-diagram with value $\F(\ast)$. By construction we may identify $[X](\F(\ast))$ with $\F(X)$ and $[\hat{i}_x](\F(\ast))$ with $\F(\hat{i}_x)$ and so the desired result follows.
\end{proof}

We next proceed to the establish the inductive step. As in the proof of \ref{p:step-1} we will use a Kan extension argument (though this time it will be a right Kan extension). Since the subcategory inclusion $\cS^{m-1}_m \subseteq \cS^m_m$ is not fully-faithful this requires a slightly more elaborate setup for which it will be convenient to make use of the language of \textbf{marked simplicial sets}, as developed by Lurie \cite{higher-topos}. Given an $m \geq -1$ let 
$$ \Cone_m = \cS^{m}_{m} \coprod_{\cS^{m-1}_{m}\times\Del^{\{0\}}}\left[\cS^{m-1}_{m}\times(\Del^1)^{\sharp}\right] $$ 
be the right marked mapping cone of the inclusion $\iota: \cS^{m-1}_{m} \hrar \cS^{m}_{m}$. Let
$$ \Cone_m \hrar \M^{\natural} \x{r}{\lrar} \Del^1 $$
be a factorization of the projection $\Cone_m \lrar (\Del^1)^{\sharp}$ into a trivial cofibration follows by a fibration in the Cartesian model structure over $(\Del^1)^{\sharp}$. In particular, $r: \M \lrar \Del^1$ is a Cartesian fibration and the marked edges of $\M^{\natural}$ are exactly the $r$-Cartesian edges. Let $\iota_0: \cS^{m}_{m} \hrar \M \times_{\Del^1} \Del^{\{0\}} \subseteq \M$ and $\iota_1: \cS^{m-1}_{m} \hrar \M \times_{\Del^1} \Del^{\{1\}} \subseteq \M$ be the corresponding inclusions. Then $\iota_0$ and $\iota_1$ exhibit $r:\M \lrar \Del^1$ as a correspondence from $\cS^{m}_{m}$ to $\cS^{m-1}_{m}$ which is the one associated to the functor $\iota:\cS^{m-1}_{m} \hrar \cS^{m}_{m}$. 


\begin{prop}\label{p:rkan}
Let $\D$ be an $\infty$-category which admits $\K_m$-indexed limits and let $\F: \cS^{m-1}_{m} \lrar \D$ be a functor which satisfies the following property: for every $X \in \cS^{m-1}_m$ the collection of maps $\F(\hat{i}_x): \F(X) \lrar \F(\ast)$ exhibits $\F(X)$ as the limit in $\D$ of the constant $X$-indexed diagram with value $\F(\ast)$. Then the following assertion hold:
\begin{enumerate}
\item
There exists a right Kan extension
\begin{equation}\label{e:rkan}
\xymatrix{
\cS^{m-1}_{m}  \ar^-{\F}[r]\ar_{\iota_1}[d] & \D \\
\M \ar@{-->}^{\ovl{\F}}[ur] & \\
}
\end{equation}
\item
An extension $\ovl{\F}: \M \lrar \D$ as above is a right Kan extension if and only if $\ovl{\F}$ maps $r$-Cartesian edges to equivalences in $\D$.
\end{enumerate}
\end{prop}

\begin{rem}
An extension $\ovl{\F}$ as in~\eqref{e:rkan} is equivalent to the data of a functor $\F': \cS^m_m \lrar \D$ together with a natural transformation $\tau: \F'\circ \iota \Rightarrow \F$ (where $\iota: \cS^{m-1}_m \lrar \cS^m_m$ is the inclusion as above), and $\ovl{\F}$ is a right Kan extension if and only if $\tau$ exhibits $\F'$ as a right Kan extension of $\F$ along $\iota$. We should hence morally consider Proposition~\ref{p:rkan} as pertaining to right Kan extensions of $\F$ along $\iota: \cS^{m-1}_m \lrar \cS^m_m$. In particular, we could have worked directly with $\cS^m_m$ instead of $\M$ at the expense of replacing~\eqref{e:rkan} with a diagram which commutes up to a prescribed natural transformation. We also note that this issue did not arise in Proposition~\ref{p:step-1} since, unlike $\iota$, the map $\cS^{m}_{n-1} \lrar \cS^m_n$ appearing in Proposition~\ref{p:step-1} is fully-faithful (cf.~\cite[\S 4.3.2,4.3.3]{higher-topos}).
\end{rem}

\begin{proof}[Proof of proposition~\ref{p:rkan}]
For an object $X \in \cS^{m}_{m}$ let us set 
$$ \I_{X} = \M_{\iota_0(X)/} \times_{\M} \cS^{m-1}_{m} .$$
To prove (1), it will suffice by~\cite[Lemma 4.3.2.13]{higher-topos} to show that the composed map 
$$ \F_{X}:\I_{X} \lrar \cS^{m-1}_{m} \lrar \D $$ 
can be extended to a limit diagram in $\D$ for every $X\in \cS^{m}_{m}$. Now an object of $\I_{X}$ corresponds to an object $Y \in \cS^{m-1}_{m}$ and a morphism $\iota_0(X) \lrar \iota_1(Y)$ in $\M$, or, equivalently, a morphism $X \lrar \iota(Y)$ in $\cS^{m}_{m}$, i.e., a span
\begin{equation}\label{e:obj}
\xymatrix{
& Z \ar^{f}[dr]\ar_{g}[dl] & \\
X && Y \\
}
\end{equation}
of $m$-finite spaces. 

Recall that we have denoted by $\cS_{m,m-1} \subseteq \cS_m$ the subcategory of $\cS_m$ consisting of all objects and $(m-1)$-truncated maps between them. Then we have a commutative square
\begin{equation}\label{e:op}
\xymatrix{
\cS^{\op}_{m,m-1} \ar[r]\ar[d] & \cS_m^{\op} \ar[d] \\
\cS^{m-1}_m \ar[r] & \cS^m_m \\
}
\end{equation}
where the vertical functors map are the identity on objects and send an $(m-1)$-truncated map $f: X \lrar Y$ to the span $Y \x{f}{\llar} X \lrar X$. Let $\J_X = \cS_{m,m-1}^{\op} \times_{\cS_m^{\op}} (\cS_m^{\op})_{X/}$ be the associated comma $\infty$-category. We will write the objects of $\J_X$ as maps $X \x{g}{\llar} Y$ of $m$-finite spaces, or simply as pairs $(Y,g)$. We note that morphisms from $X \x{g}{\llar} Y$ to $X \x{g'}{\llar} Y'$ are commutative triangles of the form
$$ \xymatrix{
Y \ar^{g}[dr] && Y' \ar_{g'}[dl]\ar_{h}[ll] \\
& X & \\
}$$
such that $h$ is $(m-1)$-truncated. The square~\eqref{e:op} induces a fully-faithful fuctor $\rho:\J_{X} \hrar \I_{X}$ whose essential image consists of those objects as in~\eqref{e:obj} for which $f$ is an equivalence. We now claim that $\rho$ is coinitial. 

To prove this, we need to show that for every object $(Y,Z,f,g) \in \I_Y$ as in~\eqref{e:obj}, the comma $\infty$-category $(\J_X)_{/(Y,Z,f,g)} := \J_X \times_{\I_X} (\I_X)_{/(Y,Z,f,g)}$ is weakly contractible. Given an object $(Y',h)$ in $\J_X$, we may identify the mapping space from $\rho(Y',h)$ to $(Y,Z,f,g)$ in $\I_X$ with the homotopy fiber of the map 
\begin{equation}\label{e:fib}
h_*:\Map_{\cS^{m-1}_m}(Y',Y) \lrar \Map_{\cS^m_m}(X,Y)
\end{equation} 
over the span $(Z,f,g) \in \Map_{\cS^m_m}(X,Y)$. As in the proof Proposition~\ref{p:step-1} we may identify these mapping spaces as $\Map_{\cS^{m-1}_m}(Y',Y) \simeq ((\cS^{\op}_{m,m-1})_{Y'/})^{\simeq} \times_{\cS_m} ((\cS_{m})_{/Y})^{\simeq}$ and $\Map_{\cS^m_m}(X,Y) \simeq ((\cS^{\op}_{m})_{X/})^{\simeq} \times_{\cS_m} ((\cS_{m})_{/Y})^{\simeq}$ we may identify the homotopy fiber of~\eqref{e:fib} with the homotopy fiber of the map
\begin{equation}\label{e:fib-2}
h_*:((\cS^{\op}_{m,m-1})_{Y'/})^{\simeq} \lrar ((\cS^{\op}_{m})_{X/})^{\simeq}
\end{equation} 
over the object $X \x{g}{\llar} Z$. We may then conclude that the functor from $\J_X$ to spaces given by $(Y',h) \mapsto \Map_{\I_Y}(\rho(Y',h),(Y,Z,f,g))$ is represented in $\J_X$ by the object $X \x{g}{\llar} Z$. This implies that the comma $\infty$-category $(\J_X)_{/(Y,Z,f,g)}$
has an terminal object and is hence weakly contractible. Since this is true for any $(Y,Z,f,g) \in \I_X$ it follows that $\rho$ is coinitial, as desired.
It will hence suffice to show that each of the diagrams 
$$ \F_{X}|_{\J_{X}}:\J_{X} \lrar \D $$ 
can be extended to a limit diagram. 

Let $\J'_{X} = \J_X \times_{\cS^{\op}_{m,m-1}} \{\ast\} \subseteq \J_{X}$ be the full subcategory spanned by objects of the form $X \x{g}{\llar} \ast$. Then $\J'_{X}$ is an $\infty$-groupoid which is equivalent to the underlying space of $X$, and the composed functor $\J'_{X} \lrar \J_{X} \lrar \D$ is constant with value $\F(\ast) \in \D$. By our assumption on $\F$ it follows that the restricted functor $\F|_{\cS^{\op}_{m,m-1}}: \cS_{m,m-1}^{\op} \lrar \D$ is a right Kan extension $\F|_{\{\ast\}}$. Now since the projection $\J_X \lrar \cS^{\op}_{m,m-1}$ is a left fibration it induces an equivalence $(\J_X)_{(Y,h)/} \lrar (\cS^{\op}_{m,m-1})_{Y/}$ for every $(Y,h) \in \J_X$. We may then conclude that $\F|_{\J_X}$ is a right Kan extension of $\F|_{\J_X'}$. Since $\D$ is $m$-semiadditive it admits $\K_m$-indexed limits and hence the diagram $\F_{X}|_{\J'_{X}}$ admits a limit. It follows that the diagram $\F_{X}|_{\J_{X}}: \J_{X} \lrar \D$ admits a limit, as desired.

Let us now prove (2). Let $\ovl{\F}: \M \lrar \D$ be a map extending $\F$. Then for every $X \in \cS^{m}_{m}$ the functor $\ovl{\F}$ determines a diagram
$$ \ovl{\F}_{X}: \J_{X}^{\triangleleft} \lrar \D $$
extending $\F_{X}|_{\J_{X}}$. By the considerations above $\ovl{\F}$ is a right Kan extension of $\F$ if and only if each $\ovl{\F}_X$ is a limit diagram. Let $\J_{X}'' \subseteq \J_{X}$ denote the full subcategory spanned by those objects $X \x{g}{\llar} Y$ such that $g$ is $(m-1)$-truncated. By the above arguments the functor $\F_{X}|_{\J''_{X}}$ is a right Kan extension of $\F_{X}|_{\J'_{X}}$, and so by~\cite[Proposition 4.3.2.8]{higher-topos} we have that $\F_{X}$ is a right Kan extension of $\F_{X}|_{\J''_{X}}$. It follows that $\ovl{\F}_{X}$ is a limit diagram if and only  if $\ovl{\F}_{X}|_{(\J''_{X})^{\triangleleft}}$ is a limit diagram. Let $\diamond \in (\J''_{X})^{\triangleleft}$ be the cone point. We now observe that the $\infty$-category $\J''_{X}$ has initial objects, namely every object of the form $X \x{g}{\llar} X'$ such that $g$ is an equivalence. It follows that $\ovl{\F}_{X}|_{(\J''_{X})^{\triangleleft}}$ is a limit diagram if and only if $\ovl{\F}_{X}$ sends every edge connecting $\diamond$ with an initial object of $\J''_{X}$ to an equivalence in $\D$. To finish the proof it suffices to observe that these edges are exactly the edges which map to $r$-Cartesian edges by the natural map $(\J''_{X})^{\triangleleft} \lrar \M$, and that all $r$-Cartesian edges are obtained in this way.

\end{proof}

\begin{cor}\label{c:step-2}
Let $\D$ be an $m$-semiadditive $\infty$-category. Then the restriction map
$$ \Fun_{\K_m}(\cS^{m}_{m},\D) \lrar \Fun_{\K_m}(\cS^{m-1}_{m},\D) $$
is an equivalence of $\infty$-categories.
\end{cor}
\begin{proof}
Let $r: \M^{\natural} \lrar \Del^1$ be as above and consider the marked simplicial set $\D^{\natural} = (\D,M)$ where $M$ is the collection of edges which are equivalences in $\D$. For two marked simplicial sets $(X,M),(Y,N)$ let $\Fun^{\flat}((X,M),(Y,N)) \subseteq \Fun(X,Y)$ be the full sub-simplcial set spanned by those functors $X \lrar Y$ which send $M$ to $N$. 
We will denote by $\Fun^{\flat}_{\K_m}(\M^{\natural},\D^{\natural}) \subseteq \Fun^{\flat}(\M^{\natural},\D^{\natural})$ and by $\Fun^{\flat}_{\K_m}(\Cone_m,\D^{\natural}) \subseteq \Fun^{\flat}(\Cone_m,\D^{\natural})$ the respective full subcategories spanned by those makred functors whose restriction to $\cS^{m-1}_{m}$ preserves $\K_m$-indexed colimits. 
Now consider the commutative diagram of functors categories and restriction maps
\begin{equation}\label{e:res}
\xymatrix{
& \Fun^{\flat}_{\K_m}(\M^{\natural},\D^{\natural}) \ar[dr]\ar[dl] & \\
\Fun^{\flat}_{\K_m}(\Cone_m,\D^{\natural}) \ar^{\iota_1^*}[rr] && \Fun_{\K_m}(\cS^{m-1}_{m},\D) \\
}
\end{equation}
Since the inclusion of marked simplicial sets $\Cone_m \lrar \M^{\natural}$ is a trivial cofibration in the Cartesian model structure over $(\Del^1)^{\sharp}$ it follows that the left diagonal map is a trivial Kan fibration. On the other hand by Proposition~\ref{p:rkan} and~\cite[Proposition 4.3.2.15]{higher-topos} the right diagonal map is a trivial Kan fibration. We may hence deduce that $\iota_1^*$ is an equivalence of $\infty$-categories.

Since the inclusion $\cS^{m}_{m} \hrar \Cone_m$ is a pushout along the inclusion $\cS^{m-1}_{m} \times \Del^{\{0\}} \hrar \cS^{m-1}_{m} \times (\Del^1)^{\sharp}$ (which is itself a trivial cofibration in the \textbf{coCartesian} model structure over $\Del^0$) it follows that the map $i_0^*:\Fun^{\flat}(\Cone_m,\D^{\natural}) \lrar \Fun(\cS^{m}_{m},\D)$ is a trivial Kan fibation and that the composed functor 
$$ \xymatrix{
\Fun^{\flat}(\Cone_m,\D^{\natural})  \ar[r]^-{i_0^*}_-{\simeq} &  \Fun(\cS^{m}_{m},\D) \ar[r]^-{\iota^*} & \Fun(\cS^{m-1}_{m},\D) \\
}$$
is homotopic to $i_1^*:\Fun^{\flat}(\Cone_m,\D^{\natural}) \x{\simeq}{\lrar} \Fun(\cS^{m-1}_{m},\D)$. We may consequently conclude that $i_0^*$ induces an equivalence between $\Fun^{\flat}_{\K_m}(\Cone_m,\D^{\natural}) \subseteq \Fun^{\flat}(\Cone_m,\D^{\natural})$ and the full subcategory of $\Fun(\cS^m_m,\D)$ spanned by those functors whose restriction to $\cS^{m-1}_m$ preserves $\K_m$-indexed colimits. By Corollary~\ref{c:n-colimits} these are exactly the functors $\cS^m_m \lrar \D$ which preserves $\K_m$-indexed colimits. We may finally conclude that 
$$ \iota^*: \Fun_{\K_m}(\cS^m_m,\D) \lrar \Fun_{\K_m}(\cS^{m-1}_m,\D)$$ 
is an equivalence of $\infty$-categories, as desired.
\end{proof}

\begin{cor}\label{c:step-2a}
Let $-1 \leq m \leq n$ integer and let $\D$ be an $m$-semiadditive $\infty$-category which admits $\K_n$-indexed colimits. If $\D$ is $(n,m-1)$-good then $\D$ is $(n,m)$-good.
\end{cor}
\begin{proof}
By Corollary~\ref{c:step-1a} we know that $\D$ is $(n,m)$-good if and only if $\D$ is $(m,m)$-good, and that $\D$ is $(n,m-1)$-good if and only if $\D$ is $(m,m-1)$-good. The desired result now follows directly from Corollary~\ref{c:step-2a}.
\end{proof}

\begin{proof}[Proof of Theorem~\ref{t:main}]
We want to prove that if $\D$ is an $m$-semiadditive $\infty$-category which admits $\K_n$-indexed colimits then $\D$ is $(n,m)$-good. Let us consider the set $\A = \{(a,b) \in \ZZ \times \ZZ | a \leq b\}$ as partially ordered saying that $(a,b) \leq (c,d)$ if $a \leq c$ and $b \leq d$. We now note that for every $(-2,-2) \leq (n',m') \leq (n,m)$ in $\A$, the $\infty$-category $\D$ is $m'$-semiadditive and admits $\K_{n'}$-indexed colimits. Furthermore, $\D$ is tautologically $(-2,-2)$-good. It follows that there exists a pair $(-2,-2) \leq (n',m') \leq (n,m)$ for which $\D$ is $(n',m')$-good and which is maximal with respect to this property. If $n' < n$ then Corollary~\ref{c:step-1a} implies that $\D$ is $(n'+1,m')$-good, contradicting the maximality of $(n',m')$. On the other hand, if $n'=n$ and $m' < m$ then $m' < n'$. By Corollary~\ref{c:step-2a} we may conclude that $\D$ is $(n',m'+1)$-good, a contradiction again. It follows that $(n',m') = (n,m)$ and hence $\D$ is $(n,m)$-good, as desired.
\end{proof}

\section{Applications}

\subsection{$m$-semiadditive $\infty$-categories as modules over spans}\label{s:modules}


By Corollary~\ref{c:amusing}, every $\infty$-category with $\K_m$-indexed colimits which carries a compatible action of $\cS^m_m$ is $m$-semiadditive. On the other hand, our main theorem~\ref{t:main} implies that every $m$-semiadditive $\infty$-category $\D$ carries an action of $\cS^m_m$ (by identifying $\D$, for example, with $\Fun_{\K_m}(\cS^m_m,\D)$). This suggests that the theory of $m$-semiadditive $\infty$-categories is strongly related to that of $\cS^m_m$-modules in $\Cat_{\K_m}$. In this section we will make this idea more precise by proving a suitable equivalence of $\infty$-categories. This equivalence is strongly related to the fact that $\cS^m_m$ is an \textbf{idempotent object} of $\Cat_{\K_m}$, a corollary we will deduce below.

Let $\Add_m \subseteq \Cat_{\K_m}$ denote the full subcategory spanned by $m$-semiadditive $\infty$-categories. The discussion above implies that the essential image of the forgetful functor
$$ \U:\Mod_{\cS^m_m}(\Cat_{\K_m}) \lrar \Cat_{\K_m} $$
is exactly $\Add_m$. 
We note that $\U$ admits a left adjoint $\F: \Cat_{\K_m} \lrar \Mod_{\cS^m_m}(\Cat_{\K_m})$ given by $\F(\D) = \cS^m_m \otimes_{\K_m} \D$. 

\begin{lem}\label{l:key-1}
Let $\C$ be an $\infty$-category. Then the unit map
$$ u_\C: \C \lrar \cS^m_m \otimes_{\K_m} \C $$
associated to $\F \dashv \U$ is an equivalence on $\C$ if and only if $\C$ is $m$-semiadditive. 
\end{lem}
\begin{proof}
If $u_\C$ is an equivalence then $\C$ carries an $\cS^m_m$-module structure and is hence $m$-semiadditive by Corollary~\ref{c:amusing}. Now assume that $\C$ is $m$-semiadditive. By Corollary~\ref{c:amusing} again $\cS^m_m \otimes_{\K_m} \C$ is $m$-semiadditive and hence it will suffice to show that for every $m$-semiadditive $\infty$-category $\D$ the induced map
$$ u^*_{\C}: \Fun_{\K_m}(\cS^m_m \otimes_{\K_m} \C,\D) \lrar \Fun_{\K_m}(\C,\D) $$
is an equivalence. Identifying the functor $\infty$-category $\Fun_{\K_m}(\cS^m_m \otimes_{\K_m} \C,\D)$ with $\Fun_{\K_m}(\cS^m_m, \Fun_{\K_m}(\C,\D))$ and $u_\C^*$ with evaluation at $\ast \in \cS^m_m$ it will suffice, in light of Theorem~\ref{t:main}, to show that $\Fun_{\K_m}(\C,\D)$ is $m$-semiadditive. But this follows from Corollary~\ref{c:amusing} since $\Fun_{\K_m}(\C,\D)$ carries an action of $\cS^m_m$ (given by pre-composing the action on $\C$, whose existence is insured by Theorem~\ref{t:main}).
\end{proof}

\begin{lem}\label{l:key-2}
Let $\C$ be an $\cS^m_m$-module. Then the counit map 
$$ v_\C:\cS^m_m \otimes_{\K_m} \U(\C) \lrar \C $$ 
is an equivalence of $\cS^m_m$-modules. In particular, $\U:\Mod_{\cS^m_m}(\Cat_{\K_m}) \lrar \Cat_{\K_m}$ is fully-faithful.
\end{lem}
\begin{proof}
Since $\U$ is conservative it will suffice to show that $\U(v_\C)$ is an equivalence of $\infty$-categories. 
Since the composition
$$ \U(\C) \x{u_{\U(\C)}}{\lrar} \cS^m_m \otimes_{\K_m} \U(\C) \x{\U(v_\C)}{\lrar} \U(\C) $$
is homotopic to the identity we are reduced to showing that the functor
$$ u_{\U(\C)}: \U(\C) \lrar \cS^m_m \otimes_{\K_m} \U(\C)$$
is an equivalence. But this now follows from Lemma~\ref{l:key-1} since $\U(\C)$ is $m$-semiadditive in virtue of Corollary~\ref{c:amusing}. 
\end{proof}

\begin{cor}\label{c:equiv}
The forgetful functor induces an equivalence of $\infty$-categories $\Mod_{\cS^m_m}(\Cat_{\K_m}) \simeq \Add_m$.
\end{cor}

\begin{cor}\label{c:adjoints}
The inclusion $\Add_m \hrar \Cat_{\K_m}$ admits both a left adjoint, given by $\D \mapsto \cS^m_m \otimes_{\K_m}\D$, and a right adjoint, given by $\D \mapsto \Fun_{\K_m}(\cS^m_m,\D)$.
\end{cor}

\begin{cor}
$\cS^m_m$ is an \textbf{idempotent} algebra object in $\Cat_{\K_m}$. In particular, the monoidal product map $\cS^m_m \otimes_{\K_m} \cS^m_m \x{\simeq}{\lrar} \cS^m_m$ is an equivalence.
\end{cor}

Since $\cS^m_m$ is an idempotent algebra object of $\Cat_{\K_m}$ the functor $L: \Cat_{\K_m} \lrar \Cat_{\K_m}$ given by $\C \mapsto \cS^m_m \otimes_{\K_m} \C$ is a \textbf{localization functor} (see~\cite[Proposition 4.8.2.4]{higher-algebra}), and the $L$-local objects are those $\infty$-categories $\C$ such that the map $\C \lrar \cS^m_m \otimes_{\K_m} \C$ (induced by the unit $\cS_m \hrar \cS^m_m$) is an equivalence, which are exactly the $m$-semiadditive $\infty$-categories by Lemma~\ref{l:key-1}. In other words, the $\infty$-category $\Add_m$ is a localization of $\Cat_{\K_m}$ with localization functor $\cS^m_m \otimes_{\K_m} (-)$.

We shall now discuss tensor products of $m$-semiadditive $\infty$-categories.
\begin{prop}\label{p:monoidal}
There exists a symmetric monoidal structure $\Add_m^{\otimes} \lrar \Ne(\Fin_\ast)$ on $\Add_m$ such that the functor $\F:\Cat_{\K_m} \lrar \Add_m$ given by $\F(\D) = \cS^m_m \otimes_{\K_m} \D$ extends to a symmetric monoidal functor $\F^{\otimes}: \Cat_{\K_m}^{\otimes} \lrar \Add^{\otimes}_m$. In particular, $\cS^m_m$ is the unit of $\Add_m^{\otimes}$. 
\end{prop}
\begin{proof}
This is a particular case of~\cite[Proposition 4.8.2.7]{higher-algebra}.
\end{proof}

\begin{cor}
$\cS^m_m$ carries a canonical commutative algebra structure making it the initial object of $\CAlg(\Add_m)$.
\end{cor}

Let us refer to commutative algebra objects in $\Cat_{\K_m}$ as \textbf{$\K_m$-symmetric monoidal $\infty$-categories}. These can identified with ordinary symmetric monoidal $\infty$-categories such that the underlying $\infty$-category admits $\K_m$-indexed colimits and the monoidal product preserves $\K_m$-indexed colimits it each variable separately.
\begin{prop}
The inclusion $\Add_m^{\otimes} \hrar \Cat^{\otimes}_{\K_m}$ is symmetric monoidal. Furthermore the induced map
$$ \R:\CAlg(\Add_m) \lrar \CAlg(\Cat_{\K_m}) $$
is fully-faithful and its essential image is spanned by those $\K_m$-symmetric monoidal $\infty$-categories whose underlying $\infty$-category is $m$-semiadditive.
\end{prop}
\begin{proof}
Recall that the symmetric monoidal structure on $\Add_m$ was inherited from $\Cat_{\K_m}$ by identifying $\Add_m$ with the localization of $\Cat_{\K_m}$ associated to the localization functor $\L(\C) := \cS^m_m \otimes_{\K_m} \C$. In this case the inclusion $\Add_m \subseteq \Cat_{\K_m}$ of local objects is always lax symmetric monoidal structure (\cite[Proposition 2.2.1.9(3)]{higher-algebra}) and is given informally by the formula $(\C,\D) \mapsto \L(\C \otimes_{\K_m} \D)$. To show that in this case the inclusion is actually symmetric monoidal we need to show that if $\C,\D$ are local then $\C \otimes_{\K_m} \D$ is local as well. But this is a direct consequence of the fact that our localization functor is obtained by tensoring with an idempotent object. 



To show the second part of the claim, we note that the symmetric monoidal left adjoint $\F^{\otimes}: \Cat^{\otimes}_{\K_m} \lrar \Add^{\otimes}_m$ of Proposition~\ref{p:monoidal} induces a left adjoint $\F':\CAlg(\Cat_{\K_m}) \lrar \CAlg(\Add_m)$ to $\R$ whose value on the underlying objects is given by $\F'$. In particular, the counit of $\F'\dashv\R$ is given by the counit of $\F \dashv \U$ on the underlying $\infty$-categories and is hence an equivalence by Lemma~\ref{l:key-2}. We then get that $\R$ is fully-faithful. Let $\E \subseteq \CAlg(\Cat_{\K_m})$ denote the full subcategory spanned by those $\K_m$-symmetric monoidal $\infty$-categories whose underlying $\infty$-category is $m$-semiadditive, so that the image of $\R$ is contained in $\E$. To finish the proof we need to show that every object in $\E$ is in the image of $\R$. For this, it will suffice to show the unit of $\F'\dashv \R$ is an equivalence on objects whose underlying $\infty$-category is $m$-semiadditive. But this now follows from Lemma~\ref{l:key-1} since the unit of $\F'\dashv\R$ is given by the unit of $\F \dashv \U$ on underlying $\infty$-categories.

\end{proof}

We hence obtain yet another universal characterization of $\cS^m_m$:
\begin{cor}
The $\K_m$-symmetric monoidal $\infty$-category $\cS^m_m$ is \textbf{initial} among those $\K_m$-symmetric monoidal $\infty$-categories whose underlying $\infty$-category is $m$-semiadditive.
\end{cor}

\subsection{Higher commutative monoids}\label{s:monoids}
In \S\ref{s:modules} we discussed the inclusion of $\Add_m$ inside the $\infty$-category of $\Cat_{\K_m}$ of $\infty$-categories admitting $\K_m$-indexed colimits. But there is also a dual story, when one embeds $\Add_m$ inside the $\infty$-category $\Cat^{\K_m}$ consisting of those $\infty$-categories which admit $\K_m$-indexed \textbf{limits}. Indeed, the symmetry here is complete: the operation $\D \mapsto \D^{\op}$ which sends an $\infty$-category to its opposite induces an equivalence $\Cat^{\K_m} \simeq \Cat_{\K_m}$ which maps $\Add_m$ to itself. We may hence apply any of the constructions of the previous section to $\infty$-categories with $\K_m$-indexed limits by ``conjugating'' it with the operation $\D \mapsto \D^{\op}$. From an abstract point of view this seems to yield no additional interest. However, for one of the procedures above applying this conjugation yields an interesting relation with the theory of commutative monoids, which is worthwhile spelling out.

By Corollary~\ref{c:adjoints}, if $\D$ is an $\infty$-category which admits $\K_m$-indexed colimits, then the restriction functor $r:\Fun_{\K_m}(\cS^m_m,\D) \lrar \D$ exhibits $\Fun_{\K_m}(\cS^m_m,\D)$ as the universal $m$-semiadditive $\infty$-category carrying a $\K_m$-colimit preserving functor to $\D$. In other words, any $\K_m$-colimit preserving functor from any other $m$-semiadditive $\infty$-category $\C$ factors essentially uniquely through $r$. 

Now suppose that $\D$ admits $\K_m$-indexed \textbf{limits}. Then $\D^{\op}$ admits $\K_m$-indexed colimits and $\Fun_{\K_m}(\cS^m_m,\D^{\op})$ is the universal $m$-semiadditive $\infty$-category admitting a $\K_m$-colimit preserving functor to $\D^{\op}$. We now note that 
$$ (\Fun_{\K_m}(\cS^m_m,\D^{\op}))^{\op} \simeq \Fun^{\K_m}((\cS^m_m)^{\op},\D) \simeq \Fun^{\K_m}(\cS^m_m,\D), $$ 
where $\Fun^{\K_m}(-,-) \subseteq \Fun(-,-)$ denotes the full subcategory spanned by $\K_m$-limit preserving functors. It the follows that $\Fun^{\K_m}(\cS^m_m,\D)$ is the universal $m$-semiadditive $\infty$-category admitting a $\K_m$-limit preserving functor to $\D$. Our next goal is to relate the $\infty$-category $\Fun^{\K_m}(\cS^m_m,\D)$ with the theory of commutative monoid objects in $\D$.

\begin{define}
Let $m \geq -1$ be an integer and let $\D$ be an $\infty$-category admitting $\K_m$-indexed limits. An \textbf{$m$-commutative monoid} in $\D$ is a functor $\F: \cS^{m-1}_m \lrar \D$ with the following property: for every $X \in \cS^{m-1}_m$ the collection of maps $\F(\hat{i}_x): \F(X) \lrar \F(\ast)$ exhibits $\F(X)$ as the limit in $\D$ of the constant $X$-indexed diagram with value $\F(\ast)$. We will denote by $\CMon_m(\D) \subseteq \Fun^{\K_m}(\cS^{m-1}_m,\D)$ the full subcategory spanned by the $m$-commutative monoids.
\end{define}

\begin{example}
If $m=-1$ then $\cS^{m-1}_m = \cS^{-2}_{-1} = \cS_{-1}$ is the $\infty$-category of $(-1)$-finite spaces and ordinary maps between them. In particular, we may identify $\cS_{-1}$ with the category consisting of two objects $\emptyset,\ast$ and a unique non-identity morphism $\emptyset \lrar \ast$. An $\infty$-category $\D$ admits $\K_{-1}$-indexed limits if and only if it admits a final object. A functor $\cS_{-1} \lrar \D$ is completely determined by the associated morphism $\F(\emptyset) \lrar \F(\ast)$ in $\D$. By definition such a functor $\F$ is a $(-1)$-commutative monoid if and only if $\F(\emptyset)$ is a terminal object of $\D$. We may hence identify $\CMon_{-1}(\D)$ with the full subcategory of the arrow category of $\D$ spanned by those arrows $A \lrar B$ for which $A$ is a final object. In particular, if we fix a particular final object $\star \in \D$ then we may form an equivalence $\CMon_{-1}(\D) \simeq \D_{\star/}$. In other words, we may identify $\CMon_{-1}(\D)$ with the $\infty$-category of \textbf{pointed objects} $\D$.
\end{example}

\begin{example}
If $m=0$ then we may identify $\cS^{m-1}_m = \cS^{-1}_0$ with the category whose objects are finite sets, and such that a morphism from a finite set $A$ to a finite set $B$ is a pair $(C,f)$ where $C$ is a subset of $A$ and $f: C \lrar B$ is a map. In particular, $\cS^{-1}_0$ is equivalent to the nerve of a discrete category. By sending a finite set $A$ to the pointed set $A_+ = A \coprod \{\ast\}$ and sending a map $(C,f)$ to the map $f': A_+ \lrar B_+$ which restricts to $f$ on $C$ and sends $A \bksl C$ to the base point of $B_+$ we obtain an equivalence $\cS^{-1}_0 \simeq \Fin_\ast$, where $\Fin_\ast$ is the category of finite pointed sets. To say that an $\infty$-category $\D$ has $\K_0$-indexed limits is to say that $\D$ admits finite products. Unwinding the definitions we see that a functor $\cS^{-1}_0 \lrar \D$ is a $0$-commutative monoid object if and only if the corresponding functor $\Fin_\ast \lrar \D$ is a commutative monoid object in the sense of~\cite[Definition 2.4.2.1]{higher-algebra}, also known as an \textbf{$\EE_\infty$-monoid}. When $\D$ is the $\infty$-category of spaces this notion of commutative monoids was first developed by Segal under the name \textbf{special $\Gam$-spaces}.
\end{example}

\begin{lem}\label{l:mon}
Let $\D$ be an $\infty$-category which admits $\K_m$-indexed limits and let $\F: \cS^m_m \lrar \D$ be a functor. Then $\F$ preserves $\K_m$-indexed limits if and only if the restriction $\F|_{\cS^{m-1}_m}$ is an $m$-commutative monoid object.
\end{lem}
\begin{proof}
Apply Corollary~\ref{c:n-colimits} and Lemma~\ref{l:colimits} to $(\cS^m_m)^{\op} \simeq \cS^m_m$ and $\D^{\op}$.
\end{proof}

\begin{prop}\label{p:monoid}
Fix an $m \geq -1$ and let $\D$ be an $\infty$-category which admits $\K_m$-indexed limits. Then restriction along $\cS^{m-1}_m \hrar \cS^m_m$ induces an equivalence of $\infty$-categories
$$ \Fun^{\K_m}(\cS^m_m,\D) \x{\simeq}{\lrar} \CMon_n(\D) $$
\end{prop}
\begin{proof}
Let $\Cone_m$
be the right marked mapping cone of the natural inclusion $\iota: \cS^{m-1}_{m} \hrar \cS^{m}_{m}$ (see the discussion before Lemma~\ref{l:export}) and let
$ \Cone_m \hrar \M^{\natural} \x{r}{\lrar} \Del^1 $
be a factorization of the projection $\Cone_m \lrar (\Del^1)^{\sharp}$ into a trivial cofibration follows by a fibration in the Cartesian model structure over $(\Del^1)^{\sharp}$. 
Let $\iota_0: \cS^{m}_{m} \hrar \M \times_{\Del^1} \Del^{\{0\}} \subseteq \M$ and $\iota_1: \cS^{m-1}_{m} \hrar \M \times_{\Del^1} \Del^{\{1\}} \subseteq \M$ be the corresponding inclusions, so that $\iota_0$ and $\iota_1$ exhibit $r:\M \lrar \Del^1$ as a correspondence from $\cS^{m}_{m}$ to $\cS^{m-1}_{m}$ which is the one associated to the functor $\iota:\cS^{m-1}_{m} \hrar \cS^{m}_{m}$. 

Let $\Fun^{\flat}_0(\M^{\natural},\D^{\natural}) \subseteq \Fun^{\flat}(\M^{\natural},\D^{\natural})$ and $\Fun^{\flat}_0(\Cone_m,\D^{\natural}) \subseteq \Fun^{\flat}(\Cone_m,\D^{\natural})$ denote the respective full subcategories spanned by those marked functors whose restriction to $\cS^{m-1}_{m}$ is an $m$-commutative monoid in $\D$. Since the inclusion of marked simplicial sets $\Cone_m \lrar \M^{\natural}$ is a trivial cofibration in the Cartesian model structure over $(\Del^1)^{\sharp}$ it follows that the restriction map $\Fun^{\flat}_0(\M^{\natural},\D^{\natural}) \lrar \Fun^{\flat}_0(\Cone_m,\D^{\natural})$ is a trivial Kan fibration, and by Proposition~\ref{p:rkan} and~\cite[Proposition 4.3.2.15]{higher-topos} the restriction map $\Fun^{\flat}_0(\M^{\natural},\D^{\natural}) \lrar \CMon_m(\D)$ is a trivial Kan fibration. We may hence deduce that the restriction map
$$ \Fun^{\flat}_0(\Cone_m,\D^{\natural}) \x{\simeq}{\lrar} \CMon_m(\D) $$
is a an equivalence. On the other hand, by Lemma~\ref{l:mon} the image of the restriction map $\Fun^{\flat}_0(\Cone_m,\D) \lrar \Fun(\cS^{m}_{m},\D)$ consists of exactly those functors $\cS^m_m \lrar \D$ which preserves $\K_m$-indexed limits. Arguing as in the proof of Corollary~\ref{c:step-2} we may now conclude that the restriction map
$$ \Fun^{\K_m}(\cS^m_m,\D) \x{\simeq}{\lrar} \CMon_n(\D) $$
is an equivalence of $\infty$-categories, as desired.
\end{proof}

\begin{cor}\label{c:univ-cmon}
Let $\D$ be an $\infty$-category which admits $\K_m$-indexed limits. Then $\CMon_m(\D)$ is $m$-semiadditive and the forgetful functor $\CMon_m(\D) \lrar \D$ exhibits $\CMon_m(\D)$ as universal among those $m$-semiadditive $\infty$-categories admitting a $\K_m$-limit preserving map to $\D$. In particular, $\D$ is $m$-semiadditive if and only if the forgetful functor $\CMon_m(\D) \lrar \D$ is an equivalence.
\end{cor}

To get a feel for what these higher commutative monoids are, let us consider the example of the $\infty$-category $\cS$ of spaces. Let $\F: \cS^{m-1}_m \lrar \cS$ be an $m$-commutative monoid object and let us refer to $M = \F(\ast)$ as the \textbf{underlying space} of $\F$. We may then identify two types of morphisms in $\cS^{m-1}_m$. The first type are morphisms of the form
$$ \xymatrix{
& X\ar^{\Id}[dr]\ar_{f}[dl] & \\
Y && X\\
}$$
where $f$ is $(m-1)$-truncated, and which we denote $\hat{f}: Y \lrar X$ (see Definition~\ref{d:dual}). These morphisms help us to identify the spaces $\F(X)$: by definition, the collection of maps $\hat{i}_x:X \lrar \ast$ exhibit $\F(X)$ as the limit of the constant $X$-indexed diagram with value $\F(\ast)=M$. In particular, we may identify $\F(X)$ with the mapping space $\Map_{\cS}(X,M)$. Other morphisms of the form $\hat{f}: Y \lrar X$ don't really give more information: if $f: X \lrar Y$ is an $(m-1)$-truncated map then for every $x \in X$ we have $\hat{f} \circ \hat{i}_x = \hat{i}_{f(x)}$, and so the induced map 
$$ \hat{f}_*:\Map_{\cS}(Y,M) \simeq \F(Y) \lrar \F(X) \simeq \Map_{\cS}(X,M) $$ 
is forced to coincide with restriction along $f$. The second type of morphisms in $\cS^{m-1}_m$ are the spans of the form
$$ \xymatrix{
& X\ar^{g}[dr]\ar_{\Id}[dl] & \\
X && Y\\
}$$
where $g: X \lrar Y$ is any map of $m$-finite spaces. We can think of the associated map $g_*: \Map_{\cS}(X,M) \lrar \Map_{\cS}(Y,M)$ as encoding the \textbf{structure} of $M$. Let $X_y$ be homotopy fiber of $g$ over $y \in Y$, equipped with its natural map $i_{X_y}: X_y \lrar X$, and let $g_y: X_y \lrar \{y\}$ be the constant map. Then $\hat{i}_y \circ g = g_y \circ \hat{i}_{X_y}$ and so for each $\vphi \in \Map_{\cS}(X,M)$ the function $g_*(\vphi) \in \Map_{\cS}(Y,M)$ maps the point $y$ to the point $(g_y)_*(\vphi|_{X_y}) \in M$. We may hence think of the core algebraic structure of an $m$-commutative monoid as given by the maps $p_*: \Map(X,M) \lrar M$ associated to constant maps $p: X \lrar \ast$, while the other maps $g: X \lrar Y$ specify various forms of compatibility. Informally speaking, the structure of being an $m$-commutative monoid means that for every $m$-finite space $X$ we can take an $X$-family $\{\vphi(x)\}_{x \in X}$ of points in $M$ and ``integrate'' it to obtain a new point $\int_X \vphi := p_*(\vphi) \in M$. These operations are then required to satisfy various ``Fubini-type'' compatibility constraints when one is integrating over a space $X$ which is fibered over another space $Y$. We note that when $m=0$ the spaces involved are finite sets, and we obtain the usual notion of being able to sum a finite collection of points in a commutative monoid.

\begin{examples}\
\begin{enumerate}
\item
For every space $X$, the $\infty$-groupoid $(\cS_m \times_{\cS} \cS_{/X})^{\simeq}$ classifying $m$-finite spaces equipped with a map to $X$ is naturally an $m$-commutative monoid. This is the free $m$-commutative monoid generated from $X$.
\item
Any $\QQ$-vector space is an $m$-commutative monoid (in the category of $\QQ$-vector spaces). Indeed, if $X$ is an $m$-finite space then the limit $\lim_X V$ of the constant $X$-indexed diagram on $V$ is just the vector space of functions $f:\pi_0(X) \lrar V$. To such an $f$ we may associate the vector 
$$ \sum_{X_0 \in \pi_0(X)}\chi(X_0)f(X_0) \in V $$ 
where 
$\chi(X_0) = \frac{\prod_{i \geq 0} |\pi_{2i}(X_0)|}{\prod_{i \geq 0}|\pi_{2i+1}(X_0)|}$ is the ``homotopy cardinality'' of $X_0$. This yields a structure of an $m$-commutative monoid on $V$. 
\item
More generally, if $\D$ is an $m$-semiadditive $\infty$-category then any object in $\D$ carries a canonical $m$-commutative monoid structure and for each $X,Y \in \D$ the mapping space $\Map_\D(X,Y)$ is canonically an $m$-commutative monoid in spaces. For example, by the main result of~\cite{ambi}, for any two $K(n)$-local spectra $X,Y$ the mapping space from $X$ to $Y$ is an $m$-commutative monoid in spaces.
\item
If $\C$ is an $\infty$-category which admits $\K_m$-indexed colimits (resp. $\K_m$-indexed limits) then $\C$ carries the coCartesian (resp. Cartesian) $m$-commutative monoid structure in $\Cat_{\infty}$, and its maximal $\infty$-groupoid is an $m$-commutative monoid in spaces. 
\end{enumerate}
\end{examples}

Let us now discuss the role of $m$-commutative monoids in the setting of $m$-semiadditive \textbf{presentable} $\infty$-categories.
\begin{lem}\label{l:right}
Let $\D$ be a presentable $\infty$-category. Then $\CMon_m(\D)$ is presentable and the forgetful functor $\CMon_m(\D) \lrar \D$ is conservative, accessible and preserves all limits.
\end{lem}
\begin{proof}
Let $\CMon_m(\D) \subseteq \Fun(\cS^{m-1}_m,\D)$ be the natural inclusion. Then $\CMon_m(\D)$ is closed under limits in $\Fun(\cS^{m-1}_m,\D)$ and under $\kappa$-filtered colimits for any $\kappa$ such that the simplicial sets in $\K_m$ are $\kappa$-small. Since $\Fun(\cS^{m-1}_m,\D)$ is presentable it now follows from~\cite[Corollary 5.5.7.3]{higher-topos} that $\CMon_m(\D)$ is presentable and the inclusion $\CMon_m(\D) \hrar \Fun(\cS^{m-1}_m,\D)$ is accessible. This, in turn, implies that the composition $\CMon_m(\D) \hrar \Fun(\cS^{m-1}_m,\D) \x{\ev_\ast}{\lrar} \D$ is accessible and preserves limits. Finally, 
to show that $\CMon_n(\D) \lrar \D$ is conservative it is enough to note that if $f: M \lrar M'$ is a map in $\CMon_m(\D)$ such that $f_{\ast}: M(\ast) \lrar M'(\ast)$ is an equivalence in $\D$ then for any $X \in \cS^{m-1}_m$ the induced map $f_X: M(X) \simeq \lim_X M(\ast) \lrar \lim_X M'(\ast)$ is an equivalence and hence $f$ is an equivalence.
\end{proof}

When $\D$ is presentable, Lemma~\ref{l:right} and the adjoint functor theorem (\cite{higher-topos}) imply that the forgetful functor $\CMon_n(\D) \lrar \D$ admits a \textbf{left adjoint} $\F: \D \lrar \CMon_m(\D)$, which can be considered as the \textbf{free $m$-commutative monoid} functor. Given two presentable $\infty$-categories $\C,\D$ let us denote by $\Fun^{\rL}(\C,\D)$ the $\infty$-category of \textbf{left functors} from $\C$ to $\D$ (i.e. those functors which admit right adjoints) and by $\Fun^{\rR}(\C,\D)$ the $\infty$-category of \textbf{right functors} from $\C$ to $\D$.

\begin{cor}\label{c:pres-univ}
Let $\D$ be a presentable $\infty$-category and let $\E$ be a presentable $m$-semiadditive $\infty$-category. Then post-composition with the forgetful functor $\CMon(\D) \lrar \D$ induces an equivalence
$$ \Fun^{\mathrm{R}}(\E,\CMon(\D)) \x{\simeq}{\lrar} \Fun^{\mathrm{R}}(\E,\D) .$$
Dually pre-composition with $\F: \D \lrar \CMon_m(\D)$ induces an equivalence
$$ \Fun^{\mathrm{L}}(\CMon(\D),\E) \x{\simeq}{\lrar} \Fun^{\mathrm{L}}(\D,\E) .$$
In particular, the functor $\F$ exhibits $\CMon(\D)$ as the free presentable $m$-semiadditive $\infty$-category generated from $\D$.
\end{cor}
\begin{proof}
Let us prove the first claim (the second then follows by the equivalence $\Fun^{\mathrm{R}}(-,-) \simeq  \Fun^{\mathrm{L}}(-,-)$ which associates to every right functor its left adjoint). By Corollary~\ref{c:univ-cmon} it will suffice to show that if $\F: \E \lrar \CMon_n(\D)$ is a functor that preserves $\K_m$-indexed limits then $\F$ belongs to $\Fun^{\mathrm{R}}(\E,\D)$ if and only if $\ev_{\ast} \circ \F: \E \lrar \D$ belongs to $\Fun^{\mathrm{R}}(\E,\CMon(\D))$. By the adjoint functor theorem right functors between presentable $\infty$-categories are exactly the limit preserving functors which are also accessible, i.e., preserve sufficiently filtered colimits. The result is now follows from Lemma~\ref{l:right} which asserts that $\ev_{\ast}$ preserves limits and sufficiently filtered colimits, and also detects them since it is conservative.
\end{proof}

Given an $\infty$-category $\C$, let $\P_{\K_m}(\C) \subseteq \Fun(\C^{\op},\cS)$ denote the full subcategory consisting of those presheaves which send $\K_m$-indexed colimits in $\C$ to limits of spaces. Then $\P_{\K_m}(\C)$ is presentable and is an accessible localization of $\Fun(\C^{\op},\cS)$ (choose an infinite cardinal $\kappa$ such that all $\K_m$-colimit diagrams in $\C$ are $\kappa$-small and use~\cite[Corollary 5.5.7.3]{higher-topos}). Let $\Pr^{\mathrm{L}}$ denote the $\infty$-category of presentable $\infty$-categories and left functors between them. Identifying $\Pr^{\mathrm{L}}$ as a full subcategory of cocomplete $\infty$-categories and colimit preserving functors and using~\cite[Corollary 5.3.6.10]{higher-topos} we may conclude that the functor
\begin{equation}\label{e:P}
\P_{\K_m}: \Cat_{\K_m} \lrar \mathrm{Pr}^{\mathrm{L}}
\end{equation}
is left adjoint to the forgetful functor $\Pr^{\mathrm{L}} \lrar \Cat_{\K_m}$. In particular, we may consider $\P_{\K_m}(\C)$ as the free presentable $\infty$-category generated from $\C$. We hence obtain two universal characterizations of the $\infty$-category $\CMon_m(\cS)$. On the one hand, by Corollary~\ref{c:pres-univ} we may identify $\CMon_m(\cS)$ as the free presentable $m$-semiadditive $\infty$-category generated from the presentable $\infty$-category $\cS$. On the other hand, since $\cS^m_m \simeq (\cS^m_m)^{\op}$ we may interpret Proposition~\ref{p:monoid} as identifying $\CMon_m(\cS) \simeq \P_{\K_m}(\cS^m_m)$ as the free presentable $\infty$-category generated from the $\K_m$-cocomplete $\infty$-category $\cS^m_m$. Furthermore, by~\cite[Proposition 4.8.1.14]{higher-algebra} and~\cite[Remark 4.8.1.8]{higher-algebra} the functor~\eqref{e:P} is symmetric monoidal (where $\Pr^{\mathrm{L}}$ is endowed with the symmetric monoidal structure inherited from that of cocomplete $\infty$-categories). 
We may then deduce the following:
\begin{cor}\label{c:idem}
The $\infty$-category $\CMon_m(\cS)$ is an idempotent commutative algebra object in $\Pr^{\mathrm{L}}$. In particular, the monoidal product $\CMon_m(\cS) \otimes \CMon_m(\cS) \lrar \CMon_m(\cS)$ is an equivalence.
\end{cor}

\begin{lem}
Let $\D$ be a presentable $\infty$-category. Then $\D$ carries an action of $\CMon_m(\cS)$ (with respect to the symmetric monoidal structure of $\Pr^{\mathrm{L}}$) if and only if $\D$ is $m$-semiadditive.
\end{lem}
\begin{proof}
By~\cite[Remark 4.8.1.17]{higher-algebra} the data of an action of $\CMon_m(\cS)$ on a presentable $\infty$-category $\D$ is equivalent to the data of a monoidal colimit preserving functor $\CMon_m(\cS) \lrar \Fun^{\mathrm{L}}(\D,\D)$, which since~\eqref{e:P} is monoidal, is equivalent to the data of a $\K_m$-colimit preserving monoidal functor $\cS^m_m \lrar \Fun^{\mathrm{L}}(\D,\D)$, i.e., to an action of $\cS^m_m$ on $\D$ which preserves $\K_m$-colimits in $\cS^m_m$ and all colimits in $\D$. We now observe that any action of $\cS^m_m$ on $\D$ which preserves $\K_m$-colimits in $\cS^m_m$ will automatically preserve all colimits which exist in $\D$, since the object $X \in \cS^m_m$ will necessarily acts as an $X$-indexed colimit of the identity functor. The desired result now follows from Corollary~\ref{c:equiv}. 
\end{proof}

Arguing as in the proof of Lemma~\ref{l:key-2} we may now conclude the following:
\begin{cor}\label{c:equiv-pres}
The forgetful functor $\Mod_{\CMon_m(\cS)}(\Pr^{\mathrm{L}}) \lrar \Pr^{\mathrm{L}}$ is fully-faithful and its essential image consists of the $m$-semiadditive presentable $\infty$-categories.
\end{cor}

\begin{rem}
The statements of Corollary~\ref{c:equiv-pres} and~\ref{c:idem} are strongly related. In fact, under mild conditions on a symmetric monoidal $\infty$-category $\C$, the property of $A \in \CAlg(\C)$ being idempotent is equivalent to the forgetful functor $\Mod_A(\C) \lrar \C$ being fully-faithful. Idempotent commutative algebra objects in $\Pr^{\mathrm{L}}$ feature in some recent investigations of T. Schlank (\cite{mode}), where they are called \textbf{modes}. Informally speaking, modes describe aspects of presentable $\infty$-categories which are both a property and a structures, such as being pointed (the mode of pointed spaces), being semiadditive (the mode of $\EE_\infty$-spaces) being stable (the mode of spectra), being an $(n,1)$-category (the mode of $n$-truncated spaces), and more. Corollary~\ref{c:idem} then adds a new infinite family of modes: the mode of $m$-commutative monoids in spaces for every $m$, which is associated to the property of being $m$-semiadditive.
\end{rem}

Let us now consider the case where we replace $\cS$ by the $\infty$-category $\Cat_\infty$ of $\infty$-categories. As above, we may informally consider an $m$-commutative monoid structure on an $\infty$-category $\M$ as giving us a rule for taking an $X$-indexed family of objects of $\M$ (where $X$ is an $m$-finite space) and producing a new object of $\M$. Two immediate examples come to mind: if $\M$ is an $\infty$-category admitting $\K_m$-indexed colimits then we may form the colimit of any $X$-indexed family of objects in $\M$. On the other hand, if $\M$ admits $\K_m$-indexed limits then we may form the limit of any such family. One might hence expect that if $\M$ admits $\K_m$-indexed colimits (resp. limits) then there should be a canonical $m$-commutative monoid structure on $\M$, which can be called the \textbf{coCartesian} (resp. \textbf{Cartesian}) $m$-commutative monoid structure. To show that these structures indeed exist we shall prove the following theorem:

\begin{thm}\label{t:car}\
\begin{enumerate}
\item
The forgetful functor $\CMon_m(\Cat_{\K_m}) \lrar \Cat_{\K_m}$ is an equivalence. In other words, every object in $\Cat_{\K_m}$ admits an essentially unique $m$-commutative monoid structure. 
\item
The forgetful functor $\CMon_m(\Cat^{\K_m}) \lrar \Cat^{\K_m}$ is an equivalence. In other words, every object in $\Cat^{\K_m}$ admits an essentially unique $m$-commutative monoid structure. 
\end{enumerate}
\end{thm}

\begin{rem}
If $\M$ is an $\infty$-category which admits $\K_m$-indexed colimits, then we may consider it as belonging to either $\Cat_\infty$ or $\Cat_{\K_m}$. Since the faithful inclusion $\Cat_{\K_m} \hrar \Cat_\infty$ preserves limits we obtain a natural map
\begin{equation}\label{e:car}
\CMon_m(\Cat_{\K_m})\times_{\Cat_{\K_m}}\{\M\} \lrar \CMon_m(\Cat_\infty)
\times_{\Cat_\infty}\{\M\} 
\end{equation}
where the left hand side is contractible by Theorem~\ref{t:car}, and the right hand side is an $\infty$-groupoid which can be considered as the space of $m$-commutative monoid structures on $\M$. The point in $\CMon_m(\Cat_\infty) \times_{\Cat_\infty}\{\M\}$ determined by~\eqref{e:car} can be considered as identifying the coCartesian $m$-commutative monoid structure on $\M$. Similarly, if $\M$ admits $\K_m$-indexed limits then the image of the map
$$ \CMon_m(\Cat^{\K_m})\times_{\Cat^{\K_m}}\{\M\} \lrar \CMon_m(\Cat_\infty)\times_{\Cat_\infty}\{\M\} 
$$
identifies the Cartesian $m$-commutative monoid structure. 
\end{rem}

\begin{rem}
The category $\cS^m_m$ admits $\K_m$-indexed limits and colimits, but also carries an $m$-commutative monoid structure which is neither Cartesian nor coCartesian. To see this, observe that the operation $\C \mapsto \Span(\C) = \Span(\C,\C)$ which associates to any $\infty$-category $\C$ with finite limits its span category determines a limit preserving functor $\Span: \Cat^{\K_{\fin}} \lrar \Cat^{\K_{\fin}}$, where $\K_{\fin}$ denotes the collection of simplicial sets with finitely many non-degenerate simplices. We then have an induced functor
$$ \Span_*: \CMon_m(\Cat^{\K_{\fin}}) \lrar \CMon_m(\Cat^{\K_{\fin}}) .$$
Since the $\infty$-category $\cS_m$ has both $\K_m$-indexed limits and $\K_m$-indexed colimits it carries both a Cartesian $m$-commutative monoid structure and a coCartesian $m$-commutative monoid structure. Applying the functor $\Span_*$ we obtain two $m$-commutative monoid structures on $\Span(\cS_m) = \cS^m_m$. The coCartesian $m$-commutative monoid structure of $\cS_m$ induces an $m$-commutative monoid structure on $\cS^m_m$ which is both coCartesian and Cartesian. The Cartesian $m$-commutative monoid structure on $\cS_m$, however, induces a \textbf{different} $m$-commutative monoid structure on $\cS^m_m$, which is neither Cartesian nor coCartesian. The restriction of this structure to $\cS^{-1}_0$ determines a symmetric monoidal structure on $\cS^m_m$ which is the one we've been considering throughout this paper.
\end{rem}

As $\Cat^{\K_m} \simeq \Cat_{\K_m}$ by the functor which sends $\C$ to $\C^{\op}$, Theorem~\ref{t:car} will follows from Theorem~\ref{t:main} and Proposition~\ref{p:monoid} once we prove the following result:
\begin{prop}\label{p:cat_Km}
The $\infty$-category $\Cat_{\K_m}$ is $m$-semiadditive.
\end{prop}

The proof of Proposition~\ref{p:cat_Km} will be given below. Since $\Cat_{\K_m}$ has all limits it follows that $\Cat_{\K_m}^{\op}$ has all colimits and hence admits a canonical action of the $\infty$-category of spaces $\cS$ which preserves colimits in each variable separately. Dually, $\Cat_{\K_m}$ admits an action of $\cS^{\op}$ which preserves limits in each variable separately. Given a space $X \in \cS^{\op}$ this action $[X]: \Cat_{\K_m} \lrar \Cat_{\K_m}$ sends $\M$ to $\M^X = \lim_X\M$ and sends $f: X \lrar Y$ to the restriction functor $f^*\M^Y \lrar \M^X)$.
For our purposes we will only be interested in the action of the full subcategory $\cS_m^{\op} \subseteq \cS^{\op}$ on $\Cat_{\K_m}$. Given a morphism $\sig: X \lrar Y$ in $\cS^m_m$ of the form
\begin{equation}\label{e:span-2}
\xymatrix{
& Z \ar^{q}[dr]\ar_{p}[dl] & \\
X && Y \\
}
\end{equation}
where $X,Y,Z$ are $m$-finite spaces we will denote by $T_{\sig}: [Y] \Rightarrow [X]$ the natural transformation given by the composition
$$ T_{\sig}(\M): \M^Y \x{q^*}{\lrar} \M^Z \x{p_!}{\lrar} \M^X $$
where $q^*$ denotes the restriction functor along $q$ and $p_!$ denotes the left Kan extension functor, whose existence is insured by the fact that $\M$ has $\K_m$-indexed colimits. If $\sig$ is a span as in~\eqref{e:span-2} then we will denote by $\hat{\sig}: Y \lrar X$ the \textbf{dual span} $Y \x{q}{\llar} Z \x{p}{\lrar} X$.

\begin{lem}\label{l:lem-3}
Let $\tr_X: X \times X \lrar \ast$ be the span of Definition~\ref{d:diag}. 
Then the natural transformation
$$ T_{\tr_X}:\Id \lrar [X \times X] \simeq [X] \circ [X] $$ 
exhibits $[X]$ as a self-adjoint functor. Furthermore, under this self-adjunction the natural transformation $T_{\sig}: [X] \Rightarrow [Y]$ associated to a span $\sig: X \lrar Y$ is dual to the natural transformation $T_{\hat{\sig}}: [Y] \Rightarrow [X]$ associated with the dual span $\hat{\sig}$. 
\end{lem}
\begin{proof}
The Beck-Chevalley condition for pullbacks and left Kan extensions (see~\cite[Proposition 4.3.3]{ambi}) implies in particular that the association $\sig \mapsto T_\sig$ respects composition of spans up to homotopy. Both claims now follow from the fact that $\tr:  X \times X \lrar \ast$ exhibit $X$ as self-dual in the monoidal $\infty$-category $\cS^m_m$ (see Remark~\ref{r:self-dual}) and that under this self duality the dual morphism of $\sig$ is $\hat{\sig}$.
\end{proof}

\begin{proof}[Proof of Proposition~\ref{p:cat_Km}]
Arguing by induction, let us assume that $\Cat_{\K_m}$ is $(m'-1)$-semiadditive for some $-1 \leq m' \leq m$ and show that it is in fact $m'$-semiadditive. By Corollary~\ref{c:equiv} (applied to $\Cat_{\K_m}^{\op}$) we may extend the $(\cS_{m'})^{\op}$-action on $\Cat_{\K_m}$ described above to an $(\cS^{m'-1}_{m'})^{\op}$-action which preserves $\K_{m'}$-indexed limits in each variable separately. Applying Lemma~\ref{l:lem-2} and Lemma~\ref{l:effect} to $\Cat_{\K_m}^{\op}$ we may deduce that for every morphism of the form $Y \x{q}{\llar} X$ in $\cS_{m',m'-1}^{\op} \subseteq \cS^{m'-1}_{m'}$ the induced transformation $[q](\M): [X](M) \simeq \M^X \lrar \M^Y \simeq [Y](M)$ is given by the formation of left Kan extensions. Applying now Lemma~\ref{l:lem-3} we may conclude that for every $X \in \cS^{m'-1}_{m'}$, the natural transformation $[\tr_X]: \Id \Rightarrow [X] \circ [X]$ exhibits $[X]$ as a self-adjoint functor. By (the dual version of) Proposition~\ref{p:non-deg} the $\infty$-category $\Cat_{\K_m}$ is $m'$-semiadditive, as desired.
\end{proof}

\begin{rem}
Proposition~\ref{p:cat_Km} implies in particular that if $X$ is an $m$-finite space and $\M \in \Cat_{\K_m}$ is an $\infty$-category admitting $\K_m$-indexed colimits then $\Fun(X,\M) \simeq \lim_X\M$ is also a model for the \textbf{colimit} of the constant $X$-indexed diagram with value $\M$. Using Lemma~\ref{l:lem-3} we can make this claim more precise: for any $\M \in \Cat_{\K_m}$ and $X \in \cS_m$, the collection of left Kan extension functors $(i_x)_!: \M \lrar \M^X$ exhibits $\M^X$ as the colimit of the constant $X$-indexed diagram with value $\M$.
\end{rem}

\subsection{Decorated spans}\label{s:decorated}
Theorem~\ref{t:main} identifies $\cS^m_m$ as the free $m$-semiadditive $\infty$-category generated by a single object. In this section we will show how to bootstrap Theorem~\ref{t:main} in order to obtain a description of the free $m$-semiadditive $\infty$-category generated by an arbitrary small $\infty$-category $\C$. 

Let $\pi:\cS_m(\C) \lrar \cS_m$ be a Cartesian fibration classifying the functor $X \mapsto \C^X$. We may informally describe objects in $\cS_m(\C)$ as pairs $(X,\L_X)$ where $X$ is an $m$-finite space and $\L_X: X \lrar C$ is a $\C$-valued local system on $X$ (i.e., a functor). A map $(X,\L_X) \lrar (Y,\L_Y)$ in $\cS_m(\C)$ can be described in these terms as a pair $(f,T)$ where $f: X \lrar Y$ is a map of spaces and $T: \L_X \Rightarrow f^*\L_Y$ is a map of local systems on $X$ (i.e., a natural transformation). In particular, a morphism $(f,T)$ corresponds to a $\pi$-Cartesian edge of $\cS_m(\C)$ if and only $T$ is an equivalence in $\C^X$. Now since $\cS_m$ admits pullbacks it follows that $\cS_m(\C)$ admits pullbacks of diagram of the form $\vphi: \Del^1 \coprod_{\Del^{\{1\}}} \Del^1 \lrar \cS_m(\C)$ such that $\vphi|\Del^1$ is $\pi$-Cartesian. Let $\cS_m^{\car}(\C) \subseteq \cS_m(\C)$ denote the subcategory containing all objects and whose mapping spaces are the subspaces spanned by $\pi$-Cartesian edges. Then $\cS_m^{\car}$ determines a weak coWaldhausen structure on $\C$ (see \S\ref{s:spans}) and we may consider the associated span $\infty$-category
$$ \cS^m_m(\C) := \Span(\cS_m(\C), \cS_m^{\car}(\C)) .$$
By Remark~\ref{r:map-span} we may identify the objects of $\cS^m_m(\C)$ with the objects of $\cS_m(\C)$ and the mapping space in $\cS^m_m(\C)$ from $(X,\L_X)$ to $(Y,\L_Y)$ with the classifying space of spans
\begin{equation}\label{e:span-17}
\xymatrix{
& (Z,\L_Z) \ar_{(p,T)}[dl]\ar^{(q,S)}[dr] &\\
(X,\L_X) && (Y,\L_Y)\\
}
\end{equation}
such that $(p,T)$ is $\pi$-Cartesian (i.e., such that $T$ is an equivalence in $\C^Z$).

The fiber of the Cartesian fibration $\cS_m(\C) \lrar \cS_m$ over $\ast \in \cS_m$ is equivalent to $\C^{\ast} \simeq \C$ and we may fix an equivalence $\C \x{\simeq}{\lrar} \cS_m(\C) \times_{\cS_m} \{\ast\}$. We will denote by $\iota: C \lrar \cS_m(\C)$ the composition of this equivalence with the inclusion of the fiber over $\ast \in \cS_m$. Let 
$$ U \subseteq \cS_m(\C)^{\Del^1} \times_{\cS_m(\C)^{\Del^{\{1\}}}} \C $$ 
be the full subcategory spanned by those arrows $(X,\L_X) \lrar \iota(C)$ which are $\pi$-Cartesian. We may informally describe objects of $U$ as tuples $(X,\L_X,C,T)$ where $T$ is a natural equivalence from $\L_X: X \lrar \C$ to the constant functor $\ovl{C}: X \lrar \C$ with value $C$. Since $\pi:\cS_m(\C) \lrar \cS_m$ is a Cartesian fibration and $\ast$ is terminal in $\cS_m$ it follows that the maps $U \x{\simeq}{\lrar} \cS_m^{\Del^1} \times_{\cS_m^{\Del^{\{1\}}}} \C \x{\simeq}{\lrar} \cS_m \times \C$ are trivial Kan fibrations (whose composition can informally be described as sending $(X,\L_X,C,T)$ to $(X,C)$) and hence there exists an essentially unique section $s: \cS_m \times \C \lrar U$. We will denote by
$$ \iota': \cS_m \times \C \lrar \cS_m(\C) $$
the composition of $s$ with the projection $U \lrar \cS_m(\C)$. We may informally describe $\iota'$ by the formula $\iota'(X,C) = (X,\ovl{C})$, where $\ovl{C}: X \lrar \C$ denotes the constant functor with value $C$.

Recall that we denote by $\C^{\simeq} \subseteq \C$ the maximal subgroupoid of $\C$. Since $\cS_m$ admits pullbacks it follows that $\cS_m \times \C$ admits pullbacks of diagram of the form $\vphi: \Del^1 \coprod_{\Del^{\{1\}}} \Del^1 \lrar \cS_m \times \C$ such that $\vphi|\Del^1$ belongs to $\cS_m \times \C^{\simeq}$. Since the functor $\iota':\cS_m \times \C \lrar \cS_m(\C)$ maps $\cS_m \times \C^{\simeq}$ to $\cS_m^{\car}(\C)$ we obtain an induced functor of span $\infty$-categories
$$\iota'': \cS^m_m \times \C \simeq \cS^m_m \times \Span(\C,\C^{\simeq}) \simeq \Span(\cS_m \times \C,\cS_m \times \C^{\simeq}) \lrar \Span(\cS_m(\C), \cS_m^{\car}(\C)) = \cS^m_m(\C) .$$
We may informally describe the functor $\iota'':\cS^m_m \times \C \lrar \cS^m_m(\C)$ as the functor which sends the object $(X,C)$ to the object $(X,\ovl{C})$ and a pair $(X \x{p}{\llar} Z \x{q}{\lrar} Y,\alp: C \lrar D)$ of a morphism in $\cS^m_m$ and a morphism in $\C$ to the span
$$ \xymatrix{
& (Z,\ovl{C}) \ar_{(p,\Id)}[dl]\ar^{(q,\ovl{\alp})}[dr] &\\
(X,\ovl{C}) && (Y,\ovl{D})\\
}
$$ 
Our goal in this section is to prove the following characterization of the above constructions in terms of suitable universal properties:
\begin{thm}\label{t:univ}\
\begin{enumerate}
\item
The functor $\iota:\C \lrar \cS_m(\C)$ exhibits $\cS_m(\C)$ as the free $\infty$-category with $\K_m$-indexed colimits generated from $\C$.
\item
The composed functor $\C \lrar \cS_m(\C) \lrar \cS^m_m(\C)$ exhibits $\cS^m_m(\C)$ as the free $m$-semiadditive $\infty$-category generated from $\C$.
\item
The functor $\cS^m_m \times \cS_m(\C) \lrar \cS^m_m(\C)$ exhibits $\cS^m_m(\C)$ as the tensor product $\cS^m_m \otimes_{\K_m} \cS_m(\C)$ in $\Cat_{\K_m}$.
\end{enumerate}
\end{thm}

The rest of this section is devoted to the proof of Theorem~\ref{t:univ}, which is covered by Corollaries~\ref{c:free-colimit},~\ref{c:univ} and~\ref{c:tensor}. We begin with the following general lemma about colimits in Cartesian fibrations.
\begin{lem}\label{l:cartesian}
Let $K$ be a Kan complex and let $\ovl{\vphi}: K^{\triangleright} \lrar \C$ be a diagram taking values in an $\infty$-category $\C$. Let $\pi:\D \lrar \C$ be a Cartesian fibration classified by a functor $\chi: \C^{\op} \lrar \Cat_{\infty}$ which sends $\ovl{\vphi}$ to a limit diagram in $\Cat_{\infty}$. Then a lift $\ovl{\psi}:K^{\triangleright} \lrar \D$ of $\ovl{\vphi}$ is a $\pi$-colimit diagram in $\D$ if and only if $\ovl{\psi}$ sends every morphism in $K^{\triangleright}$ to a $\pi$-Cartesian edge.
\end{lem}
\begin{proof}
Let $\vphi = \ovl{\vphi}|_{K}$ and $\psi = \ovl{\psi}|_{K}$ and consider the induced map $\pi_*:\D_{\psi/} \lrar \C_{\vphi/}$. By definition, $\ovl{\psi}$ is a $\pi$-colimit diagram if and only if the object $\ovl{\psi} \in \D_{\psi/}$ is $\pi_\ast$-initial. By (the dual of)~\cite[Proposition 2.4.3.2]{higher-topos} the map $\pi_*$ is a Cartesian fibration, and hence by~\cite[Corollary 4.3.1.16]{higher-topos} we have that $\ovl{\psi}$ is $\pi_*$-initial if and only if it is initial when considered as an object of $\D_{\psi/} \times_{\C_{\vphi/}}\{\ovl{\vphi}\}$. Using the natural equivalence (see~\cite[\S 4.2.1]{higher-topos} for the two types of slice constructions)
\begin{equation}\label{l:alternative}
\D_{\psi/} \times_{\C_{\vphi/}}\{\ovl{\vphi}\} \simeq \D^{\psi/} \times_{\C^{\vphi/}}\{\ovl{\vphi}\} \simeq
\Fun(K^{\triangleright},\D) \times_{\Fun(K,\D) \times \Fun(K^{\triangleright},\C)} \{(\psi,\ovl{\vphi})\} 
\end{equation}
it will suffice to show that $\ovl{\psi}: K^{\triangleright} \lrar \D$ is initial when considered as an object of the RHS of~\eqref{l:alternative} if and only it sends all edges to $\pi$-Cartesian edges. Let $L = (K \times \Del^1)^{\triangleright}$ and let $L_1,L_2 \subseteq L$ be the subsimplicial sets given by
\begin{equation}
\xymatrix{
L_1 := (K \times \Del^{\{1\}})^{\triangleright} \coprod_{K \times \Del^{\{1\}}} K \times \Del^1 \ar@{^{(}->}[r]  
& (K \times \Del^1)^{\triangleright} &  
(K \times \Del^{\{0\}})^{\triangleright} := L_2 \ar@{_{(}->}[l] \\
}
\end{equation}
Let $\ovl{L}$ be the marked simplicial set whose underlying simplicial set is $L$ and the marked edges are those which are contained in $(K \times \Del^{\{1\}})^{\triangleright}$. Similarly, let $\ovl{L}_1$ and $\ovl{L}_2$ be the marked simplicial sets whose underlying simplicial sets are $L_1$ and $L_2$ respectively and whose markings are inherited from $L$. In particular, $\ovl{L}_2 = L_2^{\flat} \cong K^{\triangleright}$. We now claim that the inclusions $\ovl{L}_1 \hrar \ovl{L}$ and $\ovl{L}_2 \hrar \ovl{L}$ are marked anodyne. For $\ovl{L}_1$ this follows from the fact that $L_1 \hrar L$ is inner anodyne by~\cite[Lemma 2.1.2.3]{higher-topos} and all the marked edges of $L$ are contained in $L_1$. For $\ovl{L}_2$ we can write the inclusion $K \times \Del^{\{0\}} \hrar K \times \Del^1$ as a transfinite composition of pushouts along $\partial \Del^n \hrar \Del^n$ for $n \geq 0$, yielding a description of $\ovl{L}_2 \hrar \ovl{L}$ as a transfinite composition of pushouts along marked maps of the form $(\Lam^{n+1}_{n+1},(\Lam^{n+1}_{n+1})_1 \cap \Del^{\{n,n+1\}}) \hrar (\Del^{n+1},\Del^{\{n,n+1\}})$ which are marked anodyne by definition.
Let $\D^{\natural}$ be the marked simplicial set whose underlying simplicial set is $\D$ and the marked edges are the $\pi$-Cartesian edges. Then $\D^{\natural}$ is fibrant in the Cartesian model structure over $\C$ and so we obtain a zig-zag of trivial Kan fibrations
\begin{equation}\label{e:zig-zag}
\xymatrix{
\Fun^{\flat}(\ovl{L}_1,\D^{\natural}) \times_{\Fun(K \times \Del^{\{0\}},\D) \times \Fun(L_1,\C)} \{(\psi,\ovl{\vphi}'_1)\} \\
\Fun^{\flat}(\ovl{L},\D^{\natural}) \times_{\Fun(K \times \Del^{\{0\}},\D) \times \Fun(L,\C)} \{(\psi,\ovl{\vphi}')\} \ar^{\simeq}[d]\ar_{\simeq}[u] \\ 
\Fun^{\flat}(\ovl{L}_2,\D^{\natural}) \times_{\Fun(K \times \Del^{\{0\}},\D) \times \Fun(L_2,\C)} \{(\psi,\ovl{\vphi}'_2)\}\\
}
\end{equation}
where $\ovl{\vphi}': L \lrar \C$ is the composition of $\ovl{\vphi}$ with the projection $L \lrar K^{\triangleright}$ and $\ovl{\vphi}'_i = \ovl{\vphi}'|_{\ovl{L}_i}$. Let $\ovl{\rho}: \ovl{L}_1 \lrar \D^{\natural}$ be an object which corresponds to $\ovl{\psi}: L_2 \lrar \D$ under the zig-zag of equivalences~\eqref{e:zig-zag}. We now observe that if a map $\ovl{L} \lrar \D^{\natural}$ sends all edges in $\ovl{L}_2$ to $\pi$-Cartesian edges then it sends all edges in $\ovl{L}$ to $\pi$-Cartesian edges. It then follows that $\ovl{\psi}$ sends all edges to $\pi$-Cartesian edges if and only if $\ovl{\rho}$ sends all edges to $\pi$-Cartesian edges. To finish the proof it will hence suffice to show that $\ovl{\rho}$ is initial in  
$ \Fun^{\flat}(\ovl{L}_1,\D^{\natural}) \times_{\Fun(K \times \Del^{\{0\}},\D) \times \Fun(L_1,\C)} \{(\psi,\ovl{\vphi}'_1)\} $ if and only if it sends all edges to $\pi$-Cartesian edges. 

We now invoke our assumption that the functor $\chi: \C^{\op} \lrar \Cat_{\infty}$ maps $\ovl{\vphi}$ to a limit diagram in $\Cat_{\infty}$. By~\cite[Proposition 3.3.3.1]{higher-topos} and using the fact that $K$ is a Kan complex we may conclude that the projection
$$ \Fun^{\flat}(\ovl{L}_1,\D^{\natural}) \times_{\Fun(K \times \Del^{\{0\}},\D) \times \Fun(L_1,\C)} \{(\psi,\ovl{\vphi}'_1)\} \x{\simeq}{\lrar} $$
$$ \Fun(K \times \Del^1,\D) \times_{\Fun(K \times \Del^{\{0\}},\D) \times \Fun(K \times \Del^1,\C)} \{(\psi,\vphi')\} $$
is a weak equivalence, where $\vphi': K \times \Del^1 \lrar \C$ is the composition of $\vphi$ with the projection $K \times \Del^1 \lrar K$. We now observe that $\ovl{\rho}|_{K \times \Del^1}$ is initial in 
$$ \Fun(K \times \Del^1,\D) \times_{\Fun(K \times \Del^{\{0\}},\D)\times \Fun(K \times \Del^1,\C)} \{(\psi,\vphi')\} \simeq (\Fun(K,\D) \times_{\Fun(K,\C)}\{\vphi\})_{\psi/} ,$$
if and only if the morphism in $\Fun(K,\D) \times_{\Fun(K,\C)}\{\vphi\}$ determined by $\ovl{\rho}|_{K \times \Del^1}$ is an equivalence, and so the desired result follows. 
\end{proof}

For an object $(X,\L_X) \in \cS_m(\C)$ and a point $x \in X$, let us denote by $i_x: (\{x\},\L_X(x)) \lrar (X,\L_X)$ the corresponding morphism in $\cS_m(\C)$.
\begin{lem}\label{l:colim-C-1}\
\begin{enumerate}
\item
The $\infty$-category $\cS_m(\C)$ admits $\K_m$-indexed colimits. Furthermore, if $\ovl{\psi}: K^{\triangleright} \lrar \cS_m(\C)$ is a cone diagram with $K \in \K_m$ then $\ovl{\psi}$ is a colimit diagram if and only if $\pi \circ \ovl{\psi}: K^{\triangleright} \lrar \cS_m$ is a colimit diagram and $\ovl{\psi}$ sends every morphism in $K^{\triangleright}$ to a $\pi$-Cartesian edge.
\item
For any $\infty$-category $\D$ with $\K_m$-indexed colimits, an arbitrary functor $\F: \cS_m(\C) \lrar \D$ preserves $\K_m$-indexed colimits if and only if for every $(X,\L_X) \in \cS_m(\C)$ the collection of maps $\F(i_x):\F(\{x\},\L_X(x)) \lrar \F(X,\L_X)$ exhibits $\F(X,\L_X)$ as the colimit of the diagram $\{\F(\{x\},\L_X(x))\}_{x \in X}$.
\end{enumerate}
\end{lem}
\begin{proof}
Let us first prove (1). By definition, the Cartesian fibration $\pi: \cS_m(\C) \lrar \cS_m$ is classified by the functor $\chi_\C: \cS_m^{\op} \lrar \Cat_{\infty}$ given by $\chi_\C(X) = \C^X$. Since the inclusion $\cS_m \hrar \cS$ preserves $\K_m$-indexed colimits (Lemma~\ref{l:les}) and the inclusion $\cS \lrar \Cat_{\infty}$ preserves all colimits it follows that $\chi_\C$ sends $\K_m$-indexed colimit diagrams to limit diagrams in $\Cat_{\infty}$. Now let $K$ be an $m$-finite Kan complex, let $\psi: K \lrar \cS_m(\C)$ be a diagram and let $\vphi = \pi \circ \psi: K \lrar \cS_m$. Since $\cS_m$ admits $\K_m$-indexed colimits we may extend $\vphi$ to a colimit diagram $\ovl{\vphi}: K^{\triangleright} \lrar \cS_m$. Since $\chi_\C \circ \ovl{\vphi}^{\op}: (K^{\op})^{\triangleleft} \lrar \Cat_{\infty}$ is a limit diagram and $K$ is a Kan complex we may use~\cite[Proposition 3.3.3.1]{higher-topos} to deduce the existence of a dotted lift
$$ \xymatrix{
K \ar^{\psi}[r]\ar[d] & \cS_m(\C) \ar^{\pi}[d] \\
K^{\triangleright} \ar@{-->}[ur]^{\ovl{\psi}}\ar^{\ovl{\vphi}}[r] & \cS_m \\
}$$
such that $\ovl{\psi}$ sends all edges in $K^{\triangleright}$ to $\pi$-Cartesian edges. By Lemma~\ref{l:cartesian} we may conclude that $\ovl{\psi}$ is a $\pi$-colimit diagram, and since $\ovl{\vphi}$ is a colimit diagram it follows that $\ovl{\psi}$ is also a colimit diagram in $\cS_m(\C)$. Finally, by uniqueness of colimits this construction covers all colimit of $\K_n$-indexed diagram, and so the characterization of colimits given in (1) follows.

We shall now prove (2). The ``only if'' direction is clear since the collection of maps $i_x:(\{x\},\L_X(x)) \lrar (X,\L_X)$ exhibits $(X,\L_X)$ as the colimit in $\cS_n(\C)$ of the diagram $\{(\{x\},\L_X(x))\}_{x \in X}$ by the characterization of colimits diagram given in (1). Now suppose that for every $X \in \cS_n(\C)$ the collection $\F(i_x):\F(\{x\},\L_X(x)) \lrar \F(X,\L_X)$
exhibits $\F(X,\L_X)$ as the colimit of the diagram $\{(\{x\},\L_X(x))\}$. Let $Y \in \K_n$ be an $n$-finite space and let $\vphi: Y \lrar \cS_m(\C)$ be a $Y$-indexed diagram. For each $y \in Y$ let us write $\vphi(y) = (Z_y,\L_y)$ where $Z_y$ is an $m$-finite space and $\L_y: Z_y \lrar \C$ is a local system. By (1) we may identify the colimit of $\vphi$ in $\cS_m(\C)$ with the pair $(Z,h)$ where $Z$ is the total space of the left fibration $p:Z \lrar Y$ classified by $\vphi$ and $\L: Z \lrar \C$ is the essentially unique local system such that $\L|_{Z_y} = \L_y$. We now proceed as in the proof of Lemma~\ref{l:colimits}. Let $\psi := \F\circ \vphi: Y \lrar \D$ and let $\rho: Z \lrar \D$ be the diagram which sends $z \in Z$ to $\F(\{z\},\L(z))$. 
We then have a natural transformation $\sig:\rho \Rightarrow p^*\psi$ and our assumption on $\F$ implies that $\sig$ exhibits $\psi: Y \lrar \D$ as a left Kan extension of $\rho: Z \lrar \D$ along $p: Z \lrar Y$. Similarly, our assumption implies that the natural transformation from $\rho$ to the constant diagram on $\F(Z,\L)$ exhibits the latter as the colimit of $\rho$. By the pasting lemma for left Kan extensions we may now conclude that the natural transformation from $\psi$ to the constant diagram on $\F(Z,\L)$ exhibits the latter as the colimit of $\psi$, as desired.
%
%
%
%
\end{proof}

Now let $(X,\L_X) \in \cS_m(\C)$ be an object. The forgetful functor $\pi:\cS_m(\C)_{/(X,\L_X)} \lrar (\cS_m)_{/X}$ is a Cartesian fibration classifying the functor $(Y \x{f}{\lrar} X) \mapsto \C^Y_{/f^*\L_X}$. Since all the fibers of $\pi$ have terminal objects we may choose an essentially unique Cartesian section $s: (\cS_m)_{/X} \lrar \cS_m(\C)_{/(X,\L_X)}$ such that $s(f:Y \lrar X)$ is terminal in the fiber over $f$ for every $f:Y \lrar X$ in $(\cS_m)_{/X}$. 
We may then form the diagram of $\infty$-categories
$$ \xymatrix{
V \ar[r]\ar_-{s'}[d] & (\cS_m)_{/X} \ar^{s}[d] \\
\C \times_{\cS_m(\C)} \left(\cS_m(\C)_{/(X,\L_X)}\right) \ar[r]\ar[d] & \cS_m(\C)_{/(X,\L_X)} \ar[d] \\
\C \ar[r] & \cS_m(\C)  \\
}$$
where $V$ is chosen so that the top square is Cartesian (and the bottom square is Cartesian by construction as well). It then follows that the external rectangle is Cartesian. Using this we may then identify objects of $V$ with pairs $(C,f: Y \lrar X,\eta)$ where $\eta$ is an equivalence from $(\ast,C)$ to $(Y,f^*\L_X)$ in $\cS_m(\C)$. In particular, $Y \simeq \ast$ must be contractible and the data of $\eta$ is just an equivalence from $\L_X(f(Y)) = \L_X(f(\ast))$ to $C$ in $\C$. We may hence identify $V$ with the $\infty$-groupoid $X$, in which case the map 
$$ s': X \lrar \C \times_{\cS_m(\C)} \left(\cS_m(\C)_{/(X,\L_X)}\right) $$ 
can be informally described by the formula
$$ s'(x) = (\L_X(x),(\{x\},\L_X(x))) \in \C \times_{\cS_m(\C)} \left(\cS_m(\C)_{/(X,\L_X)}\right) $$

\begin{lem}\label{l:cofinal}
The map $s'$ is cofinal.
\end{lem}
\begin{proof}
The map $s'$ is a base change of the map $s$, which is cofinal since it is a terminal section of a Cartesian fibration.
\end{proof}

\begin{cor}\label{c:free-colimit}
The inclusion $\iota:\C \lrar \cS_m(\C)$ exhibits $\cS_m(\C)$ as the $\infty$-category obtained from $\C$ by freely adding $\K_m$-indexed colimits. In particular, if $\D$ is an $\infty$-category with $\K_m$-indexed colimits then restriction along $\iota$ induces an equivalence of $\infty$-categories
$$ \Fun_{\K_m}(\cS_m(\C),\D) \x{\simeq}{\lrar} \Fun(\C,\D) $$ 
\end{cor}
\begin{proof}
By Lemma~\ref{l:colim-C-1}(1) we know that $\cS_m(\C)$ has $\K_m$-indexed colimits. Now suppose that $\D$ is an $\infty$-category that admits $\K_m$-indexed colimits and let $\F: \C \lrar \D$ be a functor. Since $\D$ admits colimits indexed by $X$ for every $X \in \cS_m$ Lemma~\ref{l:cofinal} and~\cite[Lemma 4.3.2.13]{higher-topos} together imply that any functor $\F: \C \lrar \D$ admits a left Kan extension $\ovl{\F}: \cS_m(\C) \lrar \D$, and that an arbitrary functor $\ovl{\F}: \cS_m(\C) \lrar \D$ extending $\F$ is a left Kan extension of $\F$ if and only if for every $(X,\L_X) \in \cS_m(\C)$ the collection of maps $\{\ovl{\F}(i_x):\ovl{\F}(\{x\},\L_X(x)) = \F(\L_X(x)) \lrar \ovl{\F}(X,\L_X)\}$ exhibit $\ovl{\F}(X,\L_X)$ as the colimit of the diagram $\{\F(\L_X(x))\}_{x \in X}$. By Lemma~\ref{l:colim-C-1}(2) the latter condition is equivalent to the condition that $\ovl{\F}$ preserves $\K_m$-indexed colimits. The desired result now follows from the uniqueness of left Kan extensions (see \cite[Proposition 4.3.2.15]{higher-topos}).
\end{proof}

We now address the universal property of $\cS^m_m(\C)$ as described in the second claim of Theorem~\ref{t:univ}. We begin with the following Lemma:

\begin{lem}\label{l:colim-C-2}\
\begin{enumerate}
\item
The $\infty$-category $\cS^m_m(\C)$ admits $\K_m$-indexed colimits and the inclusion $\cS_m(\C) \lrar \cS^m_m(\C)$ preserves $\K_m$-indexed colimits. Furthermore, every $\K_m$-indexed diagram in $\cS^m_m(\C)$ comes from a $\K_m$-indexed diagram in $\cS_m(\C)$.
\item
For any $\infty$-category $\D$ with $\K_m$-indexed colimits, an arbitrary functor $\F: \cS^m_m(\C) \lrar \D$ preserves $\K_m$-indexed colimits if and only if for every $(X,\L_X) \in \cS^m_m(\C)$ the collection of maps $\{\F(i_x):\F(\{x\},\L_X(x)) \lrar \F(X,\L_X)\}$ exhibit $\F(X,\L_X)$ as the colimit of the diagram $\{\F(\{x\},\L_X(x))\}$.
\item
The functor $\iota'':\cS^m_m \times \C \lrar \cS^m_m(\C)$ preserves $\K_m$-indexed colimits in the $\cS^m_m$ variable.
\end{enumerate}
\end{lem}
\begin{proof}
Let us begin with Claim (1). We first claim that every equivalence in $\cS^m_m(\C)$ is in the image of the map $\cS_m(\C) \lrar \cS^m_m(\C)$. Indeed, a morphism in $\cS^m_m(\C)$ is given by a span
\begin{equation}\label{e:span-18}
\xymatrix{
& (Z,\L_Z) \ar_{(p,T)}[dl]\ar^{(q,S)}[dr] &\\
(X,\L_X) && (Y,\L_Y)\\
}
\end{equation}
such that $T$ is an equivalence in $\C^Z$. If~\eqref{e:span-18} is an equivalence then its image in $\cS^m_m$ is an equivalence which means by Lemma~\ref{l:yoneda} that $p: Z \lrar X$ is an equivalence in $\cS_m$ and hence that $(p,T): (Z,\L_Z) \lrar (X,\L_X)$ is an equivalence in $\cS_m(\C)$. In this case the span~\eqref{e:span-18} is essentially equivalent to an honest map, i.e., is in the image of $\cS_m(\C) \lrar \cS^m_m(\C)$. Since the inclusion $\cS_m(\C) \lrar \cS^m_m(\C)$ is faithful it follows that any $\K_m$-indexed diagram in $\cS^m_m(\C)$ is the image of an essentially unique $\K_m$-indexed diagram in $\cS_m(\C)$. It will hence suffice to prove that the map $\cS_m(\C) \lrar \cS^m_m(\C)$ preserves $\K_m$-indexed colimits.

By Lemma~\ref{l:colim-C-1}(2) it will suffice to show that for every $(X,\L_X) \in \cS^m_m(\C)$, the collection of maps $(\{x\},\L_X(x)) \lrar (X,\L_X)$ exhibit $(X,\L_X)$ as the colimit of the $X$-indexed diagram $\{(\{x\},\L_X(x))\}_{x \in X}$ in $\cS^m_m(\C)$. In other words, we need to show that the data of a span of the form~\eqref{e:span-18} such that $T: \L_Z \x{\simeq}{\lrar} p^*\L_X$ is an equivalence in $\C^Z$ is equivalent to the data of an $X$-indexed family of spans
\begin{equation}\label{e:span-20}
\xymatrix{
& (Z_x,\L_Z|_{Z_x}) \ar_{(p|_{Z_x},T|_{Z_x})}[dl]\ar^{(q|_{Z_x},S|_{Z_x})}[dr] &\\
(\{x\},\L_X(x)) && (Y,\L_Y)\\
}
\end{equation}
where $Z_x$ denotes the homotopy fiber of $p:Z \lrar X$ over $x \in X$. But this is now a consequence of the straightening-unstraightening equivalence which implies that the collection of fiber functors $i_x^*: (\cS_m)_{/X} \lrar \cS$ identifies $(\cS_m)_{/X}$ with $\cS_m^X$, and furthermore for every $Z \lrar X$ the collection of maps $i_x^*Z \lrar Z$ exhibits $Z$ as the homotopy colimit of the $X$-indexed family $\{Z_x\}_{x \in X}$. Claim (2) is now a direct consequence of the above and Lemma~\ref{l:colim-C-1}(2).

Let us now prove Claim (3). We have a commutative diagram of $\infty$-categories
$$ \xymatrix{
\cS_m \times \C \ar[d]\ar^{\iota'}[r] & \cS_m(\C) \ar[d] \\
\cS^m_m \times \C \ar^{\iota''}[r] & \cS^m_m(\C) \\
}$$
where the vertical maps are faithful. Let $C \in \C$ be an object, let $K$ be an $m$-finite Kan complex and let $\vphi: K \lrar \cS^m_m \times \{C\}$ be a diagram. By Lemma~\ref{l:yoneda} the diagram $\vphi$ is the image of an essentially unique diagram $\vphi': K \lrar \cS_m \times \{C\}$, and the inclusion $\cS_m \times \{C\} \lrar \cS^m_m \times \{C\}$ preserves $\K_m$-indexed colimits. By Claim (2) above it will suffice to show that the top horizontal map preserves $\K_m$-indexed colimits, which is clear in light of the characterization of colimit cones in $\cS_m(\C)$ given in~\ref{l:colim-C-1}(1).
\end{proof}

Since $\cS_m(\C)$ admits $\K_m$-indexed colimits it carries a canonical action of $\cS_m$, given informally by the formula $X \otimes (Y,\L_Y) = \colim_{x \in X} (Y,\L_Y) = (X \times Y,p_Y^*\L_Y)$, where $p_Y: X \times Y \lrar Y$ is the projection on the second factor. By Lemma~\ref{l:colim-C-1}(1) we get that if $f: X \lrar X'$ is a map of $m$-finite Kan complexes then the induced edge $X \otimes (Y,\L_Y) \lrar X' \otimes (Y,\L_Y)$ is $\pi$-Cartesian in $\cS_m(\C)$. On the other hand, for a fixed space $X$ and a $\pi$-Cartesian map $(Y,\L_Y) \lrar (Z,\L_Z)$ the induced map $X \otimes (Y,\L_Y) \lrar X \otimes (Z,\L_Z)$ is again $\pi$-Cartesian. It then follows that the action of $\cS_m$ on $\cS_m(\C)$ induces an action of $\Span(\cS_m)$ on $\Span(\cS_m(\C),\cS_m^{\car}(\C)) = \cS^m_m(\C)$. Furthermore, by Lemma~\ref{l:colim-C-2} and Lemma~\ref{l:yoneda} this action preserves $\K_m$-indexed colimits in each variable separately. By Corollary~\ref{c:amusing} we now get that $\cS^m_m(\C)$ is \textbf{$m$-semiadditive}.

Let us now consider the left marked mapping cone 
$$ \Cone_m(\C) =  \left[\cS^m_{m} \times \C \times(\Del^1)^{\sharp}\right]\coprod_{\cS^m_{m} \times \C \times\Del^{\{1\}}}\cS^m_m(\C) $$ 
of the inclusion $\iota'': \cS^m_{m} \times \C  \hrar \cS^{m}_{m}(\C)$. Let
$ \Cone_m \hrar \M^{\natural} \x{r}{\lrar} \Del^1 $
be a factorization of the projection $\Cone_m \lrar (\Del^1)^{\sharp}$ into a trivial cofibration follows by a fibration in the coCartesian model structure over $(\Del^1)^{\sharp}$. In particular, $r: \M \lrar \Del^1$ is a coCartesian fibration and the marked edges of $\M^{\natural}$ are exactly the $r$-coCartesian edges. Let $\iota_0: \cS^m_{m} \times \C \hrar \M \times_{\Del^1} \Del^{\{0\}} \subseteq \M$ and $\iota_1: \cS^{m}_{m}(\C) \hrar \M \times_{\Del^1} \Del^{\{1\}} \subseteq \M$ be the corresponding inclusions. Then $\iota_0$ and $\iota_1$ exhibit $r:\M \lrar \Del^1$ as a correspondence from $\cS^m_{m} \times \C $ to $\cS^{m}_{m}(\C)$ which is the one associated to the functor $\iota'':\cS^m_{m} \times \C  \hrar \cS^{m}_{m}(\C)$. 

\begin{prop}\label{p:step-C}
Let $\D$ be an $m$-semiadditive $\infty$-category and let $\F: \cS^m_m \times \C \lrar \D$ be a functor which preserves $\K_m$-indexed colimits in the $\cS^m_m$ variable. Then the following holds:
\begin{enumerate}
\item
$\F$ admits a left Kan extension 
$$ \xymatrix{
\cS^m_m \times \C \ar_{\iota_0}[d]\ar^{\F}[r] & \D \\
\M \ar@{-->}[ur]_{\ovl{\F}} & \\
}$$
\item
An arbitrary functor $\ovl{\F}: \M \lrar \D$ extending $\F$ is a left Kan extension if and only if $\ovl{\F}$ maps $r$-coCartesian edges in $\M$ to equivalences in $\D$ and $\ovl{\F} \circ \iota_1: \cS^m_m(\C) \lrar \D$ preserves $\K_m$-indexed colimits.
\end{enumerate}
\end{prop}
\begin{proof}
For $(Y,\L_Y) \in \cS^m_m(\C)$ let us denote by
$$ \I_{(Y,\L_Y)} := (\cS^m_m \times \C) \times_{\cS^m_m \times \C} \cS^m_m(\C)_{/(Y,\L_Y)} $$
the associated comma $\infty$-category.
To prove (1), it will suffice by~\cite[Lemma 4.3.2.13]{higher-topos} to show that the composed map 
$$ \F_{(Y,\L_Y)}:\I_{(Y,\L_Y)} \lrar \cS^m_m \times \C \lrar \D $$ 
can be extended to a colimit diagram in $\D$ for every $(Y,\L_Y)\in \cS^m_m(\C)$. Now an object of $\I_{(Y,\L_Y)}$ corresponds to an object $(X,C) \in \cS^m_m \times \C$ and a morphism $(X,\ovl{C}) \lrar (Y,\L_Y)$ in $\cS^m_m(\C)$, i.e., a span
\begin{equation}\label{e:object-3}
\xymatrix{
& (Z,\L_Z) \ar^{(q,S)}[dr]\ar_{(p,T)}[dl] & \\
(X,\ovl{C}) && (Y,\L_Y) \\
}
\end{equation}
where $T: \L_Z \lrar p^*\ovl{C}$ is an equivalence in $\C^Z$. 
Let $\J_{(Y,\L_Y)} := (\cS_m \times C) \times_{\cS_m(\C)} (\cS_m(\C))_{/(Y,\L_Y)}$ be the comma $\infty$-category over $(Y,\L_Y)$ associated to the inclusion $\iota':\cS_m \times \C \lrar \cS_m(\C)$. Then the faithful maps $\cS_m \times \C \hrar \cS^m_m \times \C$ and $\cS_m(\C) \lrar \cS_m^m(\C)$ induce a fully-faithful inclusion $\rho:\J_{(Y,\L_Y)} \hrar \I_{(Y,\L_Y)}$ whose essential image consists of those objects as in~\eqref{e:object-3} for which $p: Z \lrar X$ is an equivalence. 
We now claim that $\rho$ is cofinal.

Consider an object $P \in \I_{(Y,\L_Y)}$ of the form~\eqref{e:object-3}. We need to show that the comma $\infty$-category $(\J_{(Y,\L_Y)})_{P/} := \J_{(Y,\L_Y)} \times_{\I_{(Y,\L_Y)}} (\I_{(Y,\L_Y)})_{P/}$ is weakly contractible. Given an object $(q',S'):(X',\ovl{C'}) \lrar (Y,\L_Y)$ of $\J_{(Y,\L_Y)}$ the mapping space from $P$ to $\rho(X',\ovl{C'},q',S')$ in $\I_{(Y,\L_Y)}$ is given by the homotopy fiber of the map
\begin{equation}\label{e:fiber-5}
\Map_{\cS^m_m \times \C}((X, C),(X', C')) \lrar \Map_{\cS^m_m(\C)}((X,\ovl{C}),(Y,\L_Y))
\end{equation}
over the map determined by $P$. In light of Remark~\ref{r:map-span} we may identify the homotopy fiber of~\eqref{e:fiber-5} with the homotopy fiber of the map 
\begin{equation}\label{e:fiber-6}
((\cS_m)_{/X})^{\simeq} \times_{\cS_m} ((\cS_m)_{/X'})^{\simeq} \times \Map_{\C}(C,C') \lrar (\cS_m^{\car}(\C)_{/(X,\ovl{C})})^{\simeq} \times_{\cS_m(\C)} (\cS_m(C)_{/(Y,\L_Y)})^{\simeq}
\end{equation}
over the object corresponding to $P$. Now since the map $((\cS_m)_{/X})^{\simeq} \lrar (\cS_m^{\car}(\C)_{/(X,\ovl{C})})^{\simeq}$ is an equivalence we may identify the homotopy fiber of~\eqref{e:fiber-6} with the homotopy fiber of the map
\begin{equation}\label{e:fiber-7}
\Map_{\cS_m}(Z,X') \times \Map_{\C}(C,C') \lrar \Map_{\cS_m(\C)}((Z,\ovl{C}),(Y,\L_Y)) 
\end{equation}
over the point $(q,S) \in \Map_{\cS_m(\C)}((Z,\ovl{C}),(Y,\L_Y))$ determined by $P$. Unwinding the definitions we recover that the map~\eqref{e:fiber-7} sends a pair $(q'': Z \lrar X',\alp: C \lrar C')$ to the composition 
$$\xymatrix{
(Z,\ovl{C}) \ar[r]^-{(q'',\ovl{\alp})} & (X',\ovl{C'}) \ar[r]^{(q',S')} & (Y,\L_Y) .\\}  
$$ 
We may then conclude that the functor $\J_{(Y,\L_Y)} \lrar \cS$ defined by $(X',\ovl{C'},q',S') \mapsto \Map_{\I_{(Y,\L_Y)}}(P,\rho(X',\ovl{C'},q',S'))$ is corepresented in $\J_{(Y,\L_Y)}$ by the object $(q,S):(Z,\ovl{C}) \lrar (Y,\L_Y)$. It then follows that $(\J_{(Y,\L_Y)})_{P/}$ has an initial object and is hence weakly contractible. This means that $\rho:\J_{(Y,\L_Y)} \hrar \I_{(Y,\L_Y)}$ is cofinal, as desired.

It will now suffice to show that each of the diagrams 
$$ \F_{(Y,\L_Y)}|_{\J_{(Y,\L_Y)}}:\J_{(Y,\L_Y)} \lrar \D $$
can be extended to a colimit diagram. Let $\J'_{(Y,\L_Y)} := \J_{(Y,\L_Y)} \times_{\cS_{n}} \{\ast\} \subseteq \J_{Y}$ be the full subcategory spanned by objects of the form $(q,S):(\ast,\ovl{C}) \lrar (Y,\L_Y)$. 
Since we assumed that $\F: \cS^m_m \times \C \lrar \D$ preserves $\K_m$-indexed colimits in the first coordinate it follows from Proposition~\ref{p:adjoint} that the restriction $\F|_{\cS_{m} \times \C}: \cS_{m} \times \C \lrar \D$ preserves $\K_{m}$-indexed colimits in the first coordinate and by combining Lemma~\ref{l:colimits} with Lemma~\ref{l:cofinal} we may conclude that the functor $\F|_{\cS_{m} \times \C}$ is a left Kan extension of its restriction to $\{\ast\} \times \C \in \cS_{m}$. Now since the projection $\J_Y \lrar \cS_{m} \times \C$ is a right fibration it induces an equivalence $(\J_{(Y,\L_Y)})_{/(X',C',q',S')} \lrar (\cS_{m} \times \C)_{/(X',C')}$ for every $(X',C',q',S') \in \J_{(Y,\L_Y)}$. We may then conclude that $\F_{(Y,\L_Y)}|_{\J_{(Y,\L_Y)}}$ 
is a left Kan extension of $\F_{(Y,\L_Y)}|_{\J_{(Y,\L_Y)}'}$. Since $\D$ admits $\K_m$-indexed colimits and $\J'_{(Y,\L_Y)}$ contains an $m$-finite Kan complex as a cofinal subcategory by Lemma~\ref{l:cofinal} the diagram $\F_{(Y,\L_Y)}|_{\J'_{(Y,\L_Y)}}$ admits a colimit. It then follows that the diagram $\F_{(Y,\L_Y)}|_{\J_{(Y,\L_Y)}} \lrar \D$ admits a colimit, as desired. 

To prove (2), we begin by noting that by the above considerations, an arbitrary extension $\ovl{\F}: \M \lrar \D$ is a left Kan extension if and only if for every $(Y,\L_Y)$ the diagram
$$ (\J'_{(Y,\L_Y)})^{\triangleright} \lrar \D $$
determined by $\ovl{\F}$ is a colimit diagram. By Lemma~\ref{l:cofinal} the functor $Y \lrar \J'_{(Y,\L_Y)}$ sending $y \in Y$ to the object $(\{y\},\L_Y(y)) \lrar (Y,\L_Y)$ is cofinal and so $\ovl{\F}$ is a left Kan extension of $\F$ if and only if for every $(Y,\L_Y)$ the diagram
\begin{equation}\label{e:colim-diag}
\ovl{\F}_{(Y,\L_Y)}: Y^{\triangleright} \lrar \D 
\end{equation}
determined by $\ovl{\F}$ is a colimit diagram. Now by Lemma~\ref{l:colimits} and Lemma~\ref{c:n-colimits} we know that for each $Y \in \cS^m_m$ the collection of maps $\iota_y:\{y\} \lrar Y$ exhibits $Y$ as the colimit in $\cS^m_m$ of the constant $Y$-diagram with value $\ast$. Since $\F: \cS^m_m \times \C \lrar \D$ preserves $\K_m$-indexed colimits in the first variable it follows that each $(Y,C) \in \cS^m_m \times \C$ the collection of maps $\F(\iota_y,\Id_C):\F(\{y\},C) \lrar \F(Y,C)$ exhibit $\F(Y,C)$ as the colimit of the diagram $\{\F(\{y\},C)\}_{y \in Y}$. This means that $\ovl{\F}_{(Y,\ovl{C})}$ is a colimit diagram if and only if $\ovl{\F}$ maps every $r$-coCartesian edge in $\M$ of the form $(Y,C) \lrar (Y,\ovl{C})$ (covering the map $0 \lrar 1$ of $\Del^1$) to an equivalence in $\D$. Since all the other $r$-coCartesian edges of $\M$ are equivalences we may conclude that $\ovl{\F}$ maps $r$-coCartesian edges to equivalences if and only if the diagrams $\ovl{\F}_{(Y,\ovl{C})}$ are colimit diagrams for every $Y \in \cS^m_m$ and $C \in \C$. On the other hand, when these two equivalent conditions hold for $Y = \ast$ and all $C \in \C$ then the condition that $\ovl{\F}_{(Y,\L_Y)}$ is a colimit diagram for every $(Y,\L_Y)$ is equivalent by Lemma~\ref{l:colim-C-2}(2) to the condition that $\ovl{\F} \circ \iota_1: \cS^m_m(\C) \lrar \D$ preserves $\K_m$-indexed colimits. We may hence conclude that $\ovl{\F}$ is a left Kan extension of $\ovl{\F}$ if and only if it maps all $r$-coCartesian edges of $\M$ to equivalences in $\D$ and $\ovl{\F} \circ \iota_1: \cS^m_m(\C) \lrar \D$ preserves $\K_m$-indexed colimits.
\end{proof}

Given an $\infty$-category $\D$ admitting $\K_m$-indexed colimits, let us denote by 
$$\Fun_{\K_m/\C}(\cS^m_m \times \C,\D) \subseteq \Fun(\cS^m_m \times \C,\D)$$ 
the full subcategory spanned by those functors which preserves $\K_m$-indexed colimits in the $\cS^m_m$ variable. 
\begin{cor}\label{c:univ}
Let $\D$ be an $m$-semiadditive $\infty$-category. Then restriction along $\iota'': \cS^m_m \times \C \hrar \cS^m_m(\C)$ induces an equivalence of $\infty$-categories:
$$ \Fun_{\K_m}(\cS^m_m(\C),\D) \x{\simeq}{\lrar} \Fun_{\K_m/\C}(\cS^m_m \times \C,\D) .$$
\end{cor}
\begin{proof}
Let $r: \M^{\natural} \lrar \Del^1$ be as above and consider the marked simplicial set $\D^{\natural} = (\D,M)$ where $M$ is the collection of edges which are equivalences in $\D$. Let $\Fun^{\flat}_{\K_m}(\M^{\natural},\D^{\natural}) \subseteq \Fun^{\flat}(\M^{\natural},\D^{\natural})$ and $\Fun^{\flat}_{\K_m}(\Cone_m(\C),\D^{\natural}) \subseteq \Fun^{\flat}(\Cone_m(\C),\D^{\natural})$ be the respective full subcategories spanned by those makred functors whose restriction to $\cS^{m}_{m} \times \C$ preserves $\K_m$-indexed colimits in the left variable and whose restriction to $\cS^m_m(\C)$ preserves $\K_m$-indexed colimits. Since the map $\Cone_m(\C) \lrar \M^{\natural}$ is marked anodyne it follows that the restriction map $\Fun^{\flat}_{\K_m}(\M^{\natural},\D^{\natural}) \lrar \Fun^{\flat}_{\K_m}(\Cone_m(\C),\D^{\natural})$ is an equivalence and by Proposition~\ref{p:step-C} the restriction map $\Fun^{\flat}_{\K_m}(\M^{\natural},\D^{\natural}) \lrar \Fun_{\K_m/\C}(\cS^m_m \times \C,\D)$ is an equivalence. We may hence deduce that the restriction map 
$$ \iota_0^*:\Fun^{\flat}_{\K_m}(\Cone_m(\C),\D^{\natural}) \lrar \Fun_{\K_m/\C}(\cS^m_m \times \C,\D) $$ 
is an equivalence.

Now since the inclusion $\cS^{m}_{m} \times \C \hrar \Cone_m(\C)$ is a pushout along the inclusion $\cS^{m}_{m} \times \times \C \times \Del^{\{1\}} \hrar \cS^{m}_{m} \times \C \times (\Del^1)^{\sharp}$ (which is itself a trivial cofibration in the \textbf{Cartesian} model structure over $\Del^0$) it follows that the map $i_1^*:\Fun^{\flat}(\Cone_m(\C),\D^{\natural}) \lrar \Fun(\cS^{m}_{m}(\C),\D)$ is a trivial Kan fibation and that the composed functor 
$$ \xymatrix{
\Fun^{\flat}(\Cone_m(\C),\D^{\natural})  \ar[r]^-{i_1^*}_-{\simeq} &  \Fun(\cS^{m}_{m}(\C),\D) \ar[r]^-{(\iota'')^*} & \Fun(\cS^{m}_{m} \times \C,\D) \\
}$$
is homotopic to $i_0^*:\Fun^{\flat}(\Cone_m(\C),\D^{\natural}) \x{\simeq}{\lrar} \Fun(\cS^{m}_{m} \times \C,\D)$. We may consequently conclude that $i_1^*$ induces an equivalence between $\Fun^{\flat}_{\K_m}(\Cone_m(\C),\D^{\natural})$ and the full subcategory of $\Fun_{\K_m}(\cS^m_m(\C),\D)$ spanned by those functors whose restriction to $\cS^m_m \times \C$ which preserves $\K_m$-indexed colimits in the left variable. By Lemma~\ref{l:colim-C-2}(3) the latter condition is automatic and hence the restriction map $\iota^*_1:\Fun^{\flat}_{\K_m}(\Cone_m(\C),\D^{\natural}) \lrar \Fun^{\flat}_{\K_m}(\cS^m_m(\C),\D^{\natural})$ is an equivalence. We may then conclude that 
$$ (\iota'')^*: \Fun_{\K_m}(\cS^m_m(\C),\D) \lrar \Fun_{\K_m/\C}(\cS^{m}_m \times \C,\D)$$ 
is an equivalence of $\infty$-categories, as desired.
\end{proof}

\begin{cor}\label{c:univ-2}
Let $\D$ be an $m$-semiadditive $\infty$-category. Then restriction along the inclusion $\{\ast\} \times \C \hrar \cS^m_m(\C)$ induces an equivalence of $\infty$-categories:
$$ \Fun_{\K_m}(\cS^m_m(\C),\D) \x{\simeq}{\lrar} \Fun(\C,\D) .$$
\end{cor}
\begin{proof}
Combine Corollary~\ref{c:univ} and Theorem~\ref{t:main}.
\end{proof}

\begin{cor}\label{c:tensor}
The functor $\cS^m_m \times \cS_m(\C) \lrar \cS^m_m(\C)$ exhibits $\cS^m_m(\C)$ as the tensor product $\cS^m_m \otimes_{\K_m} \cS_m(\C)$ in $\Cat_{\K_m}$.
\end{cor}
\begin{proof}
Combine Corollary~\ref{c:univ} and Corollary~\ref{c:free-colimit}.
\end{proof}

\begin{rem}
For $m \leq n$ one may also consider the subcategory $\cS^m_n(\C) \subseteq \cS^n_n(\C)$ containing all objects and whose mapping spaces are spanned by those morphisms as in~\eqref{e:span-17} for which $p$ is $m$-truncated. A similar argument then shows that $\cS^m_n(\C)$ is the free $m$-semiadditive $\infty$-category with $\K_n$-indexed colimits generated from $\C$.
\end{rem}

\subsection{Higher semiadditivity and topological field theories}\label{s:tft}

In this section we will discuss a relation between the results of this paper and 1-dimensional topological field theories, and more specifically, with the notion of \textbf{finite path integrals} as described in~\cite[\S 3]{tft}. 
We first discuss the universal constructions of \S\ref{s:decorated} in the presence of a symmetric monoidal structure. Recall that by~\cite[Proposition 4.8.1.10]{higher-topos} the free-forgetful adjunction $\Cat_{\infty} \dashv \Cat_{\K_m}$ induces an adjunction $\CAlg(\Cat_{\infty}) \dashv \CAlg(\Cat_{\K_m})$ on commutative algebra objects which is compatible with the free-forgetful adjunction. In particular, if $\D^{\otimes} \in \CAlg(\Cat_{\infty})$ is a symmetric monoidal $\infty$-category then the $\infty$-category $\cS_m(\D)$ (which, by Corollary~\ref{c:free-colimit}, is the image of $\D$ in $\Cat_{\K_m}$ under the free functor $\Cat_{\infty} \lrar \Cat_{\K_m}$) carries a canonical symmetric monoidal structure, under which it can be identified with the image of $\D^{\otimes}$ under the left adjoint $\CAlg(\Cat_{\infty}) \lrar \CAlg(\Cat_{\K_m})$. Since the tensor product on $\cS_m(\C)$ preserves $\K_m$-colimits in each variable separately the characterization of colimits given in Lemma~\ref{l:colim-C-1} yields an explicit formula for the monoidal product as $(X,\L_X) \otimes (Y,\L_Y) = (X \times Y,\L_X \otimes \L_Y)$, where $\L_X \otimes \L_Y: X \times Y \lrar \D$ is the local system $(\L_X \otimes \L_Y)(x,y) = \L_X(x) \otimes \L_Y(y)$. We also note that by the above the unit map $\D \lrar \cS_m(\D)$ is symmetric monoidal, and if $\D$ already has $\K_m$-indexed colimits and its monoidal structure commutes with $\K_m$-indexed in each variable separately then the counit map $\cS_m(\D) \lrar \D$ is symmetric monoidal as well. 

Corollary~\ref{c:tensor} tells us that we have a similar phenomenon with $\cS^m_m(\D)$: indeed, by Proposition~\ref{p:monoidal} the $\infty$-category $\cS^m_m(\D)$ inherits a canonical commutative algebra structure in $\Add_m \simeq \Mod_{\cS^m_m}(\Cat_{\K_m})$ under which it can be identified with the image of $\cS_m(\D)^{\otimes} \in \CAlg(\Cat_{\K_m})$ under the left functor $\Cat_{\K_m} \lrar \Add_m$. Combined with the above considerations we may further identify $\cS^m_m(\C)^{\otimes} \in \CAlg(\Add_m)$ with the image of $\D$ under the left functor $\CAlg(\Cat_{\infty}) \lrar \CAlg(\Add_m)$. In explicit terms, $\cS^m_m(\D)$ carries a symmetric monoidal structure which preserves $\K_m$-indexed colimits in each variable separately and the unit map $\D \lrar \cS^m_m(\D)$ extends to a symmetric monoidal functor. Furthermore, if $\D$ is already $m$-semiadditive and its symmetric monoidal structure commutes with $\K_m$-indexed colimits in each variable separately then the counit map $\cS^m_m(\D) \lrar \D$ is symmetric monoidal as well.

The following lemma appears to be well-known, but we could not find a reference. Note that, while the lemma is phrased for $\cS^m_m(\D)$, it has nothing to do with the finiteness or truncation of the spaces in $\cS^m_m$. In particular, the analogous claim holds if one replaces $\cS^m_m(\D)$ by the analogous $\infty$-category of decorated spans between arbitrary spaces.
\begin{lem}\label{l:dual}
Let $\D$ be a symmetric monoidal $\infty$-category. Let $(X,\L_X) \in \cS^m_m(\D)$ be such that $\L_X(x)$ is dualizable in $\D$ for every $x \in X$. Then $(X,\L_X)$ is dualizable in $\cS^m_m(\D)$.
\end{lem}
\begin{proof}
Let $\D^{\dl} \subseteq \D$ be the full subcategory spanned by dualizable objects and let $(\D^{\dl})^{\simeq} \subseteq \D^{\dl}$ be the maximal subgroupoid of $\D^{\dl}$. Let $\Bord^{\otimes}_1$ be the 1-dimensional framed cobordism $\infty$-category. By the 1-dimensional cobordism hypothesis (\cite{tft-lurie},\cite{bord1}), evaluation at the positively 1-framed point $\ast_+ \in \Bord_1$ induces an equivalence 
\begin{equation}\label{e:bord}
\Fun^{\otimes}(\Bord^{\otimes}_1,\D^{\otimes}) \x{\simeq}{\lrar} (\D^{\dl})^{\simeq} .
\end{equation}
Now let $(X,\L_X)$ be an object of $\cS^m_m(\D)$ such that $\L_X(x)$ is dualizable for every $x \in X$. Then the local system $\L_X: X \lrar \D$ factors through a local system $\L_X': X \lrar (\D^{\dl})^{\simeq}$. By the equivalence~\eqref{e:bord} we may lift $\L_X$ to a local system $\L_X':X \lrar \Fun^{\otimes}(\Bord^{\otimes}_1,\D^{\otimes})$ valued in topological field theories. Evaluation at the negatively 1-framed point $\ast_- \in \Bord_1$ now yields a local system $\hat{\L}_X:X \lrar \D$. Furthermore, for every $x \in X$, evaluation at the evaluation bordism $\ev: \ast_+ \coprod \ast_- \lrar \emptyset$ induces a map $\L_X'(x,\ev):\L_X(x) \otimes \hat{\L}_X(x) \lrar 1_{\D}$ exhibiting $\hat{\L}_X(x)$ as dual to $\L_X(x)$. Allowing $x$ to vary we obtain a natural transformation $\L_X'(\ev): (\L_X \otimes \hat{\L}_X) \circ \del \Rightarrow \ovl{1_{\D}}$ of $\D$-valued local systems on $X$, where $\del: X \lrar X \times X$ is the diagonal map. Similarly, we may evaluate at the coevaluation cobordism $\coev: \emptyset \lrar \ast_- \coprod \ast_+$ and obtain a natural transformation $\L_X'(\coev): \ovl{1_{\D}} \Rightarrow (\hat{\L}_X \otimes \L_X) \circ \del$, which, for each $x$, determines a compatible coevaluation map $1 \lrar \hat{\L}_X(x) \otimes \L_X(x)$.

Now let $q: X \lrar \ast$ denote the constant map and consider the morphisms $\ev_{(X,\L_X)}: (X,\L_X) \otimes (X,\hat{\L}_X) = (X \times X,\L_X \otimes \hat{\L}_X) \lrar \ast$ and $\ast \lrar (X \times X,\hat{\L}_X \otimes \L_X)$ in $\cS^m_m(\C)$ given by the spans
\begin{equation}\label{e:span-22}
\xymatrix{
& (X,(\L_X \otimes \hat{\L}_X) \circ \del) \ar_{(\del,\Id)}[dl]\ar^{(q,\L_X'(\ev))}[dr] &\\
(X \times X,\L_X \otimes \hat{\L}_X) && (\ast,1_{\D})\\
}
\end{equation}
and
\begin{equation}\label{e:span-21}
\xymatrix{
& (X,\ovl{1_{\D}}) \ar_{(q,\Id)}[dl]\ar^{(\del,\L_X'(\coev))}[dr] &\\
(\ast,1_{\D}) && (X \times X,\hat{\L}_X \otimes \L_X)\\
}
\end{equation}
It is now straightforward to check that the morphisms~\eqref{e:span-22} and~\eqref{e:span-21} satisfy the evaluation-coevaluation identities and hence exhibit $(X,\L_X)$ and $(X,\hat{\L}_X)$ as dual to each other.
\end{proof}

Let us now explain the relation of the above construction with the notion of finite path integrals as described in~\cite{tft}. Given a local system $\L_X: X \lrar \D$ of dualizable objects in $\D$ (e.g., a family of invertible objects), one obtains, as described in Lemma~\ref{l:dual}, a dualizable object $(X,\L_X)$ of the decorated span $\infty$-category $\cS^m_m(\D)$. By the cobordism hypothesis this object determines a $1$-dimensional topological field theory $Z: \Bord_1 \lrar \cS^m_m(\D)$ which sends the point to $(X,\L_X)$. The term \textbf{quantization} is sometimes used to describe a procedure in which the topological field theory $Z$ can be ``integrated'' into a topological field theory taking values in $\D$ (see, e.g., \cite{urs}). This can often be achieved, at various levels of rigor, but performing some kind of a \textbf{path integral}. 

Such a path integral is described informally in~\cite{tft} in the setting of \textbf{finite groupoids} (i.e. $m=1$) and where the target $\infty$-category $\D$ is the category of vector spaces over the complex numbers. More generally, the authors of~\cite{tft} work with an $n$-categorical version of the span construction and consider $n$-dimensional topological field theories. However, as the paper~\cite{tft} is expository in nature, it discusses these ideas somewhat informally, leaving many assertions without a formal proof or a precise formulation. In a recent paper~\cite{trova}, F. Trova suggests to use the formalism of Nakayama categories in order to give a formal definition of quantization in the setting of finite groupoids and $1$-dimensional field theories, when the target is an ordinary category satisfying suitable conditions. We shall now explain how the results of the present paper can be used to give a formal definition of quantization when the target is an $m$-semiadditive $\infty$-category and finite groupoids are generalized to $m$-finite spaces.

Let $Z:\Bord_1 \lrar \cS^m_m(\D)$ be the topological field theory determined by a local system $\L_X: X \lrar \D$ of dualizable objects in $\D$. Suppose that $\D$ is $m$-semiadditive and that the monoidal structure on $\D$ preserves $\K_m$-indexed colimits in each variable separately. Then we may consider the \textbf{counit map} $\nu_\D: \cS^m_m(\D) \lrar \D$ associated to the free-forgetful adjunction $\CAlg(\Cat_{\infty}) \dashv \CAlg(\Cat_{\K_m})$. This counit map is a symmetric monoidal functor, and we may consequently post-compose the topological field theory $Z$ with $\nu_D$ to obtain a topological field theory $\ovl{Z}:\Bord_1 \lrar \D$.
%
The association $Z \mapsto \ovl{Z}$ can be considered as a \textbf{quantization procedure}, and by comparing values on the point it must be compatible with the approach of~\cite{tft}. We note that one may write the counit map $\nu_D:\cS^m_m(\D) \lrar \D$ explicitly by using the formation of colimits and the formal summation of $\K_m$-families of maps in $\D$ via its canonical enrichment in $m$-commutative moniods established in \S\ref{s:monoids}. The resulting formulas can then be considered as explicit forms of path integrals. We may also summarize this process with the following corollary:

\begin{cor}
Let $\D$ be an $m$-semiadditive $\infty$-category equipped with a symmetric monoidal structure which preserves $\K_m$-indexed colimits in each variable separately. Then the collection of dualizable objects in $\D$ is closed under $\K_m$-indexed colimits. Furthermore, if $X$ is $m$-finite space and $\L_X: X \lrar \D$ is a local system of dualizable objects in $\D$, then the $1$-dimensional topological field theory $\Bord_1 \lrar \D$ determined by the dualizable object $\colim\L$ is the \textbf{quantization} of the topological field theory $\Bord_1 \lrar \cS^m_m(\D)$ determined by the dualizable object $(X,\L_X) \in \cS^m_m(\D)$.
\end{cor}

\end{document}